\documentclass[11pt]{amsart}

\usepackage{amsmath, latexsym, amsfonts, amssymb, amsthm}
\usepackage{graphicx, hyperref,dsfont,epsfig,caption,wrapfig,subfig}
\usepackage[dvipsnames]{xcolor}
\usepackage{pstricks,yfonts,mathalfa}
\usepackage{mathtools}
\usepackage{movie15}
\hypersetup{
    linktoc=page,
    linkcolor=red,          
    citecolor=blue,        
    filecolor=blue,      
    urlcolor=cyan,
    colorlinks=true           
}

\numberwithin{equation}{section}

\newtheorem{theorem}{Theorem}[section]
\newtheorem{lemma}[theorem]{Lemma}
\newtheorem{proposition}[theorem]{Proposition}

\newtheorem{remark}[theorem]{Remark}

\newcounter{thmc}

\newtheorem{thmcite}[thmc]{Theorem}

\theoremstyle{definition}

\DeclareMathSymbol{\leqslant}{\mathalpha}{AMSa}{"36} 
\DeclareMathSymbol{\geqslant}{\mathalpha}{AMSa}{"3E} 
\DeclareMathSymbol{\eset}{\mathalpha}{AMSb}{"3F}     
\renewcommand{\leq}{\;\leqslant\;}                   
\renewcommand{\geq}{\;\geqslant\;}                   

\newcommand{\C}{\mathbb{C}}
\newcommand{\R}{\mathbb{R}}

\newcommand{\E}{\mathds{E}}

\newcommand{\ps}[1]{\left\langle #1 \right\rangle}

\def\eps{\varepsilon}
\def\S{\mathbb{S}}

\def\bi{\begin{itemize}}
\def\ei{\end{itemize}}

\def\bnum{\begin{enumerate}}
\def\enum{\end{enumerate}}
\def\<#1{\langle #1 \rangle}

\definecolor{ao(english)}{rgb}{0.0, 0.5, 0.0}
\definecolor{darkcandyapplered}{rgb}{0.64, 0.0, 0.0}
\definecolor{amber}{rgb}{1.0, 0.49, 0.0}

\def\V{\bm{\mathrm V}}
\def\X{\bm{\mathrm X}}

\newcommand{\norm}[1]{\left\lvert#1\right\rvert}
\newcommand{\expect}[1]{\mathbb{E}\left[#1\right]}

\usepackage{bm}
\usepackage{bbm}

\title{Ward identities in the $\mathfrak{sl}_3$ Toda conformal field theory}

\author{Baptiste Cercl\'e}
\email{baptiste.cercle@universite-paris-saclay.fr}
\address{Laboratoire de Math\'ematiques d'Orsay, B\^atiment 307. Facult\'e des Sciences d'Orsay, Universit\'e Paris-Saclay. F-91405 Orsay Cedex, France}
\author{Yichao Huang}
\email{yichao.huang@helsinki.fi}
\address{University of Helsinki, Department of Mathematics and Statistics, P.O. Box 68, FIN-00014 University of Helsinki, Finland}

\thanks{The authors are indebted to A. Kupiainen, R. Rhodes and V. Vargas for fruitful discussions on Toda theories. Y. Huang is supported by ERC grant QFPROBA, No.741487.}

\begin{document}

\maketitle
\begin{abstract}
Toda conformal field theories are natural generalizations of Liouville conformal field theory that enjoy an enhanced level of symmetry. In Toda conformal field theories this \emph{higher-spin symmetry} can be made explicit, thanks to a path integral formulation of the model based on a Lie algebra structure.
The purpose of the present document is to explain how this higher level of symmetry can manifest itself within the rigorous probabilistic framework introduced by R. Rhodes, V. Vargas and the first author in~\cite{Toda_construction}. One of its features is the existence of holomorphic currents that are introduced via a rigorous derivation of the Miura transformation. More precisely, we prove that the spin-three Ward identities, that encode higher-spin symmetry, hold in the $\mathfrak{sl}_3$ Toda conformal field theory; as an original input we provide explicit expressions for the descendent fields which were left unidentified in the physics literature. This representation of the descendent fields provides a new systematic method to find the degenerate fields of the $\mathfrak{sl}_3$ Toda (and Liouville) conformal field theory, which in turn implies that certain four-point correlation functions are solutions of a hypergeometric differential equation of the third order.
\end{abstract}
 

\section{Introduction}
\subsection{Toda theories and higher-spin symmetry}
Providing a definition to the notion of random surface in the context of two-dimensional quantum gravity has been a seminal topic since the pioneering work of Polyakov~\cite{Pol81} in 1981. In this groundbreaking article were laid the foundations of the \emph{Liouville conformal field theory} (Liouville theory in the sequel), which may be understood as a canonical way of picking at random a geometry on a Riemann surface $\Sigma$ with fixed topology~\cite{Sei90}. More precisely, if $\Sigma$ has Riemannian metric $g$ then such a geometry on $(\Sigma,g)$ may be described via a random map $\varphi:\Sigma\to\R$ that represents the conformal factor of the (conformal metric) $e^{\varphi}g$. The Liouville theory relies on a Lagrangian formulation based on the Liouville action functional
 \begin{equation}\label{action}
 S_{L}(\varphi,g)\coloneqq  \frac{1}{4\pi} \int_{\Sigma}  \Big (  \norm{\partial_g \varphi(x)}^2+QR_g \varphi(x) +4\pi \mu e^{\gamma    \varphi(x)}   \Big)\,{\rm v}_{g}(dx).
\end{equation}
The metric $g$ has scalar curvature $R_g$, gradient $\partial_g$ and volume form ${\rm v}_{g}$; the quantum parameters are the cosmological constant $\mu>0$, the coupling constant $\gamma\in(0,2)$ and the background charge $Q$ defined via $Q=\frac\gamma2+\frac2\gamma$. With a path integral approach based on this Lagrangian and given $\alpha$ a real number, one can explicitly define primary fields $V_\alpha$ ---which are usually referred to as \emph{Vertex Operators}--- and introduce the correlation functions of Vertex Operators via the (formal) expression 
\begin{equation}\label{eq:correl_LCFT}
    \ps{\prod_{k=1}^NV_{\alpha_k}(z_k)}\coloneqq \frac{1}{\mathcal{Z}_g}\int \prod_{k=1}^N e^{\alpha_k\varphi(z_k)}e^{-S_L(\varphi,g)} D\varphi
\end{equation}
for $(z_1,\cdots,z_N)$ distinct elements of $\Sigma$.
In the latter expression $\mathcal{Z}_g$ stands for a normalization constant while the differential element $D\varphi$, purely formal, should be thought of as a \lq\lq Lebesgue measure" over fields $\Sigma\rightarrow\R$.
From now on, the underlying surface $(\Sigma,g)$ will be the sphere $\mathbb{S}^2$ equipped with its standard metric or equivalently (via the stereographic projection) the Riemann sphere $\C\cup\lbrace\infty\rbrace$ (or extended complex plane) equipped with the metric (see Subsection \ref{reminders})
\begin{equation}
    \hat{g}\coloneqq\frac{4}{(1+|z|^2)^2}|dz|^2.
\end{equation}

Following the work of Polyakov, deriving the exact expression of the correlation functions above has become a key issue in the understanding of two-dimensional quantum gravity. 
This question was addressed in a subsequent work~\cite{BPZ} by Belavin, Polyakov and Zamolodchikov, where the authors described a general method for solving \emph{two-dimensional Conformal Field Theories} (2d CFTs in the sequel), based on a systematic way of exploiting the conformal symmetry assumptions through the study of the Virasoro algebra (the symmetry algebra of 2d CFTs). As a feature of this machinery, the conformal symmetry of the model yields the existence of a holomorphic current (of spin $2$), the so-called \emph{stress-energy tensor} usually denoted by $\bm{\mathrm{T}}$. This tensor admits a (formal) expansion in Laurent series which takes the form
\begin{equation}
\bm {\mathrm T}(z_0)=\sum_{n\in\mathbb{Z}}\frac{\bm {\mathrm L}_n(z)}{(z_0-z)^{n+2}}
\end{equation}
around some point $z$, and where the modes $\bm {\mathrm L}_n$ are the generators of the Virasoro algebra, with the commutation relations given by
\begin{equation}
[\bm {\mathrm L}_n,\bm {\mathrm L}_m]=(n-m)\bm {\mathrm L}_{n+m} +\frac{c}{12}(n-1)n(n+1)\delta_{n+m,0}\mathrm{Id},
\end{equation}
with $c$ the central charge of the CFT. In Liouville theory this tensor can (formally) be defined via the variation of the correlation functions with respect to the metric $g$ and thus admits an alternative expression in terms of the conformal factor $\varphi$ (see Equation~\eqref{eq:stress_tensor} below).
One of the key properties of this tensor is its \emph{Operator Product Expansion} (OPE in the sequel) with Vertex Operators, which (again formally) takes the form:
\begin{equation}
    \bm {\mathrm T}(z_0)V_\alpha(z)=\frac{\Delta(\alpha)V_\alpha(z)}{(z_0-z)^2}
    +\frac{\partial_z V_\alpha(z)}{z_0-z}+\text{holomorphic terms}
\end{equation}
as $z_0\rightarrow z$ and implies that the correlation functions of Vertex Operators solve the so-called \emph{Ward identity}:
\begin{equation}
    \ps{\bm {\mathrm T}(z_0)\prod_{k=1}^NV_{\alpha_k}(z_k)}=\sum_{l=1}^N\left(\frac{\Delta_{\alpha_l}}{(z_0-z_l)^2}
    +\frac{\partial_{z_l}}{z_0-z_l}\right)\ps{\prod_{k=1}^NV_{\alpha_k}(z_k)}.
\end{equation}
Here we have introduced $\Delta_{\alpha}\coloneqq\frac{\alpha}{2}\left(Q-\frac{\alpha}2\right)$ the conformal dimension of the Vertex Operator $V_\alpha$. The derivative should be understood as a complex (\textit{i.e.} Wirtinger) derivative\footnote{Unless explicitly stated, holomorphic derivatives will be considered throughout the rest of the document.}, that is $\partial_zf(x,y)=\frac12\left(\partial_x-i\partial_y\right)f(x,y)$; $\partial_{\bar z}$ is defined analogously via $\partial_{\bar z}f(x,y)=\frac12\left(\partial_x+i\partial_y\right)f(x,y)$. When combined with the holomorphicity at infinity of this tensor, by which is meant that $\bm{\mathrm{T}}(z)\sim\frac{1}{z^4}$ as $z\rightarrow\infty$ and which is usually axiomatic in 2d CFTs, the above identity implies that Liouville correlation functions enjoy a property of conformal covariance in the sense that for any M\"obius transform $\psi$ of the complex plane,
 \begin{equation*}
 \ps{\prod_{k=1}^NV_{\alpha_k}(\psi (z_k) )}= \prod_{l=1}^N\norm{\psi'(z_l)}^{-2\Delta_{\alpha_l}}\ps{\prod_{k=1}^NV_{\alpha_k}(z_k)}.
 \end{equation*}

Models with enhanced symmetry, by which is meant that the symmetry algebra of the model contains the Virasoro algebra, appeared shortly after in a work by Zamolodchikov~\cite{Za85} where the author introduced the notion of $W$ (or higher-spin) symmetry, based on extensions of the Virasoro algebra called \emph{$W$-algebras}. Instances of models with this additional level of symmetry emerged in subsequent works, first in~\cite{FaZa} and then in general in~\cite{FaLu}. Toda conformal field theories (Toda theories hereafter) arise as natural extensions of the Liouville theory in the prospect of constructing CFTs with an additional level of symmetry. Similarly to the Liouville action principle, Toda theories can be defined based on a Lagrangian formulation, but the action functional is slightly more complicated since it is based on a structure of Lie algebra. More precisely, for a given finite-dimensional semi-simple and complex Lie algebra $\mathfrak{g}$, the Toda action functional is defined by:
\begin{equation}\label{Toda_action}
    S_{T,\mathfrak{g}}(\varphi,g)\coloneqq  \frac{1}{4\pi} \int_{\Sigma}  \left(  \left\langle\partial_g \varphi(x), \partial_g \varphi(x) \right\rangle   +R_g \langle Q, \varphi(x) \rangle +4\pi \sum_{i=1}^{r} \mu_i e^{\gamma \langle e_i,\varphi(x) \rangle}   \right) {\rm v}_{g}(dx).
\end{equation}
In this expression, the coupling constant $\gamma$ now belongs to $(0,\sqrt 2)$ and the background charge $Q$, as well as the field $\varphi$, are no longer real-valued but rather elements of the Cartan sub-algebra of $\mathfrak{g}$, denoted $\mathfrak{h}$ and viewed as a \textbf{real} vector space. The background charge has an expression similar to the one in Liouville theory\footnote{Note that the expression for $Q$ differs from the one in Liouville theory because of our convention on $\gamma$. The standard one can be recovered by scaling $\gamma$ by a multiplicative factor $\sqrt2$. This scaling is due to the fact that the simple roots are not orthonormal but rather satisfy $\ps{e_i,e_i}=2$.}: 
\begin{equation}
Q\coloneqq\left(\gamma+\frac2\gamma\right)\rho
\end{equation}
where $\rho$ is the \textit{Weyl vector} of $\mathfrak{h}$, defined in Equation~\eqref{eq:def_Weyl} below. The additional terms that appear and that depend on the Lie algebra are a scalar product $\ps{\cdot,\cdot}$ on $\mathfrak{h}$ and the simple roots $(e_i)_{1\leq i\leq r}$ of the Lie algebra $\mathfrak{g}$ with respect to $\mathfrak{h}$. The path integral definition of Toda correlation functions is given, in the same way as Equation~\eqref{eq:correl_LCFT} for Liouville theory, by
\begin{equation}\label{eq:correl_TCFT}
    \ps{\prod_{k=1}^NV_{\alpha_k}(z_k)}\coloneqq \frac{1}{\mathcal{Z}_g}\int \prod_{k=1}^N e^{\ps{\alpha_k,\varphi(z_k)}}e^{-S_{T,\mathfrak{g}}(\varphi,g)} D\varphi
\end{equation}
for $(z_1,\cdots,z_N)$ distinct elements of $\Sigma$ and where the weights $\left(\alpha_1,\cdots,\alpha_N\right)$ are now elements of $\mathfrak{h}^*$ the dual space of $\mathfrak{h}$. More generally the path integral allows to define for suitable functionals $F$ over fields $\Sigma\to\mathfrak{h}$
\begin{equation}\label{eq:path_integral_TCFT}
    \ps{F}\coloneqq \frac{1}{\mathcal{Z}_g}\int \prod_{k=1}^N F(\varphi)e^{-S_{T,\mathfrak{g}}(\varphi,g)} D\varphi.
\end{equation}

A similar reasoning involving OPEs with respect to holomorphic currents should remain valid when the model being studied enjoys \emph{higher-spin symmetry} in addition to the conformal symmetry. This additional level of symmetry also comes with additional \emph{holomorphic currents} that contain information related to this higher-spin symmetry. Indeed in Toda theories there is a family of \emph{higher-spin currents} $\bm{\mathrm{W}}^{(i)}$ (up to $i\leq r$), that are defined via the Miura transformation\footnote{The interested reader may find details on the role of this transformation in the construction of two-dimensional CFTs having higher-spin symmetry for instance in~\cite{FaLu}, where the Miura transformation is used to construct representations of $W$-algebras.}
\begin{equation}\label{Miura}
    \prod_{i=1}^{r+1}\left(\frac {q}{2}\partial+\ps{h_{i},\partial\varphi}\right)\coloneqq \sum_{i=0}^{r}\bm {\mathrm W^{(r-i)}}\left(\frac {q}{2}\partial\right)^i
\end{equation}
where the $(h_i)_{1\leq i\leq r+1}$ are the \emph{fundamental weights in the first fundamental representation} $\pi_1$ of $\mathfrak g$, $q\coloneqq\gamma+\frac2\gamma$ and like before $\partial$ is a holomorphic derivative.

In the sequel, we focus on the study of the $\mathfrak{g}=\mathfrak{sl}_3$ Toda theory. There will be one additional holomorphic current of spin three $\bm {\mathrm W}\coloneqq\bm {\mathrm W^{(3)}}$, which admits the Laurent series expansion
\begin{equation}
\bm {\mathrm W}(z_0)=\sum_{n\in\mathbb{Z}}\frac{\bm {\mathrm W}_n(z)}{(z_0-z)^{n+3}}.
\end{equation}
The $W_3$ algebra is then a \emph{Vertex Operator algebra} generated by the $(\bm {\mathrm L}_n,\bm {\mathrm W}_m)_{n,m\in\mathbb{Z}}$ and with commutation rules given by
\begin{equation}
[\bm {\mathrm L}_m,\bm {\mathrm W}_n]=(2m-n)\bm {\mathrm W}_{m+n}.
\end{equation}
The commutation rules for the $(\bm {\mathrm W}_n)_{n\in\mathbb{Z}}$ is rather complicated and bilinear in the $\left(\bm {\mathrm L}_n\right)_{n\in\mathbb{Z}}$ (see \cite[Equation (2.1)]{BCP}). In particular the $W_3$ algebra is not a Lie algebra. In the present document we will not use this modes representation but rather define the tensor according to the Miura transformation~\eqref{Miura} 
\begin{equation}\label{eq:W_tensor}
       \begin{split}
           \bm {\mathrm W}\coloneqq& -\frac{q^2}{8}\ps{h_2,\partial^3\varphi}+\frac{q}{4}\left(\ps{h_2-h_1,\partial^2\varphi}\ps{h_1,\partial\varphi}+\ps{h_3-h_2,\partial^2\varphi}\ps{h_3,\partial\varphi}\right)\\
    &+\ps{h_1,\partial\varphi}\ps{h_2,\partial\varphi}\ps{h_3,\partial\varphi}.
       \end{split} 
\end{equation}
Similarly the stress-energy tensor will be defined via the expression
\begin{equation}\label{eq:stress_tensor}
\bm {\mathrm T}\coloneqq\ps{Q,\partial^2\varphi}-\ps{\partial\varphi,\partial\varphi}.
\end{equation}
In the $\mathfrak{sl}_3$ Toda theory Vertex Operators $V_\alpha(z)$ still depend on positions $z\in\C$ but the weights $\alpha$ now belong\footnote{The weights rather belong to the dual space $\mathfrak{h}_3^*$ of $\mathfrak{h}_3$, but these two spaces may be identified via Riesz representation theorem.} to $\mathfrak{h}_3$ the Cartan subalgebra of $\mathfrak{sl}_3$. The so-called \emph{WV Operator Product Expansion}, axiomatic in the physics literature, is key in the understanding of the higher-spin symmetry since it provides an explicit expression where all the OPE coefficients are completely determined. More precisely, in the $\mathfrak{sl}_3$ Toda theory, the $WV$ OPE takes the form:
\begin{equation}\label{eq:WVOPEinformal}
    \bm {\mathrm W}(z_0)V_\alpha(z)=\frac{w(\alpha)V_\alpha(z)}{(z_0-z)^3}
    +\frac{\bm {\mathrm W_{-1}}V_\alpha(z)}{(z_0-z)^2}+\frac{\bm {\mathrm W_{-2}}V_\alpha(z)}{z_0-z}+\text{holomorphic terms}
\end{equation}
where $w(\alpha)\in\C$ is the \emph{quantum number} associated to $\bm {\mathrm W}$, and the $\bm{\mathrm{W}_{-i}}V_\alpha(z)$ are the \emph{descendent fields}. These fields are said to be local, in the sense that they should only depend on the weight $\alpha$ and derivatives of the Toda field at the point $z$.

To the best of our knowledge, explicit expressions for $\bm{\mathrm{W}_{-i}}V_\alpha(z)$ remain unknown in the physics literature and should look like ``derivatives in an extra direction'' of the Vertex operator (cf.~\cite[Subsection 8.2]{Watts}). As stressed in~\cite[Chapter 18.2]{FrBz}, \lq\lq It would be interesting to identify the deformation problems related to ... $W$-algebras". Nevertheless, these can be considered as the building blocks for solving the $\mathfrak{sl}_3$ Toda theories since the $T$ and $W$ descendent states are assumed to span the space of states of Toda theories viewed as meromorphic CFTs (see \emph{e.g.}~\cite[Subsection 3.1]{BouSch}). 

Equation~\eqref{eq:WVOPEinformal} formally holds in the sense of operators, but can be rephrased in terms of correlation functions as
\begin{equation}
    \ps{\bm {\mathrm W}(z_0)\prod_{k=1}^NV_{\alpha_k}(z_k)}=\sum_{l=1}^N\left(\frac{w(\alpha_l)}{(z_0-z_l)^3}
    +\frac{\bm {\mathcal W_{-1}^{(l)}}}{(z_0-z_l)^2}+\frac{\bm {\mathcal W_{-2}^{(l)}}}{z_0-z_l}\right)\ps{\prod_{k=1}^N V_{\alpha_k}(z_k)},
\end{equation}
where $\bm {\mathcal W_{-i}^{(l)}}\ps{\prod_{k=1}^NV_{\alpha_k}(z_k)}$ stands for $\ps{\bm {\mathrm W_{-i}}V_{\alpha_l}(z_l)\prod_{k\neq l}V_{\alpha_k}(z_k)}$.
This equality is usually referred to as the \emph{spin-three Ward identity} and is an instance of the higher-spin symmetry enjoyed by the model. This identity is the starting point for a machinery similar to the one developed by Belavin, Polyakov and Zamolodchikov~\cite{BPZ} in the Liouville case, and that exploits extensively the symmetries of the model to compute correlation functions. More precisely when certain Vertex Operators $V_{\alpha_1}$, dubbed \emph{degenerate fields}, are inserted within a correlation function some quantities such as $\bm {\mathcal W_{-1}^{(1)}}\ps{\prod_{k=1}^NV_{\alpha_k}(z_k)}$ may be expressed in terms of derivatives of the correlation function. Combined with the $WV$ OPE~\eqref{eq:WVOPEinformal} this implies that certain correlation functions are solutions of a differential equation of the third order. This method has been investigated by Fateev and Litvinov in~\cite{FaLi0} and subsequent works and allowed them to predict the value of certain three-point correlation functions in the general $\mathfrak{sl}_n$ Toda theories. 

\subsection{A probabilistic framework for Toda theories}
The topic of two-dimensional conformal field theory has been recently thoroughly investigated by the mathematics community, with several programs aiming to interpret this notion thanks to methods ranging from probability to representation theory (\emph{e.g.} Vertex Operator Algebras, see ~\cite{Borcherds, FLM, FrBz}). Following the groundbreaking work by Schramm~\cite{schramm} in which were introduced the so-called \emph{Schramm-Loewner Evolutions}, several viewpoints within the probabilistic community have emerged in an attempt of describing two-dimensional CFTs through the lens of random geometry. These approaches often rely on the study of the scaling limit of certain random planar maps, either via the proof of existence of a certain limiting object (a \emph{Brownian surface} in~\cite{LeG13,Mie13}) or by providing an explicit construction in the continuum of \emph{Liouville Quantum Gravity surfaces} (see \emph{e.g.}~\cite{DS10,DMS14,DDDF,DFGPS,GM20}). Closer to the language developed by Belavin, Polyakov and Zamolodchikov~\cite{BPZ} is the program initiated by David, Kupiainen, Rhodes and Vargas in~\cite{DKRV} aiming to provide a mathematically rigorous meaning to the machinery introduced in~\cite{BPZ} and mostly developed in the physics literature~\cite{Teschner}. Among their achievements are a proof~\cite{KRV_DOZZ} of the \emph{DOZZ formula} (for the Liouville three-point constant) as well as a rigorous justification~\cite{GKRV} of the recursive procedure known as \emph{conformal bootstrap}.

Following these recent developments and based on the formalism of~\cite{DKRV}, a rigorous probabilistic construction for the correlation functions of simply-laced Toda theories is proposed in~\cite{Toda_construction}, in which the authors make sense of the path integral approach to Toda theories~\eqref{eq:path_integral_TCFT} in the realm of probability theory. To do so were introduced the Toda correlation functions of Vertex Operators (and more generally the law of the quantum Toda field) in a fully rigourous mathematical framework based on two key probabilistic objects: Gaussian Free Fields (GFFs in the sequel) and Gaussian Multiplicative Chaos (GMC in the sequel). We present this probabilistic framework in Subsection~\ref{subsec:proba}. This construction relies on an appropriate regularization of the Vertex Operators $V_{\alpha,\eps}(z)$ and the investigation of the existence of the quantity
\[
\lim\limits_{\eps\rightarrow0}\ps{\prod_{k=1}^NV_{\alpha_k,\eps}(z_k)},
\]
based on a probabilistic interpretation of Equation~\eqref{eq:correl_TCFT}.
The statement of Theorem~\ref{propcorrel1} (adapted from~\cite[Theorem 3.1]{Toda_construction}) ensures that this construction of the correlation functions makes sense provided that some assumptions on the Vertex Operators, dubbed \emph{Seiberg bounds}~\cite{Sei90}, hold (see Equation~\eqref{seiberg_bounds} from Theorem~\ref{propcorrel1}).

Under this framework, the quantities that appear in the above Ward identities are defined via a similar regularization procedure (involving an additional regularization of the correlation functions via a parameter $\delta$) by setting 
\begin{equation}
\begin{split}
\ps{\bm {\mathrm T}(z_0)\prod_{k=1}^NV_{\alpha_k}(z_k)}&\coloneqq\lim\limits_{\delta,\eps\rightarrow0}\ps{\bm {\mathrm T}_\eps(z_0)\prod_{k=1}^NV_{\alpha_k,\eps}(z_k)}_\delta,\\
    \ps{\bm {\mathrm W}(z_0)\prod_{k=1}^NV_{\alpha_k}(z_k)}&\coloneqq\lim\limits_{\delta,\eps\rightarrow0}\ps{\bm {\mathrm W}_\eps(z_0)\prod_{k=1}^NV_{\alpha_k,\eps}(z_k)}_\delta.
\end{split}
\end{equation}
Similarly the descendent states will be defined via
\begin{equation}
        \bm {\mathcal W_{-i}^{(l)}}\ps{\prod_{k=1}^NV_{\alpha_k}(z_k)}\coloneqq\lim\limits_{\delta,\eps\rightarrow0}\ps{\bm {\mathrm W_{-i,\eps}}V_{\alpha_l,\eps}(z_l)\prod_{k\neq l}V_{\alpha_k,\eps}(z_k)}_\delta,\quad i=1,2.
\end{equation}
The precise definitions of these (regularized) expressions will be made explicit in Subsection~\ref{sub_currents}.

We will rely on this construction and focus on the case where the underlying Lie algebra corresponds to the special linear algebra $\mathfrak{sl}_3$, which is the simplest instance of Toda theory for which higher-spin symmetry manifests itself. Our main goal in the present document is to build a bridge between the ideas developed in the physics literature and the ones that come from the probabilistic viewpoint by showing that these two frameworks are consistent. To do so, we recover within the probabilistic formalism some of the axioms used in the standard construction of the $\mathfrak{sl}_3$ Toda theory on the Riemann sphere, by confirming that the OPEs between currents (defined via the Miura transformation~\eqref{Miura}) and Vertex Operators are indeed true in the sense that the Ward identities \eqref{Ward_stress} and \eqref{Ward_higher} do hold.
\begin{theorem}\label{main_theorem}
Let $\gamma\in(0,\sqrt 2)$, $(z_k)_{1\leq k\leq N}$ be distinct complex numbers and assume that the weights $(\alpha_k)_{1\leq k\leq N}\in\left(\mathfrak{h}_3^*\right)^N$ are such that the Seiberg bounds~\eqref{seiberg_bounds} hold. Let $z_0$ be in $\C\setminus\lbrace z_1,\cdots,z_N\rbrace$.
Then the \emph{spin-two} Ward identity holds:
\begin{equation}\label{Ward_stress}
    \ps{\bm {\mathrm T}(z_0)\prod_{k=1}^{N}V_{\alpha_k}(z_k)}=\sum_{l=1}^{N}\left(\frac{\Delta_{\alpha_l}}{(z_0-z_l)^2}
    +\frac{\partial_{z_l}}{z_0-z_l}\right)\ps{\prod_{k=1}^{N}V_{\alpha_k}(z_k)}
\end{equation}
with the conformal weight $\Delta_{\alpha}\coloneqq\ps{\frac{\alpha}{2},Q-\frac{\alpha}{2}}$. Similarly the \emph{spin-three} Ward identity holds in the form
\begin{equation}\label{Ward_higher}
    \ps{\bm {\mathrm W}(z_0)\prod_{k=1}^NV_{\alpha_k}(z_k)}=-\frac18\sum_{l=1}^N\left(\frac{w(\alpha_l)}{(z_0-z_l)^3}
    +\frac{\bm {\mathcal W_{-1}^{(l)}}}{(z_0-z_l)^2}+\frac{\bm {\mathcal W_{-2}^{(l)}}}{z_0-z_l}\right)\ps{\prod_{k=1}^{N}V_{\alpha_k}(z_k)}
\end{equation}
where the quantum number associated to the current $\bm {\mathrm W}$ is given by
\begin{equation}
    w(\alpha)\coloneqq \ps{\alpha-Q,h_1}\ps{\alpha-Q,h_2}\ps{\alpha-Q,h_3}
\end{equation}
and the $\bm{\mathcal W_{-i}^{(l)}}\ps{\prod_{k=1}^{N}V_{\alpha_k}(z_k)}$, with $i=1,2$, correspond to the well-defined limit
\begin{equation}
    \bm {\mathcal W_{-i}^{(l)}}\ps{\prod_{k=1}^NV_{\alpha_k}(z_k)}\coloneqq\lim\limits_{\delta,\eps\rightarrow0}\ps{\bm{\mathrm W_{-i,\eps}}(z_l,\alpha_l)\prod_{k=1}^NV_{\alpha_k,\eps}(z_k)}_\delta,
\end{equation}
where we have set:
\begin{equation}\label{eq:descendent}
\begin{split}
    \bm {\mathrm W_{-1,\eps}}(z,\alpha)&\coloneqq -qB(\alpha,\partial\varphi_\eps(z))-2C(\alpha,\alpha,\partial\varphi_\eps(z)),\\
    \bm {\mathrm W_{-2,\eps}}(z,\alpha)&\coloneqq q\left(B(\partial^2\varphi_\eps( z),\alpha)-B(\alpha,\partial^2\varphi_\eps(z))\right)\\
    &\quad -2C(\alpha,\alpha,\partial^2\varphi_\eps(z)) +4C(\alpha,\partial\varphi_\eps(z),\partial\varphi_\eps(z)).
\end{split}
\end{equation}
In the above expressions, $\varphi_\eps$ denotes the regularized field (see Equation~\eqref{eqdef:field}), while the $(h_i)_{1\leq i\leq 3}$ and $B$, $C$ are defined in Equations~\eqref{eq:definition_hi}, ~\eqref{def:B} and~\eqref{def:C}. The very last term in Equation~\eqref{eq:descendent}, quadratic in $\partial\varphi_\eps$, should be understood as a Wick product.
\end{theorem}

This identity~\eqref{Ward_higher} translates in a probabilistic language the $WV$ OPE and associated Ward identity present in the physics literature\footnote{There is a difference here between the form of equation~\eqref{Ward_higher} and what is commonly written down in the physics literature. Indeed it is standard (see \emph{e.g.} the discussion below Equation (2.8) in~\cite{Za85}) to scale by a multiplicative factor $i\sqrt{\frac{48}{22+5c}}$ the expression of $\bm{\mathrm W}$, with $c$ the central charge of the theory -- and by doing so the expression of the quantum numbers $w(\alpha)$ and the descendent operators -- in order for the $WW$ OPE (the contraction of two $W$ currents) to be written down in an elegant fashion.}, see \emph{e.g.} Equations (2.3) and (2.4) in~\cite{FaLi1}. We stress that, as an original input, our probabilistic framework provides an explicit expression for the $W$-descendent fields, which take the form
\begin{equation}
\bm {\mathrm W_{-i}}V_{\alpha}(z)=\bm {\mathrm W_{-i}}(z,\alpha)\times V_{\alpha}(z).
\end{equation}
We believe these explicit expressions for the descendent fields will have several applications in the understanding of Toda theories. Some basic consequences are highlighted in Section~\ref{section:BPZ}. For instance, this representation allows to find so-called \emph{degenerate fields} (see Propositions~\ref{degenerate_1} and~\ref{degenerate_2}), which are the fields that meet the requirements for the BPZ-type equation~\eqref{eq:BPZ_correl} below to hold.

Our second result can be understood as an additional manifestation of the $W$-symmetry of the model by showing the \emph{global Ward identities} (usually axiomatic in the physics literature and equivalent to the \lq\lq holomorphicity at infinity" of the spin-three current):
\begin{theorem}
In the setting of Theorem~\ref{main_theorem}, the global Ward identities hold for $0\leq n\leq 4$:
\begin{equation}\label{global_Ward}
    \sum_{l=1}^N \left(z_l^n\bm {\mathcal W_{-2}^{(l)}}+nz_l^{n-1}\bm {\mathcal W_{-1}^{(l)}}+\frac{n(n-1)}{2}z_l^{n-2}w(\alpha_l)\right)\ps{\prod_{k=1}^{N}V_{\alpha_k}(z_k)}=0.
\end{equation}
\end{theorem}
A similar statement holds for the stress-energy tensor $\bm{\mathrm T}$ and amounts to conformal covariance of the Toda correlation functions in the sense of~\cite[Theorem 3.1]{Toda_construction}. Its form is similar to the one in Liouville theory~\cite[Equation (A.6)]{BPZ}. We stress that despite our probabilistic representation, it is far from obvious to prove identites~\eqref{global_Ward} via direct computations, nor by using conserved quantities. The approach we follow rather relies on the fact that $W$ behaves like a covariant tensor of the third order (we explain this in Subsection~\ref{sub:global_Ward}).

Our last result is concerned with the integrability of Toda theories. It claims that certain four-point correlation functions are solutions of a differential equation very similar to the one encountered in the case of Liouville (see \cite[Equation (4.26)]{KRV_loc}). Such a differential equation is the starting point for computing the exact value of the three-point correlation function in Liouville theory~\cite{KRV_DOZZ}, it is therefore very tempting to think that such a program would work for the $\mathfrak{sl}_3$ Toda theory.
\begin{theorem}\label{thm:BPZ}
Let $\chi\in\lbrace\gamma,\frac{2}{\gamma}\rbrace$ and take $\kappa\in\R$, $\alpha_0$ and $\alpha_\infty$ such that $\bm\alpha=(-\chi h_1,\alpha_0,-\kappa h_3,\alpha_\infty)$ satisfies the Seiberg bounds~\eqref{seiberg_bounds}. Then the four-point correlation function \\
${\ps{V_{-\chi h_1}(z)V_{\alpha_0}(0)V_{-\kappa h_3}(1)V_{\alpha_\infty}(\infty)}}$ is real analytic on $\C\setminus\{0,1\}$ and can be written as 
\begin{equation}\label{eq:BPZ_correl}
    \ps{V_{-\chi h_1}(z)V_{\alpha_0}(0)V_{-\kappa h_3}(1)V_{\alpha_\infty}(\infty)}=\norm{z}^{\chi\ps{h_1,\alpha_0}}\norm{z-1}^{\frac{\chi\kappa}3}\mathcal H(z),
\end{equation}
with $\mathcal H$ solution of the hypergeometric differential equation of the third order
\begin{equation}\label{eq:BPZ_hyper}
    \begin{split}
    &\Big[z\left(A_1+z\partial_z\right)\left(A_2+z\partial_z\right)\left(A_3+z\partial_z\right)-\left(B_1-1+z\partial_z\right)\left(B_2-1+z\partial_z\right)z\partial_z\Big]\mathcal H=0.
    \end{split}
\end{equation}
Here we have introduced  
\begin{equation}
\begin{split}
    A_i&:=\frac{\chi}2\ps{\alpha_0-\kappa h_3-\chi h_1-Q,h_1}+\frac{\chi}2\ps{\alpha_\infty-Q,h_i}\\
    B_i&:=1+\frac{\chi}2\ps{\alpha_0-Q,h_1-h_{i+1}}.
\end{split}
\end{equation}
\end{theorem}
This differential equation has been studied in the physics literature by Fateev-Litvinov in~\cite{FaLi0}; more generally the authors predicted that for the general $\mathfrak{sl}_n$ Toda theory a differential equation of order $n$ would hold for certain four-point correlation functions. This allowed them to derive the value of some Toda three-point correlation functions~\cite[Equation (21)]{FaLi0}. In a future work the first author plans to provide a rigorous derivation of such expressions for these $\mathfrak{sl}_3$ Toda three-point correlation functions. However there are actually additional issues to be taken care of compared to the Liouville case, \emph{e.g.} existence of several reflection coefficients, which prevent a straightforward extension of the method developed in~\cite{KRV_DOZZ}. A first step towards recovering formulas from~\cite{FaLi0} would be to conduct a thorough investigation of these reflection coefficients. This would involve new probabilistic objects such as a generalization of the celebrated Williams path decomposition~\cite{Williams} in the setting of Cartan subalgebras. Besides it is not clear at this stage to what extent our probabilistic framework would allow to find formulas for correlation functions which are not predicted in the physics literature. For the time being we shed light on the fact that our method for proving Theorem~\ref{thm:BPZ} is essentially new and based on a certain representation of the descendent fields in the $\mathfrak{sl}_3$ Toda theory. It differs from the one usually used to derive BPZ-type equations (such as in~\cite{KRV_loc,remy1, Zhu_BPZ}) in that computations are simplified and only involve elementary algebraic manipulations (see Section~\ref{section:BPZ}). Moreover this representation provides a systematic method of finding degenerate fields, which are key in the understanding of the $\mathfrak{sl}_3$ Toda theory. It is worth noting that it is straightforward to extend our method to the case of Liouville theory.

The plan of the present manuscript is as follows: in Section~\ref{section_construction}, we review the basic tools to define the Toda correlation functions; a first study of analytic properties of these correlation functions as well as an interpretation of the higher-spin current in probabilistic terms will be developed in Section~\ref{section_interpretation}. We then proceed to the proof of the local and global Ward identities~\eqref{Ward_stress}, ~\eqref{Ward_higher} and ~\eqref{global_Ward} in Section~\ref{section_proof}. Section~\ref{section:BPZ} is dedicated to highlighting some applications of the expression of the descendent fields, namely by the study of the degenerate fields of the $\mathfrak{sl}_3$ Toda theory which will pave the way to proving Theorem~\ref{thm:BPZ}. Some auxiliary computations are gathered in Section~\ref{section_appendix}.

\section{Background on the probabilistic construction of the $\mathfrak{sl}_3$ Toda theory}\label{section_construction}
We recall here the basic materials from~\cite{Toda_construction}; more precisely we present in this section the probabilistic construction of the $\mathfrak{sl}_3$ Toda correlation functions along with the basic currents of the $W_3$-algebra. We first provide some general reminders on conformal geometry and Lie algebras. We then introduce the probabilistic framework needed in the document.

\subsection{Some reminders on conformal geometry and Lie algebras}\label{reminders}

\subsubsection{Conformal geometry on the Riemann sphere}

It is usually more convenient (for the probabilistic formulation) to map conformally the sphere $\S^2$ to the Riemann sphere (or the extended complex plane) $\C\cup\lbrace\infty\rbrace$ or $\R^2\cup\{\infty\}$. To do so, we use the standard stereographic projection, so that the sphere $\S^2$ with its standard metric is identified (using conformal invariance of the model) with ${\C}\cup\lbrace\infty\rbrace$ endowed with the spherical metric
\begin{equation}
    \hat{g}=\frac{4}{(1+|z|^2)^2}|dz|^2.
\end{equation}
This Riemannian metric $\hat g$ has constant scalar curvature equal to $R_{\hat g}=2$ and gives total volume to the plane ${\rm v}_{\hat g}(\C)=4\pi$. Here and after ${\rm v}_{\hat g}$ denotes the volume form in the $\hat g$ metric.
 
To this metric on the Riemann sphere we can associate a Green kernel
\begin{equation}\label{Green_round}
    G_{\hat g}(x,y)\coloneqq\ln\frac{1}{\norm{x-y}}-\frac14\left(\ln\hat g(x)+\ln\hat g(y)\right)+\ln2-\frac12.
\end{equation}
It is such that the $ \mathrm H^1(\R^2,\hat g)$ solution~\cite[Equation (2.3)]{Toda_construction} of the problem with $f\in L^2(\R^2,{\rm v}_{\hat g})$
\begin{equation}\nonumber
\triangle_{\hat g} u=-2\pi \left(f-m_{\hat g}(f)\right) \quad\text{on }\R^2,\quad  \int_{\R^2} u(x) \,{\rm v}_{\hat g}(dx) =0,
\end{equation}
can be written as
\begin{equation}
u=\int_{\R^2} G_{\hat g}(\cdot,x)f(x){\rm v}_{\hat g}(dx).
\end{equation}
The notations $\triangle_{\hat g}$ and $m_{\hat g}$ respectively stand for the Laplacian and the mean-value in the metric $\hat g$ with
\begin{equation*}
	m_{\hat g}(f)\coloneqq\frac{1}{4\pi}\int_{\R^2} f(x){\rm v}_{\hat g}(dx).
\end{equation*}

\subsubsection{Lie algebras and the Toda field action}    

Toda theories rely on a Lie algebra structure; basic materials on the notion of Lie algebra can be found for instance in the introductory textbook~\cite{Hum72}. We consider in this manuscript the case of the $\mathfrak{sl}_3$ Lie algebra, which we identify with $3\times3$ real matrices with zero trace. Its Cartan sub-algebra then corresponds to the subset of $\mathfrak{sl}_3$ made of diagonal matrices, that is $\mathfrak{h}_3\coloneqq \{x=(x_1,x_2,x_3)\in\R^3~|~\sum_{i=1}^3x_i=0\}$. The simple roots are given by $e_i\coloneqq v_i-v_{i+1}$, $i=1,2$, where $(v_i)_{1\leq i\leq 3}$ is the canonical basis of $\R^3$ and are viewed as elements of the dual space $\mathfrak{h}_3^*$ of $\mathfrak{h}_3$.
We denote by $\ps{\cdot,\cdot}$ the scalar product\footnote{This scalar product differs from the Killing form by a multiplicative factor whose value is not relevant in the present document.} on $\mathfrak{h}_3^*$ defined via the expression
\begin{equation*}
    \left(\ps{e_i,e_j}\right)_{i,j}\coloneqq A=\begin{pmatrix}
    2 & -1\\
    -1 & 2
    \end{pmatrix}.
\end{equation*}
Here $A$ is the Cartan matrix of $\mathfrak{sl}_3$. The \emph{fundamental weights} $(\omega_i)_{1\leq i \leq 2}$, form the basis of $\mathfrak{h}_3^*$ dual to that of the simple roots ($ \delta_{ij}$ is the Kronecker symbol):
 \begin{equation}
 \omega_i\coloneqq  \sum_{l=1}^{2} (A^{-1})_{i,l}e_l \quad\text{so that}\quad\ps{\omega_i,e_j}=\delta_{ij},\quad i,j\in\lbrace1,2\rbrace.
 \end{equation} 
Explicitly we have $\omega_1=\frac{2e_1+e_2}3$ and $\omega_2=\frac{e_1+2e_2}{3}$.

There are special elements of the Cartan sub-algebra $\mathfrak{h}_3^*$ that play a special role. These are the \emph{Weyl vector},
 \begin{equation}\label{eq:def_Weyl}
     \rho\coloneqq \omega_1+\omega_2=e_1+e_2,
 \end{equation}
 for which $\ps{\rho,e_i}=1$ for $i=1,2$, and the fundamental weights in the first fundamental representation $\pi_1$ of $\mathfrak{sl}_3$ with the highest weight $\omega_1$. These fundamental weights are given by
\begin{equation}\label{eq:definition_hi}
    h_1\coloneqq \frac{2e_1+e_2}{3},\quad h_2\coloneqq \frac{-e_1+e_2}{3}, \quad h_3\coloneqq  -\frac{e_1+2e_2}{3}.
\end{equation}
They prove to be useful in particular thanks to the following decomposition relation, which holds for any two vectors $u,v$ in $\mathfrak{h}_3^*$:
\begin{equation}
    \ps{u,v}=\sum_{i=1}^3\ps{u,h_i}\ps{h_i,v}.
\end{equation}
This follows from $\ps{h_i,h_j}=-\frac13+\delta_{ij}$. Let us also introduce the shorthands, for $u,v,w$ in $\mathfrak{h}_3^*$
\begin{equation}\label{def:B}
B(u,v)\coloneqq  \ps{h_2-h_1,u}\ps{h_1,v}+\ps{h_3-h_2,u}\ps{h_3,v} \quad\text{and}
\end{equation}
\begin{equation}\label{def:C}
{C(u,v,w)\coloneqq \ps{h_1,u}\ps{h_2,v}\ps{h_3,w}+\ps{h_1,v}\ps{h_2,w}\ps{h_3,u}+\ps{h_1,w}\ps{h_2,u}\ps{h_3,v}}.
\end{equation}

We now turn to the Toda field action. For $\gamma\neq0$ a real number we introduce 
\begin{equation}
    q\coloneqq \gamma+\frac{2}{\gamma}.
\end{equation}
The $\mathfrak{sl}_3$ Toda field is a random distribution with values in $\mathfrak{h}_3$, and whose law is described by a path integral based on the Toda action functional. For the spherical metric $\hat g$ and a map $\varphi:\C\cup\lbrace\infty\rbrace\rightarrow\mathfrak{h}_3$, the $\mathfrak{sl}_3$-Toda field action is equal to\footnote{Thanks to Riesz representation theorem we will often identify $\mathfrak{h}_3$ with its dual space $\mathfrak{h}_3^*$.}
\begin{equation}
 S_T(\varphi,\hat g)= \frac{1}{4\pi} \int_{\C}  \left(  \left\langle\partial_{\hat g} \varphi(x), \partial_{\hat g} \varphi(x) \right\rangle  + 2 \langle Q, \varphi(x) \rangle +4\pi \sum_{i=1}^{2} \mu_i e^{\gamma    \langle e_i,\varphi(x) \rangle}   \right){\rm v}_{\hat g}(dx)
\end{equation}
where, as before, $Q$ is the background charge, the positive quantities $\mu_1$ and $\mu_2$ form the cosmological constant and $\gamma >0$ is the coupling constant, with
\begin{equation}\label{eq:definition_Q}
 Q\coloneqq  q \rho.
\end{equation}
In the present manuscript we will not consider variations of the action with respect to the metric (this corresponds to the so-called \emph{Weyl anomaly}~\cite[Theorem 3.1]{Toda_construction}) so we have chosen to fix $g=\hat g$.

\begin{remark}
For the probabilistic construction to make sense and in agreement with \cite[Theorem 3.1]{Toda_construction}, \textbf{we will assume in the sequel that $\gamma\in(0,\sqrt{2})$}.
\end{remark}
 
\subsection{The probabilistic framework}\label{subsec:proba}
 We present here the main probabilistic objects for the path integral definition of the $\mathfrak{sl}_3$ Toda theory.
 \subsubsection{Gaussian Free Fields}
The geometric form of the Toda field action naturally leads to the introduction of GFFs with values in $\mathfrak{h}_3^*$. More precisely, set
\begin{equation}
\X\coloneqq X_1\omega_1 +X_2\omega_2
\end{equation}
where $X_1,X_2$ are two real-valued GFFs (see~\cite{dubedat, She07} for instance) with covariance structure given by
\begin{equation}
 \E[X_i(x) X_j(y)]=A_{i,j}   G_{\hat g}(x,y),\quad x\neq y\in\C\text{ and }i,j\in\{1,2\}.
 \end{equation}
The field $\X$ has almost surely zero-mean in the $\hat g$-metric. We will add the zero-mode $\bm{c}\in\mathfrak{h}_3^*\simeq \R^2$, which corresponds to the mean-value of the Toda field, in the next paragraph.

\subsubsection{Gaussian Multiplicative Chaos}
It is a classical fact that the field $\X$ is highly non-regular and only exists in the distributional sense. We will see that the Toda field inherits the same regularity from $X$; as a consequence the exponential potential $e^{\gamma\ps{e_i,\varphi(x)}}$ that appears in the Toda field action is a priori not well-defined. In order to make sense of this term we introduce a limiting procedure that allows to define the random measure $e^{\gamma\ps{e_i,\varphi(x)}}v_{\hat{g}}(dx)$, known as a \emph{Gaussian Multiplicative Chaos~\cite{Kah,review}}.

To start with, we define a regularization $\X_\eps$ of the GFF by setting
\begin{equation}\label{regularization}
    \X_\eps\coloneqq \X*\eta_\eps=\int_{\C}\X(\cdot-y)\eta_\eps(y)dy.
\end{equation}
where $\eta_\eps\coloneqq \frac1{\eps^2}\eta(\frac{\cdot}{\eps})$ is a smooth, compactly supported mollifier. Without loss of generality we assume that $\eta=0$ outside of some ball of radius $R$ centered at the origin. The field $\X_\eps$ obtained via this regularization procedure is now smooth, so the exponential of $\X_\eps$ is therefore well-defined.
Under the bound $\gamma<\sqrt 2$, the theory of GMC allows us to claim that, for each $e_i$ with $i=1,2$, within the space of Radon measures equipped with the weak topology, the regularized exponential term converges in probability to some non-trivial random measure~\cite{Ber,Sha}:
\begin{equation}
    M^{\gamma e_i}(dx)\coloneqq \underset{\eps \to 0}{\lim} M_\eps^{\gamma e_i}(dx),
\end{equation} 
where we have set
\begin{equation}
    M_\eps^{\gamma e_i}(dx):=\:  e^{  \langle\gamma e_i,  \X_\eps(x) \rangle-\frac12\expect{\langle\gamma e_i, \X_\eps(x) \rangle^2}} {\rm v}_{\hat g}(dx).
\end{equation}    
This limiting random measure is called a \emph{Gaussian Multiplicative Chaos measure}. Following the ideas presented in \cite{Toda_construction}, this formalism allows to translate the Toda path integral with respect to the Toda action \eqref{Toda_action} as a probabilistic object. This is done by setting for $F$ positive and continuous on $\mathrm{H}^{-1}(\C\to\mathfrak{h}_3^*,\hat g)$ (see~\cite[Equation (2.32)]{Toda_construction}):
 \begin{equation*}
\int F(\varphi)e^{-S_T(\varphi,\hat g) }D\varphi\coloneqq \lim\limits_{\eps\to 0}\int_{\mathfrak{h}_3^*}e^{-2\ps{Q,\bm c}}\E\left[F\left( \X_\eps+\frac {Q}{2}\ln \hat g_\eps+\bm c\right)e^{- \sum_{i=1}^{2} \mu_i e^{\gamma  \langle e_i, \bm c\rangle }M^{ \gamma e_i}_\eps(\C) }\right]d\bm c
\end{equation*}
with $\hat g_\eps(z)\coloneqq \hat g*\eta_\eps(z)$, and where the integral over $\mathfrak{h}_3^*$ is defined by
\begin{equation}
\int_{\mathfrak{h}_3^*}F(\bm c)d\bm c\coloneqq\int_{\R^2}F(c_1\omega_1+c_2\omega_2)dc_1dc_2.
\end{equation}

In the sequel, we denote with brackets the above quantity (not to be confused with the brackets for the scalar product on $\mathfrak{h}_3^*$):
 \begin{equation}\label{def:TPI}
\ps{F}\coloneqq \lim\limits_{\eps\rightarrow0}\int_{\mathfrak{h}_3^*}e^{-2\ps{Q,\bm c}}\E\left[F\left( \varphi_\eps\right)e^{- \sum_{i=1}^{2} \mu_i e^{\gamma  \langle e_i, \bm c\rangle }M^{ \gamma e_i}_\eps(\C) }\right]d\bm c, 
\end{equation}
where we introduced the shorthand
\begin{equation}\label{eqdef:field}
    \varphi_\eps\coloneqq \left(\X+\frac Q2\ln \hat g+\bm c\right)*\eta_\eps.
\end{equation}
The field $\varphi_\eps$ will be referred to as the regularized Toda field.

\subsection{Holomorphic currents and Toda correlation functions}\label{sub_currents}
We are now ready to provide a rigorous definition of the correlation functions of \emph{Vertex Operators}, to which  we may insert a holomorphic current. 
\subsubsection{Toda correlation functions}
In the language of representation theory, Vertex Operators are primary states of a highest-weight representation of the $W$-algebra. However this definition is not really explicit and at some point not tractable. One of the advantage of our present formulation of the $\mathfrak{sl}_3$ Toda theory is that it is possible to write down explicit expressions for these operators as follows.

For $z\in \C$ and $\alpha\in  \mathfrak{h}_3^*$ the regularized Vertex Operator $V_{\eps, \alpha}(z)$ is defined by
  \begin{equation}
 V_{\eps, \alpha}(z)\coloneqq  \hat g(z)^{\Delta_\alpha}e^{  \langle\alpha, \X_\eps(z) +\bm c\rangle-\frac{1}2\expect{\ps{\alpha,\X_\eps(z)}^2}}
 \end{equation}
  where $\X_\eps$ is the regularized field as above and recall that $\Delta_\alpha=\ps{\frac\alpha2,Q-\frac\alpha2}$. Using the interpretation for the path integral~\eqref{def:TPI} we define the $\mathfrak{sl}_3$ Toda regularized correlation functions as:
 \begin{equation}\label{correl_approx}
  \left\langle \prod_{k=1}^NV_{\alpha_k,\eps}(z_k)\right\rangle  \coloneqq  \int_{ \mathfrak{h}_3^*}    e^{-2\ps{Q,\bm c}}\E \left[\prod_{k=1}^N V_{\eps, \alpha_k}(z_k) e^{- \sum_{i=1}^{2} \mu_i e^{\gamma  \langle \bm{c},e_i\rangle }M^{ \gamma e_i,g}_\eps(\C)}  \right]d\bm c  .
\end{equation}
For future convenience we introduce 
\begin{equation}\label{eq:definition_si}
    s_i\coloneqq   \frac{\left\langle \sum_{k=1}^N \alpha_k-2Q,\omega_i \right\rangle}{\gamma} \cdot
\end{equation}
 An important result is the existence of constraints on $(s_i)_{i=1,2}$ that characterize the non-triviality of the limit of the regularized correlation functions (adapted from~\cite[Theorem 3.1]{Toda_construction}, see also \cite{Sei90}).
\begin{thmcite}\label{propcorrel1}
The correlation functions 
\begin{equation}
\Big\langle V_{\alpha_1}(z_1) \cdots V_{\alpha_N}(z_N) \Big\rangle\coloneqq \lim\limits_{\eps\rightarrow0}\Big\langle \prod_{k=1}^NV_{\alpha_k,\eps}(z_k)\Big\rangle
\end{equation} exist and are non trivial if and only if the \emph{Seiberg bounds} hold: 
\begin{equation}\label{seiberg_bounds}
\begin{split}
&\bullet\quad\text{For all }i=1,2,\quad s_i>0.\\
&\bullet\quad\text{For all }i=1,2\text{ and }1\leq k\leq N, \quad \langle\alpha_k, e_i \rangle < \langle Q, e_i \rangle= \gamma+\frac{2}{\gamma}.
\end{split}
\end{equation}
Moreover we have the following explicit expression
\begin{equation}\label{expressioncorrels}
\begin{split}
 &\langle V_{\alpha_1}(z_1) \cdots V_{\alpha_N}(z_N) \rangle=    \\
  & \left ( \prod_{i=1}^{2} \frac{\Gamma (s_i)\mu_i^{- s_i}}\gamma   \right )\prod_{ j < k} \norm{z_j-z_k}^{-\langle\alpha_j ,\alpha_k \rangle } \E \left [      \prod_{i=1}^{2}\left(\int_\C \frac{\hat g(x_i)^{-\frac\gamma4\sum_{k=1}^N\ps{\alpha_k,e_i}}}{\prod_{k=1}^N\norm{z_k-x_i}^{\gamma\ps{\alpha_k,e_i}}}M^{\gamma e_i,\hat g}(dx_i)\right)^{- s_i} \right ].
\end{split}
\end{equation}
The expectation on the right-hand side of the above expression is finite provided that the extended Seiberg bounds are satisified:
\begin{equation}\label{lemmaSeibergplus}
\text{For } i=1,2,\quad -s_i < \frac{2}{\gamma^2}\wedge \min_{k=1,\dots,N} \frac{1}{\gamma}\langle Q-\alpha_k,e_i\rangle.
\end{equation}
\end{thmcite}

\subsubsection{Miura transformation and  holomorphic currents} 
The W-symmetry manifests itself via the holomorphicity of some currents, which implies existence of the so-called \emph{Ward identities} recalled in the introduction.

In order to define these currents, let us write down the general form of the holomorphic currents associated to the symmetries of the model via the Miura transformation
\begin{equation}\label{Miura_2}
    \left(\frac q2\partial+\ps{h_{3},\partial\varphi}\right)\left(\frac q2\partial+\ps{h_{2},\partial\varphi}\right)\left(\frac q2\partial+\ps{h_{1},\partial\varphi}\right)\coloneqq \sum_{i=0}^{3}\bm {\mathrm W^{(3-i)}}\left(\frac q2\partial\right)^i.
\end{equation}
The expression~\eqref{Miura_2} should be understood in the sense of operators, by which we mean that the factors $\ps{h_i,\partial\varphi}$ act multiplicatively while $\partial$ acts as differential operator. 
For instance we have $\bm {\mathrm W^{(0)}}=1$, $\bm {\mathrm W^{(1)}}=0$ (since $h_1+h_2+h_3=0$) and 
\begin{equation*}
\bm {\mathrm W^{(2)}}=\frac q2\ps{2h_1+h_2,\partial^2\varphi}+\ps{h_1,\partial\varphi}\ps{h_2,\partial\varphi}+\ps{h_2,\partial\varphi}\ps{h_3,\partial\varphi}+\ps{h_3,\partial\varphi}\ps{h_1,\partial\varphi}.
\end{equation*}
The latter actually corresponds (up to a multiplicative factor) to the \emph{stress-energy tensor} of the $\mathfrak{sl}_3$ Toda theory, which we denote as $\bm {\mathrm T}\coloneqq 2\bm {\mathrm W^{(2)}}$. Indeed since $2h_1+h_2=\rho$ and for any $u,v\in\mathfrak{h}_3^*$,
\begin{equation}\label{eq:remarkable_identity}
\ps{h_1,u}\ps{h_{2},v}+\ps{h_2,u}\ps{h_{3},v}+\ps{h_3,u}\ps{h_{1},v}=-\frac{1}{2}\ps{u,v},
\end{equation}
this tensor can be rewritten in a more standard way as 
\begin{equation*}
\bm {\mathrm T}(z)=\ps{Q,\partial^2\varphi(z)}-\ps{\partial\varphi(z),\partial\varphi(z)}.
\end{equation*}

Since we are working with a quantum theory, the products that appear above are understood as Wick products and make sense when $\eps$-regularized. Therefore we introduce 
\begin{equation}\label{eq:approx_T}
    \bm {\mathrm T}_{\eps}(z)\coloneqq \ps{Q,\partial^2\varphi_\eps(z)}-:\ps{\partial\varphi_\eps(z),\partial\varphi_\eps(z)}:
\end{equation}
where the notation $:XY:$ is shorthand for $XY-\expect{XY}$ and $\varphi_\eps$ is the regularized Toda field $\varphi_\eps=\X_\eps+\frac{Q}{2}\ln\hat g_\eps+\bm c$ as before.

The regularized higher-spin current $\bm {\mathrm W^{(3)}}_\eps$ is defined with a similar procedure applied to the expression
\begin{equation*}
    \frac{q^2}{4}\ps{h_1,\partial^3\varphi}+\frac {q}{2}\left(\ps{h_1,\partial^2\varphi}\ps{h_2+h_3,\partial\varphi}+\ps{h_2,\partial^2\varphi}\ps{h_1,\partial\varphi}\right)
    +\ps{h_1,\partial\varphi}\ps{h_2,\partial\varphi}\ps{h_3,\partial\varphi}.
\end{equation*}
As suggested in the physics literature~\cite{BCP,FaZa,FaLi1}, in order to recover the standard commutation relations of the $W_3$-algebra and obtain an elegant expression for the quantum number associated to the spin-three tensor $w(\alpha)$, we may choose to shift this current by an additional factor $-\frac q8 \partial\bm {\mathrm T}$. Explicitly, we redefine the spin-three current by setting
\begin{equation}
       \begin{split}
           \bm {\mathrm W^{(3)}}_\eps\coloneqq& -\frac{q^2}{8}\ps{h_2,\partial^3\varphi_\eps}+\frac{q}{4}\left(:\ps{h_2-h_1,\partial^2\varphi_\eps}\ps{h_1,\partial\varphi_\eps}:+:\ps{h_3-h_2,\partial^2\varphi_\eps}\ps{h_3,\partial\varphi_\eps}:\right)\\
    &+:\ps{h_1,\partial\varphi_\eps}\ps{h_2,\partial\varphi_\eps}\ps{h_3,\partial\varphi_\eps}:
       \end{split} 
\end{equation}
with $:XYZ\coloneqq XYZ-X\expect {YZ}-Y\expect{XZ}-Z\expect{XY}$.

In the sequel we will study the $\eps\rightarrow0$ limit of the quantity
\begin{equation*}
    \left\langle \bm {\mathrm W^{(3)}}_\eps\prod_{k=1}^NV_{\alpha_k,\eps}(z_k) \right\rangle.
\end{equation*}
In order not to overload notations, we will write $\bm {\mathrm W}$ for $\bm {\mathrm W^{(3)}}$.


\section{Regularity of the Toda correlation functions}\label{section_interpretation}
The purpose of this section is to provide some elementary results needed in order to prove the local Ward identities~\eqref{Ward_stress} and ~\eqref{Ward_higher}. Namely we explain how to make sense of the quantities presented in the previous Section~\ref{section_construction} thanks to \emph{Gaussian integration by parts} and give some regularity results for the Toda correlation functions.
\subsection{Probabilistic tools}
\subsubsection{The regularized Toda correlation functions}
According to the expression of the holomorphic currents of the $\mathfrak{sl}_3$ Toda theory, we will consider expressions that involve derivatives of the Toda field. As a consequence, we will work with a regularized versions of the objects introduced before. However in the computation of the Ward identities, there will be one extra point (that we denoted $z_0$) that does not appear within the Vertex Operators (that are evaluated at $z_1,\cdots,z_N$) but rather in the holomorphic current $\bm {\mathrm T}$ or $\bm {\mathrm W}$. Therefore we will use two different kinds of regularization corresponding either to the variables $(z_1,\cdots,z_N)$ appearing in the Vertex Operators or $z_0$ appearing in the currents.
As we will see having these two regularizations prove to be useful in order to simplify the computations. We follow the approach developed in~\cite{Zhu_BPZ}.

Let $\theta:\R^+\to[0,1]$ be a smooth function which is $0$ in $[0,\frac12]$ and $1$ in $[1,+\infty)$. For positive $\delta$ and complex $z$ we set $\theta_{\delta}(z)\coloneqq \theta\left(\frac{\norm{z}}{\delta}\right)$, such that $\theta_{\delta}(z)\rightarrow 1$ when $\delta\rightarrow0$ for any $z\neq 0$. In the sequel, $\bm z$ will denote the vector of complex numbers $(z_1,\dots, z_N)$ and $\bm \alpha$ the vector of parameters $(\alpha_1,\dots, \alpha_N)$ in $\left(\mathfrak{h}_3^*\right)^N$.

Let us assume that $(z_0,\bm z)$ are distinct complex points with $z_0\neq 0$ and that we are given $\bm\alpha\in\left(\mathfrak{h}_3^*\right)^N$ satisfying the Seiberg bounds~\eqref{seiberg_bounds}. 
Then for positive $\eps$ and $\delta$, the $(\eps,\delta)$-regularized correlation functions are defined by
\begin{equation}\label{correl_approx_bis}
  \left\langle\prod_{k=0}^NV_{\alpha_k,\eps}(z_k)\right\rangle_{\delta}\coloneqq  
  \int_{\mathfrak{h}_3^*}e^{-2\ps{Q,\bm c}}\E\left[\left(\prod_{k=1}^N V_{\eps, \alpha_k}(z_k)\right)e^{-\sum_{i=1}^{2}\mu e^{\gamma\langle\bm c,e_i\rangle}M^{\gamma e_i}_{\eps,\delta}(\C)}\right]d\bm c.
\end{equation}
The only difference with the expression~\eqref{correl_approx} is that we also regularize the total mass of the GMC inside the expectation term above around the point $z_0$
\begin{equation*}
    M^{ \gamma e_i}_{\eps, \delta}(\C)\coloneqq \int_\C \theta_{\delta}(x-z_0)M^{ \gamma e_i}_\eps(dx). 
\end{equation*}
The reason for this extra regularization will be made clear later; informally we remove the singularities coming from the holomorphic currents. For future convenience, we introduce the shorthand 
\begin{equation*}
\bm{\mathrm V}_\eps\coloneqq \prod_{k=1}^NV_{\alpha_k,\eps}(z_k).
\end{equation*}

Recall that $\hat g_\eps(z)\coloneqq \hat g*\eta_\eps(z)$ with $\eta_\eps$ as in~\eqref{regularization} and set
\begin{equation*}
   \norm{x}_\eps\coloneqq \exp\left(\int_\C\int_\C \ln\norm{x-z+z'}\eta_\eps(z)\eta_\eps(z')d^2zd^2z'\right).
\end{equation*}
We also use a regularization of the complex-valued map $x\mapsto\frac1{x^p}$ for a positive integer $p$ by setting
\begin{equation*}
\frac1{(x)^p_\eps}\coloneqq \frac{(-1)^p}{(p-1)!}\int_\C\int_\C \frac{1}{(x-z+z')}\partial_z^{p-1}\eta_\eps(z)\eta_\eps(z')d^2zd^2z'.
\end{equation*}
The connection with the map $x\mapsto\frac1{x^p}$ is made clear by the following statement.
\begin{lemma}\label{approx_inverse}
Let $\eps>0$ and $p$ be a positive integer, and recall $R$ from the regularization~\eqref{regularization}. Then for any $\norm{x}>4R\eps$ it holds that
\begin{equation*}
    \frac{x^p}{(x)_\eps^p}-1=\left(\frac{\eps}{x}\right)^2\mathrm{F}\left(\frac{\eps}{x}\right)
\end{equation*}
where $\mathrm{F}$ is continuous over $\C$ and depends only on the mollifier $\eta$ of the $\eps$-regularization. In particular, the family of functions $\left(\frac{1}{(x)_\eps^p}\right)_{\eps>0}$ converges uniformly to $\frac{1}{x^p}$ on every compact set of $\C\setminus\{0\}$.
Similarly, the quantity
\begin{equation*}
\sup\limits_{0<\norm{x}\leq 4R\eps}\norm{\frac{x^p}{(x)_\eps^p}}
\end{equation*}
remains bounded uniformly on $\eps$.
\end{lemma}
We postpone the proof of this claim and of the next one to Section~\ref{section_appendix} since they are not informative.
Along the proof of Theorem~\ref{main_theorem} we will need to integrate certain correlation functions over the complex plane. The integration variables will show up in these correlation functions in the form of Vertex Operators with weights $V_{\gamma e_i}$, $i=1,2$. The following statement addresses this issue by ensuring that the quantities involved do make sense:
\begin{lemma}\label{integrability}
Assume that the $\bm{z}\coloneqq(z_1,\cdots,z_N)$ are distinct and $\bm{\alpha}\coloneqq(\alpha_1,\cdots,\alpha_N)\in \left(\mathfrak{h}_3^*\right)^N$ are such that the Seiberg bounds~\eqref{seiberg_bounds} are satisfied. Consider complex vectors $\bm{x}_1=(x_1^{(1)},\cdots,x_1^{(r_1)})$ and $\bm{x}_{2}=(x_{2}^{(1)},\cdots,x_{2}^{(r_2)})$. 
\begin{enumerate}
    \item For $\eps,\delta>0$, there exists a positive constant $C_\eps$, independent of $\delta$, such that
    \begin{equation*}
        \ps{\prod\limits_{i=1}^{r_1}V_{\gamma e_1,\eps}\left(x_1^{(i)}\right)\prod\limits_{j=1}^{r_2}V_{\gamma e_2,\eps}\left(x_2^{(j)}\right)\V_\eps}_{\delta}\leq C_\eps \prod\limits_{i=1}^{r_1}\left(1+\norm{x_1^{(i)}}\right)^{-4}\prod_{j=1}^{r_2}\left(1+\norm{x_2^{(j)}}\right)^{-4}.
    \end{equation*}
    \item For $\eps=\delta=0$ and $\rho>0$, if $\bm{x}_1,\bm{x}_2$ and $\bm{z}$ are in the domain $U_\rho\coloneqq \left\lbrace{\bm w, \rho<\min\limits_{i\neq j}\norm{w_i-w_j}}\right\rbrace$, the above bound remains true uniformly over $(\bm{x}_1,\bm{x_2})$, with a constant $C$ that depends only on $\rho$.
    \item Assume that all pairs of points in $\bm{z}$ are separated by some distance $\rho>0$ except for one pair $(z_1,z_2)$, and that all the $z$'s stay in a compact subset of $\C$. Further assume that $\ps{\alpha_1+\alpha_2-Q,e_1}<0$. Then, for any positive $\eta$, there exists a positive constant $K=K(\rho)$ such that, uniformly on $\eps$ and $\delta$,
        \begin{equation}
        \left\langle \V_\eps \right\rangle_{\delta}\leq K(\rho) \norm{z_1-z_2}^{-\ps{\alpha_1,\alpha_2}-\eta+\frac{1}{2}\ps{\alpha_1+\alpha_2-Q,e_2}^2\mathds{1}_{\ps{\alpha_1+\alpha_2-Q,e_2}>0}}.
        \end{equation}
    If $\ps{\alpha_1+\alpha_2-Q,e_2}<0$, we have the same estimate by switching $e_1\leftrightarrow e_2$.
    \item In the previous estimate, if $\alpha_1,\alpha_2\in\{\gamma e_1,\gamma e_2\}$, then there exists some $\zeta>0$ such that
    \begin{equation*}
    \left\langle \V_\eps \right\rangle_{\delta}\leq K \norm{z_1-z_2}^{-2+\zeta}.
    \end{equation*}
    \item For some $p>1$ that depends on $\gamma$ and $\ps{\alpha_0,e_1}$, the map
        \begin{equation*}
        x_1^{(1)}\mapsto \sup\limits_{\eps}\left\langle \prod\limits_{i=1}^{r_1}V_{\gamma e_1,\eps}(x_1^{(i)})\prod\limits_{j=1}^{r_2}V_{\gamma e_2,\eps}(x_2^{(j)})\V_\eps \right\rangle_{\delta}
        \end{equation*}
        is in $L^p(\C)$.
\end{enumerate}

\end{lemma}
As a first application of these technical results we provide a rather useful identity. This identity is derived from the $\mu$-dependence of the correlation functions, usually referred to as a \lq\lq KPZ-relation"~\cite{KPZ}.
\begin{lemma}[KPZ identity]\label{KPZ_integral}
Assume that the $\bm{z}$ are distinct and that the $\bm{\alpha}$ are such that the Seiberg bounds~\eqref{seiberg_bounds} hold. Then for non-negative $\eps$ and $\delta$, the following equality holds in $\mathfrak{h}_3^*$:
\begin{equation}
    \left(\sum_{k=1}^N\alpha_k-2Q\right)\ps{\V_\eps}_\delta=\gamma\sum_{i=1}^{2}\mu_ie_i\int_\C\theta_\delta(z_0-x)\ps{V_{\gamma e_i,\eps}(x)\V_\eps}_\delta d^2x.
\end{equation}
\end{lemma}
\begin{proof}
It suffices to check that for all $i=1,2$ one has that
\begin{equation*}
\ps{\sum_{k=1}^N\alpha_k-2Q,\omega_i}\ps{\V_\eps}_\delta=\mu_i\gamma \int_\C\theta_\delta(z_0-x)\ps{V_{\gamma e_i,\eps}(x)\V_\eps}_\delta d^2x.
\end{equation*}
To do so, we differentiate with respect to $\mu_i$ the expression
\begin{equation*}
  \ps{\V_\eps}_{\mu_1,\mu_2,\delta}=\int_{ \mathfrak{h}_3^*}   e^{-\ps{2Q,\bm c}} \E \left [\V_\eps e^{- \sum_{i=1}^{2} \mu_i e^{\gamma  \langle \bm c,e_i\rangle }M^{ \gamma e_i}_{\eps,\delta}(\C)} \right  ]d\bm c
\end{equation*}
to get for $i=1,2$ the following:
\begin{equation*}
-\int_{ \mathfrak{h}_3^*}   e^{-\ps{2Q,\bm c}} \E \left [\V_\eps e^{\gamma  \langle \bm c,e_i\rangle }M^{ \gamma e_i}_{\eps,\delta}(\C)e^{- \sum_{i=1}^{2} \mu_i e^{\gamma  \langle \bm c,e_i\rangle }M^{ \gamma e_i}_{\eps,\delta}(\C)} \right  ]d\bm c.
\end{equation*}
By positivity of the above quantities and Fubini-Tonelli, the latter expression is equal to
\begin{equation*}
-\int_\C\theta_\delta(z_0-x)\int_{ \mathfrak{h}_3^*}   e^{-\ps{2Q,\bm c}} \E \left [\V_\eps V_{\gamma e_i,\eps}(x)e^{- \sum_{i=1}^{2} \mu_i e^{\gamma  \langle \bm c,e_i\rangle }M^{ \gamma e_i}_{\eps,\delta}(\C)} \right]d\bm cd^2x
\end{equation*}
where we have used the fact that $M_{\eps,\delta}^{\gamma e_i}(dx)=\theta_\delta(z_0-x)V_{\gamma e_i,\eps}(x)d^2x$. This yields 
\begin{equation*}
\frac{\partial}{\partial{\mu_i}}\ps{\V_\eps}_{\mu_1,\mu_2,\delta}=-\int_\C\theta_\delta(z_0-x)\ps{V_{\gamma e_i,\eps}(x)\prod_{k=1}^NV_{\alpha_k,\eps}(z_k)}_\delta d^2x.
\end{equation*}
However, by making a change of variables in the zero-mode $\bm c\to \bm c- \frac\rho\gamma\ln\mu_i$, the $\mu_i$-dependence of the correlation functions is explicit and in agreement with Equation~\eqref{expressioncorrels}:
\begin{equation*}
\ps{\prod_{k=1}^NV_{\alpha_k,\eps}(z_k) }_{\mu_1,\mu_2,\delta}=\mu_1^{-s_1}\mu_2^{-s_2}\ps{\prod_{k=1}^NV_{\alpha_k,\eps}(z_k)}_{1,1,\delta}.
\end{equation*}
When we differentiate it with respect to $\mu_i$, we end up with (recall Equation~\eqref{eq:definition_si})
\begin{equation*}
-\frac{\ps{\sum_{k=1}^N\alpha_k-2Q,\omega_i}}{\mu_i\gamma}\ps{\prod_{k=1}^NV_{\alpha_k,\eps}(z_k)}_{\mu_1,\mu_2,\delta}.
\end{equation*}
We recover the desired result for positive $\eps$ and $\delta$ by identification.

For the $\eps\to 0$ limit, we use the last item of Lemma~\ref{integrability}. Finally, all arguments go in the same way without the $\delta$-regularization.
\end{proof} 

\subsubsection{Gaussian integration by parts}
When computing the Ward identity~\eqref{Ward_higher}, we insert the spin-three current $\bm{\mathrm{W}}(z_0)$ within the correlation functions. Put differently we will need to make sense of terms that take the form 
\begin{equation*}
    \Big\langle\ps{\alpha,\partial^p\varphi_\eps(z_0)}\bm{\mathrm V}_\eps\Big\rangle_\delta.
\end{equation*}
It is important to highlight that the extra $\delta$-regularization introduced is done with respect to a special point $z_0$, which does not appear within the Vertex Operators $\V_\eps$.

The following proposition explains how to apply Gaussian integration by parts in our context in order to treat the insertion of this current.
\begin{proposition}\label{GaussianIPP}
Let $p$ be a positive integer and $\alpha$ in $\mathfrak{h}_3^*$. Then for any positive $\eps$ and $\delta$:
\begin{equation*}
\begin{split}
&\Big\langle\ps{\alpha,\partial^p\varphi_\eps(z_0)}\bm{\mathrm V}_\eps\Big\rangle_\delta\\
=&\frac{(-1)^p(p-1)!}{2}\left(\sum_{k=1}^N\frac{\ps{\alpha,\alpha_k}}{(z_0-z_k)_\eps^p}\ps{\bm{\mathrm V}_\eps}_\delta -\sum_{i=1}^2\mu_i
\ps{\alpha,\gamma e_i}\int_\C\frac{\theta_\delta(z_0-x)}{(z_0-x)^p_\eps}\ps{V_{\gamma e_i,\eps}(x)\bm{\mathrm V}_\eps}_\delta d^2x\right).
\end{split}
\end{equation*}
\end{proposition}
\begin{proof}
We recall the Gaussian integration by parts formula for a centered Gaussian vector $(X,Y_1,\dots,Y_N)$ and $f$ a smooth function on $\R^N$ with bounded derivatives:
\begin{equation}\label{eq:Gauss_IPP}
\E\left[Xf(Y_1,\dots,Y_N)\right]=\sum_{k=1}^N\E\left[XY_k\right]\E\left[\partial_{Y_k}f(Y_1,\dots,Y_N)\right],
\end{equation}
as well as the defining expression which follows from Equation~\eqref{def:TPI}:
\begin{equation*}
\begin{split}
&\left\langle \langle\alpha,\X_\eps+\frac{Q}{2}\ln\hat g_\eps\rangle(z_0)\bm{\mathrm V}_\eps\right\rangle_{\delta}
=\int_{\mathfrak{h}_3^*}e^{-\ps{2Q,\bm c}}\E \left [\langle\alpha,\X_\eps+\frac{Q}{2}\ln\hat g_\eps\rangle(z_0)\V_\eps e^{- \sum_{i=1}^{2}  \mu_ie^{\gamma  \langle\bm{c},e_i\rangle}M^{\gamma e_i,\hat g}_{\eps,\delta}(\C)}\right]d\bm c.
\end{split}
\end{equation*}
In order to apply Gaussian integration by parts~\eqref{eq:Gauss_IPP} to the expectation term \[
\E \left [\langle\alpha,\X_\eps+\frac{Q}{2}\ln\hat g_\eps\rangle(z_0)\V_\eps e^{- \sum_{i=1}^{2}  \mu_ie^{\gamma  \langle\bm{c},e_i\rangle}M^{\gamma e_i,\hat g}_{\eps,\delta}(\C)}\right],
\]
we may write the integral over the complex plane that appears in the exponential term, $M^{\gamma e_i,\hat g}_{\eps,\delta}(\C)$ as a Riemann sum. Since we can apply Equation~\eqref{eq:Gauss_IPP} to the partial sum and then take the limit, this means that we may directly apply Gaussian integration by parts to the regularized correlation functions to get that
\begin{equation*}
\begin{split}
    \left\langle\langle\alpha,\X_\eps+\frac Q2\ln\hat g_\eps\rangle(z_0)\bm{\mathrm V}_\eps\right\rangle_\delta&=\left(\sum_{k=1}^N\langle\alpha,\alpha_k\rangle G_{\hat g,\eps}(z_0,z_k) +\langle\alpha,\frac Q2\ln\hat  g_\eps\rangle(z_0)\right)\ps{\bm{\mathrm V}_\eps}_\delta\\
    &-\sum_{i=1}^2\frac{\mu_i\gamma}2\ps{\alpha,e_i}\int_\C\theta_\delta(z_0-x)G_{\hat g,\eps}(z_0,x)\ps{V_{\gamma e_i,\eps}(x)\bm{\mathrm V}_\eps}_\delta d^2x,
\end{split}
\end{equation*}
where $G_{\hat g,\eps}$ is the covariance kernel of the regularized field $\X_\eps$. One key point is that when taking derivatives of the Toda field we no longer take into account the zero mode $\bm{c}$ in the field $\varphi_\eps$. As a consequence we see that for positive $p$, 
\begin{equation*}
\Big\langle\ps{\alpha,\partial^p\varphi_\eps(z_0)}\bm{\mathrm V}_\eps\Big\rangle_\delta=\Big\langle\langle\alpha,\partial^p\left(\X_\eps+\frac{Q}{2}\ln\hat g_\eps\right)\rangle(z_0)\bm{\mathrm V}_\eps\Big\rangle_\delta\cdot
\end{equation*}
Therefore we end up with:
\begin{equation*}
\begin{split}
    \left\langle\langle\alpha,\partial^p\left(\X_\eps+\frac Q2\ln\hat g_\eps\right)(z_0)\rangle\bm{\mathrm V}_\eps\right\rangle_\delta=&\partial^p\left(\sum_{k=1}^N\ps{\alpha,\alpha_k}G_{\hat g,\eps}(z_0,z_k) +\langle\alpha,\frac Q2\ln\hat  g_\eps\rangle(z_0)\right)\ps{\bm{\mathrm V}_\eps}_\delta\\
    -&\sum_{i=1}^2\frac{\mu_i\gamma}2\ps{\alpha,e_i}\int_\C\theta_\delta(z_0-x)\partial^p G_{\hat g,\eps}(z_0,x)\ps{V_{\gamma e_i,\eps}(x)\bm{\mathrm V}_\eps}_\delta d^2x.
\end{split}
\end{equation*}
Note that here we do \textbf{not} take derivatives of $\theta_\delta$ in the integral: one should think of the equation above as taking the derivative of the field $\varphi_\eps(z)$ then evaluate at $z=z_0$.

From the explicit expression of the Green function $G_{\hat g}$ in the round metric~\eqref{Green_round} we get
\begin{equation*}
\begin{split}
    &\Big\langle\ps{\alpha,\partial^p\varphi_\eps(z_0)}\bm{\mathrm V}_\eps\Big\rangle_\delta\\
    ={}&\frac{(-1)^p(p-1)!}{2}\left(\sum_{k=1}^N\frac{\ps{\alpha,\alpha_k}}{(z_0-z_k)^p_\eps}-\sum_{i=1}^2\mu_i\ps{\alpha,\gamma e_i}\int_\C\frac{\theta_\delta(z_0-x)}{(z_0-x)^p_\eps}\ps{V_{\gamma e_i,\eps}(x)\bm{\mathrm V}_\eps}_\delta d^2x\right)\\
    &+\frac{1}{4}\partial^p\ln\hat g_\eps(z_0)\left\langle\alpha,\left(\sum_{k=1}^N\alpha_k-2Q\right)\ps{V_\eps}_\delta-\sum_{i=1}^2\mu_i\gamma e_i\int_\C\theta_\delta(z_0-x)\ps{V_{\gamma e_i,\eps}(x)\bm{\mathrm V}_\eps}_\delta d^2x\right\rangle.
\end{split}
\end{equation*}
The above integrals are indeed well-defined thanks to item (5) of Lemma~\ref{integrability}. Now we can use the KPZ identity~\eqref{KPZ_integral} to see that the metric-dependent term ---the last line--- equals zero.
\end{proof}
As can be seen from the expression of $\bm{W}$, we will need to extend the previous statement to products of derivatives of the field. The proper way to treat such terms is to apply recursively the following formula, that follows from Proposition~\ref{GaussianIPP}
\begin{equation}\label{eq:IPP_product}
\begin{split}
    \frac{2}{(-1)^{p_r}(p_r-1)!}&\left\langle:\prod_{l=1}^r\ps{\beta_l,\partial^{p_l}\varphi_\eps(z_0)}:\V_\eps\right\rangle_\delta\\
    &=\sum_{k=1}^N\frac{\ps{\beta_r,\alpha_k}}{(z_0-z_k)_\eps^{p_r}}\ps{:\prod_{l=1}^{r-1}\ps{\beta_l,\partial^{p_l}\varphi_\eps(z_0)}:\bm{\mathrm V}_\eps}_\delta\\
    &-\sum_{i=1}^2\mu_i\ps{\beta_r,\gamma e_i}\int_\C\frac{\theta_\delta(z_0-x)}{(z_0-x)^p_\eps}\ps{:\prod_{l=1}^{r-1}\ps{\beta_l,\partial^{p_l}\varphi_\eps(z_0)}:V_{\gamma e_i,\eps}(x)\bm{\mathrm V}_\eps}_\delta d^2x,
\end{split}
\end{equation}
valid for $p_1,\cdots,p_r$ positive integers and $\beta_1,\cdots,\beta_r$ any elements of $\mathfrak{h}_3^*$. This can be derived using the same reasoning and the definition of the Wick product as before. 
\begin{remark}
We assume in the sequel that $\mu_1=\mu_2=\mu$ for better readability. This is not restrictive thanks to the explicit dependence in $\mu_1$, $\mu_2$ of the correlation functions.
\end{remark}

\subsection{Derivatives of the correlation functions}
We turn to the study of the regularity of Toda correlation functions. In this subsection, we keep the $\delta$-regularization around the point $z_0$, but this point is now a localization where a Vertex Operator is evaluated (instead of a point where the currents were elavuated as it was before). Therefore the extra $\delta$-regularization is now made around a singular point which appears in the product of Vertex Operators.

We do so to make the proof more natural and describe a general algorithmic procedure to show that the Toda correlation functions are smooth. In the present document we will not prove the smoothness, and only require the weaker result:
\begin{proposition}\label{regularity}
Assume that the weights $\bm\alpha=(\alpha_0,\alpha_1,\cdots,\alpha_N)$ are such that the Seiberg bounds~\eqref{seiberg_bounds} hold.
Then $z_0\mapsto \ps{\prod_{k=0}^NV_{\alpha_k}(z_k)}$ is $C^2$ on the set $\C\setminus\lbrace z_1,\cdots,z_N\rbrace$.
\end{proposition}
The same results holds true for the other points $z_k\mapsto \ps{\prod_{k=0}^NV_{\alpha_k}(z_k)}$, $k=1,\dots,N$, over which no $\delta$-regularization is made.
\begin{proof}
Let us consider the $(\eps,\delta)$-regularized correlation functions $\ps{\bm{\mathrm V}_\eps}_\delta$, with $\delta$ small enough so that $r\coloneqq \frac12\min\limits_{1\leq k\leq N}\norm{z_0-z_k}>\delta$. In the following, we only consider the $\frac{\partial}{\partial z}$ derivative; the calculation for the $\frac{\partial}{\partial{\bar{z}}}$ derivative is exactly the same.

From the results of Lemma~\ref{integrability} and following the proof of Proposition~\ref{GaussianIPP}, we know that $\ps{V_{\alpha_0,\eps}(z_0)\bm{\mathrm V}_\eps}_\delta$ is differentiable with respect to $z_0$, with derivative
\begin{equation*}
\begin{split}
    &-\frac12\sum_{k=1}^N\frac{\ps{\alpha_0,\alpha_k}}{(z_0-z_k)_\eps}\ps{V_{\alpha_0,\eps}(z_0)\bm{\mathrm V}_\eps}_\delta\\
    &+\mu\sum_{i=1}^2\int_\C\left(\frac{\gamma\ps{\alpha_0,e_i}}2\frac{\theta_\delta(z_0-x) }{(z_0-x)_\eps}-\partial_{z_0}\theta_\delta(z_0-x)\right)\ps{V_{\gamma e_i,\eps}(x)V_{\alpha_0,\eps}(z_0)\bm{\mathrm V}_\eps}_\delta d^2x.
\end{split}
\end{equation*}
Notice the difference between this expression and the one we got when applying Gaussian integration by parts in Proposition~\ref{GaussianIPP}: the integral now contains derivatives of $\theta_\delta$. This arises since the singularity $z_0$ is now in the Vertex Operators instead of being in the current. 

From the asymptotics at infinity in Lemma~\ref{integrability}, the only issue when taking the $\eps,\delta\rightarrow0$ limit comes from the behaviour of the integral term near the singularity $\frac{1}{z_0-x}$. 
To treat it we use the fact that $\partial_{z_0}\theta_\delta(z_0-x)=-\partial_{x}\theta_\delta(z_0-x)$ to rewrite the integral as 
\begin{align*}
    &\mu\sum_{i=1}^2\left(\int_{B(z_0,r)^c}\frac{\gamma\ps{\alpha_0,e_i}}2\frac{\theta_\delta(z_0-x) }{(z_0-x)_\eps}\ps{V_{\gamma e_i,\eps}(x)V_{\alpha_0,\eps}(z_0)\bm{\mathrm V}_\eps}_\delta d^2x \right.\\
    &\left.+\int_{A(z_0,\delta/2,r)}\left(\frac{\gamma\ps{\alpha_0,e_i}}2\frac{\theta_\delta(z_0-x) }{(z_0-x)_\eps}+\partial_{x}\theta_\delta(z_0-x)\right)\ps{V_{\gamma e_i,\eps}(x)V_{\alpha_0,\eps}(z_0)\bm{\mathrm V}_\eps}_\delta d^2x\right),
\end{align*}
where $A(z_0,\delta/2,r)$ is the annulus of radii $\delta/2,r$ centered at $z_0$. We seperate out in this way the first integral over the domain $B(z_0,r)^c$, which is easily dealt with as $\eps,\delta$ go to zero thanks to item (2) of Lemma~\ref{integrability}.

For the other integral, we proceed by integration by parts to get,\footnote{That is, we apply Stokes' formula $\oint_{\partial B(z_0,r)}f(\xi)g(\xi) \frac{\sqrt{-1}d\bar\xi}2=\int_{B(z_0,r)}\partial_xf(x)g(x)+\partial_xg(x)f(x) d^2x$ to the term $\partial_{x}\theta_\delta(z_0-x)\ps{V_{\gamma e_i,\eps}(x)V_{\alpha_0,\eps}(z_0)\bm{\mathrm V}_\eps}_\delta$.} with $A=A(z_0,\delta/2,r)$,
\begin{equation*}
\begin{split}
&\mu\sum_{i=1}^2\left[\oint_{\partial B(z_0,r)}\frac{\theta_\delta(z_0-\xi) }{(z_0-\xi)_\eps}\ps{V_{\gamma e_i,\eps}(\xi)V_{\alpha_0,\eps}(z_0)\bm{\mathrm V}_\eps}_\delta \frac{\sqrt{-1}d\bar\xi}2\right. \\
&\quad\quad\left.-\int_{A}\sum_{k=1}^N\frac{\gamma\ps{\alpha_k,e_i}}2\frac{\theta_\delta(z_0-x)}{(z_k-x)_\eps}\ps{V_{\gamma e_i,\eps}(x)V_{\alpha_0,\eps}(z_0)\bm{\mathrm V}_\eps}_\delta d^2x\right]\\
&-\sum_{i,j=1}^2\frac{(\mu\gamma)^2\ps{e_i,e_j}}2\int_{A}\int_\C \frac{\theta_\delta(z_0-x_1)\theta_\delta(z_0-x_2)}{(x_1-x_2)_\eps}\ps{V_{\gamma e_i,\eps}(x_1)V_{\gamma e_j,\eps}(x_2)V_{\alpha_0,\eps}(z_0)\bm{\mathrm V}_\eps}_\delta d^2x_1d^2x_2.
\end{split}
\end{equation*}
The first two terms remain bounded in the $\eps,\delta\rightarrow0$ limit. By symmetry between the $x_1$ and $x_2$ variables, we rewrite the last term as
\begin{equation*}
\begin{split}
&\sum_{i,j=1}^2\frac{(\mu\gamma)^2\ps{e_i,e_j}}2\int_{A}\int_\C \frac{\theta_\delta(z_0-x_1)\theta_\delta(z_0-x_2)}{(x_1-x_2)_\eps}\ps{V_{\gamma e_i,\eps}(x_1)V_{\gamma e_j,\eps}(x_2)V_{\alpha_0,\eps}(z_0)\bm{\mathrm V}_\eps}_\delta d^2x_1d^2x_2\\
={}&\sum_{i,j=1}^2\frac{(\mu\gamma)^2\ps{e_i,e_j}}2\int_{A}\int_{B(z_0,r)^c}\frac{\theta_\delta(z_0-x_1)\theta_\delta(z_0-x_2)}{(x_1-x_2)_\eps}\ps{V_{\gamma e_i,\eps}(x_1)V_{\gamma e_j,\eps}(x_2)V_{\alpha_0,\eps}(z_0)\bm{\mathrm V}_\eps}_\delta d^2x_1d^2x_2\\
&+\int_{A^2}\frac{\theta_\delta(z_0-x_1)\theta_\delta(z_0-x_2)}{(x_1-x_2)_\eps}\sum_{i,j=1}^2\frac{(\mu\gamma)^2\ps{e_i,e_j}}2\ps{V_{\gamma e_i,\eps}(x_1)V_{\gamma e_j,\eps}(x_2)V_{\alpha_0,\eps}(z_0)\bm{\mathrm V}_\eps}_\delta d^2x_1d^2x_2,
\end{split}
\end{equation*}
where the last line is identically zero by symmetry between the $x_1$, $x_2$ variables. 

For the $C^1$ regularity it remains to show that when $\eps,\delta\rightarrow0$, the quantity
\[
\sum_{i,j=1}^2\frac{(\mu\gamma)^2\ps{e_i,e_j}}2\int_{A}\int_{B(z_0,r)^c}\frac{\theta_\delta(z_0-x_1)\theta_\delta(z_0-x_2)}{(x_1-x_2)_\eps}\ps{V_{\gamma e_i,\eps}(x_1)V_{\gamma e_j,\eps}(x_2)V_{\alpha_0,\eps}(z_0)\bm{\mathrm V}_\eps}_\delta d^2x_1d^2x_2
\]
converges. The only singular terms that may occur in the integrand correspond to the case where $\norm{x_1-x_2}$ tends to zero. This in turn implies that $\norm{x_1-z_0},\norm{x_2-z_0}$ tend to $r$, and therefore that both $x_1$ and $x_2$ are close to $\partial B(z_0,r)$. From item (4) of Lemma~\ref{integrability} we know that the exponent on the two-points fusion asymptotics for the Vertex Operators $V_{\gamma e_i,\eps}(x_1)V_{\gamma e_j,\eps}(x_2)$ is given in the worst case scenario (that is when $i=j$) by $-2\gamma^2$ in the regime $0<\gamma<\sqrt{\frac23}$ and by $-2\gamma^2+\frac{(3\gamma-\frac2\gamma)^2}{2}$ when $\sqrt{\frac23}<\gamma<\sqrt 2$. In any case this exponent is strictly greater than $-2$, hence the latter integral is absolutely convergent (see also Lemma~\ref{fusion}). 

The proof for the $C^2$ regularity proceeds in the same way, but requires slightly more involved arguments. We postpone it to Subsection~\ref{proof_C2}.

Finally, the proof for the regularity associated to the other points $z_k$ is handled in a similar fashion. The only difference being that the different integration by parts do not appear naturally since there are no longer terms containing derivatives of $\theta_\delta$.
\end{proof}

\section{Ward identities for the $W_3$-algebra}\label{section_proof}
With all these tools at hand, we are now in position to prove the main result of the present document, that is that the spin-two and spin-three Ward identities~\eqref{Ward_stress},~\eqref{Ward_higher} do hold:
\begin{align*}
    &\ps{\bm {\mathrm T}(z_0)\bm{\mathrm V}}=\sum_{k=1}^N\left(\frac{\Delta_ {\alpha_k}}{(z_0-z_k)^2}
    +\frac{\partial_{z_k}}{z_0-z_k}\right)\ps{\bm{\mathrm V}}\\
    &\ps{\bm {\mathrm W}(z_0)\bm{\mathrm V}}=-\frac18\sum_{k=1}^N\left(\frac{w(\alpha_k)}{(z_0-z_k)^3}
    +\frac{\bm {\mathcal W_{-1}^{(k)}}}{(z_0-z_k)^2}+\frac{\bm {\mathcal W_{-2}^{(k)}}}{z_0-z_k}\right)\ps{\bm{\mathrm V}}\cdot
\end{align*}
The derivatives in the first line are holomorphic derivatives (and will be in the sequel).

\subsection{Computation of the Ward identities: strategy of the proof}
In order to prove these identities, we work with the regularized version of the correlation functions. We can apply Proposition~\ref{GaussianIPP} to rewrite explicitly the left-hand side $\ps{\bm {\mathrm W}(z_0)\bm{\mathrm V}}$ in terms of multiple integrals in the $x$ variable containing singularities of the form $x\mapsto\frac{1}{(x-z_0)^p}$. We proceed in the same way with the right-hand side in the Ward identity, which will also yield multiple integrals containing singularities of the form $x\mapsto\frac{1}{(x-z_k)^p}$. Our strategy to treat the terms we get is as follows:
\begin{itemize}
    \item We transform the right-hand side so that the only singularities appearing within integrals are of the form $\frac1{(x-z_0)^p}$. Using identities such as ${\frac{1}{(z_0-z_k)(x-z_k)}=\frac{1}{z_0-x}\left(\frac{1}{x-z_k}-\frac{1}{z_0-z_k}\right)}$ (that we call \lq\lq symmetrization" identities) we rewrite terms containing expressions of the form $\frac1{x-z_k}$ as derivatives of the correlation functions. We then use integration by parts to transform them into the desired form.
    \item We take the $\eps\rightarrow0$ limit of the expression obtained. Since singularities are only present around $z_0$ this limit makes sense, so we only have to take care of the residual terms coming from the symmetrization. This is done by using Lemma~\ref{approx_inverse} which shows that these become negligible in the $\eps\rightarrow0$ limit.
    \item Finally we let $\delta$ go to zero: this step requires a careful study of the remaining terms and is based on a reasoning similar to the one developed in the proof of Proposition~\ref{regularity} on the differentiability of the correlation functions. The fact that the correlation functions are $C^1$ will be used during this last step.
\end{itemize} 
In the limit we recover the desired result; the rest of this section provides details on this procedure.

\subsection{Computation of the spin-two Ward identity}
As a warm-up, let us discuss in detail the proof of the Ward identity for the stress-energy tensor $\bm {\mathrm T}$. In the regularized expression for $\bm {\mathrm T}$ given by Equation~\eqref{eq:approx_T} one has
\begin{equation*}
    \ps{\bm {\mathrm T}_\eps(z_0)\bm{\mathrm V}_\eps}_\delta=\ps{\ps{Q,\partial^2\varphi_\eps(z_0)}\V_\eps}_\delta+\sum_{p=1}^3\ps{:\ps{h_p,\partial\varphi_\eps(z_0)}\ps{h_{p+1},\partial\varphi_\eps(z_0)}:\V_\eps}_\delta
\end{equation*}
by using Equation~\eqref{eq:remarkable_identity} (with $h_4\coloneqq h_1$). We can apply Proposition~\ref{GaussianIPP} on the first term and Equation~\eqref{eq:IPP_product} for the second one to rewrite the right-hand side under the form
\begin{equation*}
\begin{split}
    &\quad\frac{1}{2}\sum_{k=1}^N\frac{\ps{Q,\alpha_k}}{(z_0-z_k)^2_\eps}\ps{\bm{\mathrm V}_\eps}_\delta -\frac {\mu \gamma q}2\sum_{i=1}^2\int_\C\frac{\ps{2h_1+h_2,e_i}}{(z_0-x)^2_\eps}\ps{V_{\gamma e_i,\eps}(x)\bm{\mathrm V}_\eps}_\delta d^2x\\
    &+\frac{1}{2}\sum_{p=1}^3\sum_{k,l=1}^N\frac{\ps{h_p,\alpha_k}\ps{h_{p+1},\alpha_l}}{(z_0-z_k)_\eps(z_0-z_l)_\eps}\ps{\bm{\mathrm V}_\eps}_\delta\\
    &-\frac {\mu \gamma }2\sum_{p=1}^3\sum_{i=1}^2\int_\C\left(\sum_{k=1}^N\frac{\ps{h_p,\alpha_k}\ps{h_{p+1},e_i}}{(z_0-z_k)_\eps(z_0-x)_\eps}+\gamma\frac{\ps{h_p,e_i}\ps{h_{p+1},e_i}}{(z_0-x)_\eps(z_0-x)_\eps}\right)\theta_\delta(z_0-x)\ps{V_{\gamma e_i,\eps}(x)\bm{\mathrm V}_\eps}_\delta d^2x\\
    &+\frac {(\mu \gamma)^2}2\sum_{p=1}^3\sum_{i,j=1}^2\int_{\C^2}\frac{\ps{h_p,e_i}\ps{h_{p+1},e_j}}{(z_0-x_1)_\eps(z_0-x_2)_\eps}\theta_\delta(z_0-x_1)\theta_\delta(z_0-x_2)\ps{V_{\gamma e_i,\eps}(x_1)V_{\gamma e_j,\eps}(x_2)\bm{\mathrm V}_\eps}_\delta d^2x_1d^2x_2.
\end{split}
\end{equation*}

We now turn to the other side of the expression. Along the same lines as in the proof of Proposition~\ref{regularity} we end up with
\begin{equation*}
\begin{split}
    \sum_{k=1}^N\left(\frac{\Delta_{\alpha_k}}{(z_0-z_k)^2}+\frac{\partial_{z_k}}{(z_0-z_k)}\right)\ps{\bm{\mathrm V}_\eps}_\delta&=\left(\sum_{k=1}^N\frac{\Delta_{\alpha_k}}{(z_0-z_k)^2}-\frac{1}{2}\sum_{k\neq l}\frac{\ps{\alpha_k,\alpha_l}}{(z_0-z_k)(z_k-z_l)_\eps}\right)\ps{\bm{\mathrm V}_\eps}_\delta\\
    &+\frac{\mu\gamma}{2}\sum_{i=1}^2\int_\C\sum_{k=1}^N\frac{\ps{\alpha_k,e_i}\theta_\delta(z_0-x)}{(z_0-z_k)(z_k-x)_\eps}\ps{V_{\gamma e_i,\eps}(x)\bm{\mathrm V}_\eps}_\delta d^2x.
\end{split}
\end{equation*}
Here the symmetrization step is rather straightforward. For the terms that do not involve integrals:
\begin{equation*}
\begin{split}
\frac12\sum_{k\neq l}\frac{\ps{\alpha_k,\alpha_l}}{(z_0-z_k)(z_k-z_l)_\eps}&=\frac14\sum_{k\neq l}\frac{\ps{\alpha_k,\alpha_l}}{(z_k-z_l)_\eps}\left(\frac{1}{z_0-z_k}-\frac{1}{z_0-z_l}\right)\\
&=\frac{1}{4}\sum_{k\neq l}\frac{\ps{\alpha_k,\alpha_l}}{(z_0-z_k)(z_0-z_l)}+\frac14\sum_{k\neq l}\frac{\ps{\alpha_k,\alpha_l}}{(z_0-z_k)(z_0-z_l)}\left(\frac{z_k-z_l}{(z_k-z_l)_\eps}-1\right).
\end{split}
\end{equation*}
We then transform the terms that contain one integral as
\begin{align*}
    &\frac{\mu\gamma}{2}\sum_{i=1}^2\int_\C\sum_{k=1}^N\frac{\ps{\alpha_k,e_i}\theta_\delta(z_0-x)}{(z_0-z_k)(z_k-x)_\eps}\ps{V_{\gamma e_i,\eps}(x)\bm{\mathrm V}_\eps}_\delta d^2x\\
    =&\frac{\mu\gamma}{2}\sum_{i=1}^2\int_\C\sum_{k=1}^N\left(\frac{\theta_\delta(z_0-x)\ps{\alpha_k,e_i}}{(z_0-x)(z_k-x)_\eps}+\frac{\theta_\delta(z_0-x)\ps{\alpha_k,e_i}}{(z_0-x)(z_0-z_k)}\right)\ps{V_{\gamma e_i,\eps}(x)\bm{\mathrm V}_\eps}_\delta d^2x\\
    &+\frac{\mu\gamma}{2}\sum_{i=1}^2\int_\C\sum_{k=1}^N\frac{\theta_\delta(z_0-x)\ps{\alpha_k,e_i}}{(z_0-x)(z_0-z_k)}\left(\frac{z_k-x}{(z_k-x)_\eps}-1\right)\ps{V_{\gamma e_i,\eps}(x)\bm{\mathrm V}_\eps}_\delta d^2x.
\end{align*}
Combining our expressions yields
\begin{align*}
    &\ps{\bm {\mathrm T}_\eps(z_0)\bm{\mathrm V}_\eps}_\delta-\sum_{k=1}^N\left(\frac{\Delta_{\alpha_k}}{(z_0-z_k)^2}+\frac{\partial_{z_k}}{(z_0-z_k)}\right)\ps{\bm{\mathrm V}_\eps}_\delta=\\
    &\frac12\sum_{k=1}^N\frac{1}{(z_0-z_k)^2}\left[\ps{Q,\alpha_k}\left(\frac{(z_0-z_k)^2}{(z_0-z_k)^2_\eps}-1\right)-\frac12\ps{\alpha_k,\alpha_k}\left(\frac{(z_0-z_k)^2}{(z_0-z_k)_\eps(z_0-z_k)_\eps}-1\right)\right]\ps{\bm{\mathrm V}_\eps}_\delta \\
    &-\frac14\sum_{k\neq l}\frac{\ps{\alpha_k,\alpha_l}}{(z_0-z_k)(z_0-z_l)}\left(\frac{(z_0-z_k)(z_0-z_l)}{(z_0-z_k)_\eps(z_0-z_l)_\eps}-1+\frac{z_k-z_l}{(z_k-z_l)_\eps}-1\right)\ps{\bm{\mathrm V}_\eps}_\delta\\
    &-\frac {\mu \gamma }2\sum_{i=1}^2\int_\C\frac{\theta_\delta(z_0-x)}{(z_0-x)^2}\left(\gamma-\frac{\gamma(z_0-x)^2}{(z_0-x)_\eps(z_0-x)_\eps}\right)\ps{V_{\gamma e_i,\eps}(x)\bm{\mathrm V}_\eps}_\delta d^2x\\
    &-\frac {\mu \gamma }2\sum_{i=1}^2\int_\C\frac{\theta_\delta(z_0-x)}{(z_0-x)^2}\left(\sum_{k=1}^N\frac{(z_0-x)\ps{\alpha_k,e_i}}{(z_0-z_k)}\left(\frac{z_k-x}{(z_k-x)_\eps}-1\right)\right)\ps{V_{\gamma e_i,\eps}(x)\bm{\mathrm V}_\eps}_\delta d^2x\\
    &-\mu\sum_{i=1}^2\int_\C\frac{\theta_\delta(z_0-x)}{z_0-x}\sum_{k=1}^N\frac{\ps{\alpha_k,\gamma e_i}}{2(z_k-x)_\eps}-\frac{\theta_\delta(z_0-x)}{(z_0-x)^2_\eps}\ps{V_{\gamma e_i,\eps}(x)\bm{\mathrm V}_\eps}_\delta d^2x\\
    &-\frac {(\mu \gamma)^2}4\sum_{i,j=1}^2\ps{e_i,e_j}\int_{\C^2}\frac{\theta_\delta(z_0-x_1)\theta_\delta(z_0-x_2)}{(z_0-x_1)_\eps(z_0-x_2)_\eps}\ps{V_{\gamma e_i,\eps}(x_1)V_{\gamma e_j,\eps}(x_2)\bm{\mathrm V}_\eps}_\delta d^2x_1d^2x_2.
\end{align*}
where we have used the fact that $\sum_{p=1}^3\ps{h_p,\alpha_k}\ps{h_{p+1},\alpha_l}=-\frac12\ps{\alpha_k,\alpha_l}$.

From Lemma~\ref{approx_inverse}, all lines but the last two vanish when $\eps\to 0$. Now for $x\not\in\lbrace z_0,\cdots,z_N\rbrace$,
\begin{align*}
    &\partial_x\ps{V_{\gamma e_i,\eps}(x)\bm{\mathrm V}_\eps}_\delta=-\frac12\sum_{k=1}^N\frac{\ps{\alpha_k,\gamma e_i}}{(x-z_k)_\eps}\ps{V_{\gamma e_i,\eps}(x)\bm{\mathrm V}_\eps}_\delta+\frac{\mu\gamma^2}{2}\sum_{i=1}^2\int_\C \frac{\ps{e_i,e_j}}{(x-x_2)_\eps}\ps{V_{\gamma e_i,\eps}(x)V_{\gamma e_j,\eps}(x_2)\bm{\mathrm V}_\eps}_\delta d^2x_2.
\end{align*}
Using integration by parts for the term $\frac{\theta_\delta(z_0-x)}{(z_0-x)^2_\eps}\ps{V_{\gamma e_i,\eps}(x)\bm{\mathrm V}_\eps}_\delta$, 
\begin{equation*}
\begin{split}
    &-\mu\sum_{i=1}^2\int_\C\left(\frac{\theta_\delta(z_0-x)}{z_0-x}\sum_{k=1}^N\frac{\ps{\alpha_k,\gamma e_i}}{2(z_k-x)_\eps}-\frac{\theta_\delta(z_0-x)}{(z_0-x)^2_\eps}\right)\ps{V_{\gamma e_i,\eps}(x)\bm{\mathrm V}_\eps}_\delta d^2x\\
    ={}&-\mu\sum_{i=1}^2\int_\C\frac{\partial_x\theta_\delta(z_0-x)}{(z_0-x)_\eps}\ps{V_{\gamma e_i,\eps}(x)\bm{\mathrm V}_\eps}_\delta d^2x\\
    &+\frac {(\mu \gamma)^2}2\sum_{i,j=1}^2\ps{e_i,e_j}\int_{\C^2}\frac{\theta_\delta(z_0-x_1)\theta_\delta(z_0-x_2)}{(z_0-x_1)_\eps(x_1-x_2)_\eps}\ps{V_{\gamma e_i,\eps}(x_1)V_{\gamma e_j,\eps}(x_2)\bm{\mathrm V}_\eps}_\delta d^2x_1d^2x_2.
\end{split}
\end{equation*}
By symmetrisation and using that $\frac{1}{(z_0-x_1)(z_0-x_2)}=\frac{1}{(x_1-x_2)}\left(\frac{1}{z_0-x_1}-\frac1{z_0-x_2}\right)$, we see that the last integral is of order $o(\eps)$. Therefore,
\begin{equation*}
\begin{split}
    \ps{\bm {\mathrm T}_\eps(z_0)\bm{\mathrm V}_\eps}_\delta-\sum_{k=1}^N\left(\frac{\Delta_{\alpha_k}}{(z_0-z_k)^2}+\frac{\partial_{z_k}}{(z_0-z_k)}\right)\ps{\bm{\mathrm V}_\eps}_\delta
    ={}&\mathfrak{R}_\eps -\mu\sum_{i=1}^2\int_\C\frac{\partial_x\theta_\delta(z_0-x)}{(z_0-x)_\eps}\ps{V_{\gamma e_i,\eps}(x)\bm{\mathrm V}_\eps}_\delta d^2x
\end{split}
\end{equation*}
where $\mathfrak{R}_\eps$ vanishes when $\eps\rightarrow0$.  When $\eps$ goes to $0$, the only term that is left is given by
\begin{equation*}
-\mu\sum_{i=1}^2\int_\C\frac{\partial_x\theta_\delta(z_0-x)}{(z_0-x)}\ps{V_{\gamma e_i}(x)\bm{\mathrm V}}_\delta d^2x.
\end{equation*}
In order to prove the spin-two Ward identity~\eqref{Ward_stress} it only remains to make sure that the $\delta\to 0$ limit of the above expression is zero. We can make a change of variable $x\to z_0-\delta x$ to get
\begin{equation*}
-\mu\sum_{i=1}^2\int_{A(1/2,1)}\frac{\partial_x\theta_1(x)}{x}\ps{V_{\gamma e_i}(\delta x+z_0)\bm{\mathrm V}}_\delta d^2x
\end{equation*}
with $A(1/2,1)$ the annulus $A(1/2,1)$ of radii $1/2,1$ centered at the origin.
When $\delta\rightarrow0$, by continuity of the correlation functions and dominated convergence this converges to 
\begin{equation*}
-\mu\sum_{i=1}^2\ps{V_{\gamma e_i,\eps}(z_0)\bm{\mathrm V}_\eps}_{\delta=0}\int_{A(1/2,1)}\frac{\partial_x\theta_1(x)}{x}d^2x.
\end{equation*}
Since $\theta_1$ is rotation invariant, the latter integral is zero (one can use polar coordinates). This shows that
\begin{equation*}
\lim\limits_{\delta\rightarrow0}\lim\limits_{\eps\rightarrow0}\left(\ps{\bm {\mathrm T}_\eps(z_0)\bm{\mathrm V}_\eps}_\delta-\sum_{k=1}^N\left(\frac{\Delta_{\alpha_k}}{(z_0-z_k)^2}+\frac{\partial_{z_k}}{(z_0-z_k)}\right)\ps{\bm{\mathrm V}_\eps}_\delta\right)=0.
\end{equation*}
Since the correlation functions are differentiable in $z_0$ and all $z_k$, according to Proposition~\ref{regularity}, we already know that the $\eps,\delta\rightarrow0$ limit of $\sum\limits_{k=1}^N\left(\frac{\Delta_{\alpha_k}}{(z_0-z_k)^2}+\frac{\partial_{z_k}}{(z_0-z_k)}\right)\ps{\bm{\mathrm V}_\eps}_\delta$ exists. This entails the existence of $\lim\limits_{\delta\to 0}\lim\limits_{\eps\to 0}\ps{\bm {\mathrm T}_\eps(z_0)\bm{\mathrm V}_\eps}_\delta$, concluding the proof of the spin-two Ward identity.

\subsection{Computation of the spin-three Ward identity}
We proceed in a similar way for the spin-three Ward identity. We start by considering the regularized version of the probabilistic objects
\begin{multline*}
    \ps{\bm {\mathrm W}_\eps(z_0)\bm{\mathrm V}_\eps}_\delta=-\frac{q^2}{8}\ps{\ps{h_2,\partial^3\varphi_\eps(z_0)}\V_\eps}_\delta+\frac{q}{4}\ps{:\ps{h_2-h_1,\partial^2\varphi_\eps(z_0)}\ps{h_1,\partial\varphi_\eps(z_0)}:\V_\eps}_\delta\\
    +\frac q4\ps{:\ps{h_3-h_2,\partial^2\varphi_\eps(z_0)}\ps{h_3,\partial\varphi_\eps(z_0)}:\V_\eps}_\delta+\ps{:\ps{h_1,\partial\varphi_\eps(z_0}\ps{h_2,\partial\varphi_\eps(z_0)}\ps{h_3,\partial\varphi_\eps(z_0)}:\V_\eps}_{\delta}.
\end{multline*}
Applying Proposition~\ref{GaussianIPP} and more generally Equation~\eqref{eq:IPP_product} we see that this quantity is given by
\[
    \ps{\bm {\mathrm W}_\eps(z_0)\bm{\mathrm V}_\eps}_\delta=-\frac{1}{8}\left( q^2\mathrm{I}_{1}+\frac q2\mathrm{I}_{2}+\mathrm{I}_3\right) ,
\]
where the $(\mathrm I_i)_{i=1,2,3}$ correspond to the (rather lengthy) expressions:
\begin{align*}
    &\mathrm I_1=-\sum_{k=1}^N \frac{\ps{h_2,\alpha_k}}{(z_0-z_k)_\eps^3}\ps{\V_\eps}_\delta
    +\mu\gamma\sum_{i=1}^{2}\ps{h_2,e_i}\int_\C \frac{\theta_\delta(z_0-x)}{(z_0-x)^3_\eps}\ps{V_{\gamma e_i,\eps}(x)\bm{\mathrm V}_\eps}_\delta d^2x,\\
    &\mathrm{I}_2=\sum_{k,l=1}^N\frac{B(\alpha_k,\alpha_l)}{(z_0-z_k)_\eps^2(z_0-z_l)_\eps}\ps{\bm{\mathrm V}_\eps}_\delta\\
    &\hspace*{-1cm}-\mu\gamma \sum_{i=1}^2\int_\C \left(\sum_{k=1}^N\frac{B(\alpha_k,e_i)}{(z_0-z_k)_\eps^2(z_0-x)_\eps}+\frac{B(e_i,\alpha_k)}{(z_0-x)^2_\eps(z_0-z_k)_\eps}+ \frac{\gamma B(e_i,e_i)}{(z_0-x)^2_\eps(z_0-x)_\eps}\right)\theta_\delta(z_0-x)\ps{V_{\gamma e_i,\eps}(x)\bm{\mathrm V}_\eps}_\delta d^2x\\
    &+(\mu\gamma)^2\sum_{i,j=1}^2\int_{\C^2} B(e_i,e_j)\frac{\theta_\delta(z_0-x_1)}{(z_0-x_1)^2_\eps}\frac{\theta_\delta(z_0-x_2)}{(z_0-x_2)_\eps}\ps{V_{\gamma e_i,\eps}(x_1)V_{\gamma e_j,\eps}(x_2)\bm{\mathrm V}_\eps}_\delta d^2x_1d^2x_2\quad\text{and}
\end{align*}
\begin{align*}
    &\mathrm{I}_3=\sum_{k,l,p=1}^N \frac{\ps{h_1,\alpha_k}\ps{h_2,\alpha_l}\ps{h_3,\alpha_p}}{(z_0-z_k)_\eps(z_0-z_l)_\eps(z_0-z_p)_\eps}\ps{\bm{\mathrm V}_\eps}_\delta\\
    &\hspace*{-1.5cm}-\mu\gamma\sum_{i=1}^2\int_\C \left(\sum_{k,l=1}^N\frac{C(\alpha_k,\alpha_l,e_i)}{(z_0-z_k)_\eps(z_0-z_l)_\eps(z_0-x)_\eps}+\sum_{k=1}^N\frac{\gamma C(\alpha_k,e_i,e_i)}{(z_0-z_k)_\eps(z_0-x)_\eps(z_0-x)_\eps}\right)\theta_\delta(z_0-x)\ps{V_{\gamma e_i,\eps}(x)\bm{\mathrm V}_\eps}_\delta d^2x\\
    &+(\mu\gamma)^2\sum_{i,j=1}^2\int_{\C^2} \left(\sum_{k=1}^N\frac{C(\alpha_k,e_i,e_j)}{(z_0-z_k)_\eps(z_0-x_1)_\eps(z_0-x_2)_\eps}+\frac{\gamma C(e_i,e_i,e_j)}{(z_0-x_1)_\eps(z_0-x_1)_\eps(z_0-x_2)_\eps}\right)\\
    &\quad\quad\quad\quad\quad\quad\quad\theta_\delta(z_0-x_1)\theta_\delta(z_0-x_2)\ps{V_{\gamma e_i,\eps}(x_1)V_{\gamma e_j,\eps}(x_2)\bm{\mathrm V}_\eps}_\delta d^2x_1d^2x_2\\
    &\hspace*{-1cm}-(\mu\gamma)^3\sum_{i,j,f=1}^2\int_{\C^3} C(e_i,e_j,e_f)\frac{\theta_\delta(z_0-x_1)\theta_\delta(z_0-x_2)\theta_\delta(z_0-x_3)}{3(z_0-x_1)_\eps(z_0-x_2)_\eps(z_0-x_3)_\eps}\ps{V_{\gamma e_i,\eps}(x_1)V_{\gamma e_j,\eps}(x_2)V_{\gamma e_f,\eps}(x_3)\bm{\mathrm V}_\eps}_\delta d^2x_1d^2x_2d^2x_3.
\end{align*}
Using the properties of Lemma~\ref{integrability} for the regularized correlation functions we see that the $\eps\rightarrow0$ limit of the previous expression exists, and taking this limit is tantamount to erasing all the $\eps$. 

We proceed in the same way with the right-hand side in~\eqref{Ward_higher}. Recall that
\begin{align*}
    &\bm {\mathrm W_{-1,\eps}}(z_k,\alpha_k)\coloneqq -\frac qB(\alpha_k,\partial\varphi_\eps(z_k))
    -2C(\alpha_k,\alpha_k,\partial\varphi_\eps(z_k))\quad\text{and}\\
    &\bm {\mathrm W_{-2,\eps}}(z_k,\alpha_k)\coloneqq q\left(B(\partial^2\varphi_\eps( z_k),\alpha_k)-B(\alpha_k,\partial^2\varphi_\eps(z_k))\right)-2C(\alpha_k,\alpha_k,\partial^2\varphi_\eps(z_k)) +4:C(\alpha_k,\partial\varphi_\eps(z_k),\partial\varphi_\eps(z_k)):.
\end{align*}
As a consequence we see that this right-hand side is given by
\[
	-\frac18\sum_{k=1}^N\left(\frac{w^{(3)}(\alpha_k)}{(z_0-z_k)^3}
    +\frac{\mathrm J_{-1}^{(k)}}{(z_0-z_k)^2}+\frac{\mathrm J_{-2}^{(k)}}{z_0-z_k}\right)
\]
where the notations $\mathrm J_{-i}^{(k)}$ stand for
\begin{align*}
	\mathrm J_{-1}^{(k)}&=\frac{1}{2}\sum_{l\neq k}\frac{q B(\alpha_k,\alpha_l)+2C(\alpha_k,\alpha_k,\alpha_l)}{(z_k-z_l)_\eps}\ps{\V_\eps}_\delta\\
    &\quad -\frac{\mu}{2}\sum_{i=1}^2\int_\C\frac{q B(\alpha_k,\gamma e_i)+2C(\alpha_k,\alpha_k,\gamma e_i)}{(z_k-x)_\eps}\theta_\delta(z_0-x)\ps{V_{\gamma e_i,\eps}(z)\V_\eps}_\delta d^2x\quad\text{and}\\
    	\mathrm J_{-2}^{(k)}&=\frac{1}{2}\sum_{l\neq k}\frac{q\left(B(\alpha_l,\alpha_k)-B(\alpha_k,\alpha_l)\right)+2C(\alpha_k,\alpha_l-\alpha_k,\alpha_l)}{(z_k-z_l)_\eps^2}\ps{\V_\eps}_\delta\\
    &\quad +\sum_{k,l,p\text{ distinct}}\frac{C(\alpha_k,\alpha_l,\alpha_p)}{(z_k-z_l)_\eps(z_k-z_p)_\eps}\ps{\V_\eps}_\delta\\
	&\quad -\frac{\mu\gamma}{2}\sum_{i=1}^2\int_\C\left(\frac{q\left(B(e_i,\alpha_k)-B(\alpha_k, e_i)\right)+2C(\alpha_k,\gamma e_i-\alpha_k,e_i)}{(z_k-x)_\eps^2}\right)\theta_\delta(z_0-x)\ps{V_{\gamma e_i,\eps}(z)\V_\eps}_\delta d^2x\\
    &\quad -\frac{\mu\gamma}{2}\sum_{i=1}^2\int_\C \left(\sum_{l\neq k}\frac{2C^{\sigma}(\alpha_k,\alpha_l,e_i)}{(z_k-z_l)_\eps(z_k-x)_\eps}\right)\theta_\delta(z_0-x)\ps{V_{\gamma e_i,\eps}(z)\V_\eps}_\delta d^2x\\
    &\quad +(\mu\gamma)^2\sum_{i,j=1}^2\int_{\C^2}\frac{C(\alpha_k,e_i,e_j)}{(z_k-x_1)_\eps(z_k-x_2)_\eps}\theta_\delta(z_0-x_1)\theta_\delta(z_0-x_2)\ps{V_{\gamma e_i,\eps}(x_1)V_{\gamma e_j,\eps}(x_2)\V_\eps}_\delta d^2x_1d^2x_2
\end{align*}
with $C^{\sigma}(\alpha_k,\alpha_l,e_i)\coloneqq C(\alpha_k,\alpha_l,e_i)+C(\alpha_k,e_i,\alpha_l)$.
\subsubsection{First step: symmetrizing}
We now turn to the first step of the proof which consists in using symmetrization identities so that in the expression of the terms $\mathrm{J}_{-1,2}^{(k)}$ the only singularities that will occur will do so around the distinguished point $z_0$. To illustrate this, let us consider the first term $\mathrm{J}_{-1}^{(k)}$. Using the \lq\lq symmetrization identity" 
\[
\frac{1}{(z-y)^2(y-x)}=\frac{1}{(z-y)^2(z-x)}+\frac{1}{(z-y)(z-x)^2}+\frac{1}{(z-x)^2(y-x)}
\]
we can write that $\frac{\mathrm{J}_{-1}^{(k)}}{(z_0-z_k)^2}$ is given by
\begin{align*}
	    &\quad\frac1{2(z_0-z_k)^2}\sum_{l\neq k}\frac{q B(\alpha_k,\alpha_l)+2C(\alpha_k,\alpha_k,\alpha_l)}{(z_0-z_l)}\ps{\V_\eps}_\delta\\
        &+\frac1{2(z_0-z_k)}\sum_{l\neq k}\frac{q B(\alpha_k,\alpha_l)+2C(\alpha_k,\alpha_k,\alpha_l)}{(z_k-z_l)(z_k-z_l)_{\eps}}\ps{\V_\eps}_\delta \\
	    &- \frac12\sum_{l\neq k}\frac{q B(\alpha_k,\alpha_l)+2C(\alpha_k,\alpha_k,\alpha_l)}{(z_0-z_l)(z_k-z_l)(z_k-z_l)_{\eps}}\ps{\V_\eps}_\delta \\
	    &-\frac{\mu}{2(z_0-z_k)^2}\sum_{i=1}^2\int_\C\frac{q B(\alpha_k,\gamma e_i)+2C(\alpha_k,\alpha_k,\gamma e_i)}{(z_0-x)}\theta_{\delta}(z_0-x)\ps{V_{\gamma e_i,\eps}(z)\V_\eps}_\delta d^2x\\
	    &-\frac{\mu}{2(z_0-z_k)}\sum_{i=1}^2\int_\C\frac{q B(\alpha_k,\gamma e_i)+2C(\alpha_k,\alpha_k,\gamma e_i)}{(z_0-x)^2}\theta_{\delta}(z_0-x)\ps{V_{\gamma e_i,\eps}(z)\V_\eps}_\delta d^2x\\
	    &-\frac{\mu}{2}\sum_{i=1}^2\int_\C\frac{q B(\alpha_k,\gamma e_i)+2C(\alpha_k,\alpha_k,\gamma e_i)}{(z_0-x)^2(z_k-x)_\epsilon}\theta_{\delta}(z_0-x)\ps{V_{\gamma e_i,\eps}(z)\V_\eps}_\delta d^2x+\mathfrak{R}_{-1}(\eps,\delta)
\end{align*}
with
\begin{align*}
&\mathfrak{R}_{-1}(\eps,\delta)\coloneqq \sum_{l\neq k}\frac{q B(\alpha_k,\alpha_l)+2C(\alpha_k,\alpha_k,\alpha_l)}{2(z_0-z_k)^2}\left(\frac{1}{(z_k-z_l)_\eps}-\frac{1}{z_k-z_l}\right)\ps{\V_\eps}_\delta\\
&\hspace*{-0.5cm}-\frac{\mu}{2}\sum_{i=1}^2\int_\C\frac{q B(\alpha_k,\gamma e_i)+2C(\alpha_k,\alpha_k,\gamma e_i)}{(z_0-z_k)(z_0-x)}\left(\frac{1}{z_0-x}+\frac{1}{z_0-z_k}\right)\left(\frac{
z_k-x}{(z_k-x)_\eps}-1\right)\theta_{\delta}(z_0-x)\ps{V_{\gamma e_i,\eps}(z)\V_\eps}_\delta d^2x.
\end{align*}
According to Lemma~\ref{approx_inverse}, we readily see that 
\begin{align*}
    \lim\limits_{\eps\rightarrow0}\mathfrak{R}_{-1}(\eps,\delta)=0.
\end{align*}
In the new expression provided for $\mathrm{J}_{-1}^{(k)}$ the last integral is the only one that features a singular term which is not localized at $z_0$ and corresponds to $\frac{1}{(z_k-x)_\eps}$. To remove this singularity we will see that when combined with the term $\frac{\mathrm{J}_{-2}^{(k)}}{(z_0-z_k)}$ this quantity is nothing but a derivative of the correlation function. Therefore this term will be treated in the second part of the proof, that is will be removed by using integration by parts.

Let us now turn to $\mathrm{J}_{-2}^{(k)}$: its study is more involved but nonetheless the very same method still works. In order to keep things simple we explain how to treat the term involving a two-fold integral. The symmetrization method yields
\begin{align*}
    &\frac{1}{z_0-z_k}\mu^2\sum_{i,j=1}^2\int_{\C^2}\frac{C(\alpha_k,e_i,e_j)}{(z_k-x_1)_\eps(z_k-x_2)_\eps}\ps{V_{\gamma e_i,\eps}(x_1)V_{\gamma e_j,\eps}(x_2)\V_\eps}_\delta d^2x_1d^2x_2=\\
    &\mu^2\sum_{i,j=1}^2\int_{\C^2}\frac{C(\alpha_k,e_i,e_j)}{(z_0-z_k)(z_0-x_1)_\eps(z_0-x_2)_\eps}\ps{V_{\gamma e_i,\eps}(x_1)V_{\gamma e_j,\eps}(x_2)\V_\eps}_\delta d^2x_1d^2x_2\\
    &\hspace*{-0.5cm}+\mu^2\sum_{i,j=1}^2\int_{\C^2}\frac{C(\alpha_k,e_i,e_j)+C(\alpha_k,e_j,e_i)}{(z_0-x_1)(x_1-x_2)_\eps(z_k-x_1)_\eps}\theta_\delta(z_0-x_1)\theta_\delta(z_0-x_2)\ps{V_{\gamma e_i,\eps}(x_1)V_{\gamma e_j,\eps}(x_2)\V_\eps}_\delta d^2x_1d^2x_2\\
    &+\mathfrak{R}_{-2}(\eps,\delta),
\end{align*}
where $\mathfrak{R}_{-2}(\eps,\delta)$ is defined by
\begin{align*}
	&\hspace*{-1.25cm}\mu^2\sum_{i,j=1}^2\int_{\C^2}\frac{C(\alpha_k,e_i,e_j)}{(z_0-z_k)(z_0-x_1)(z_0-x_2)}\left(\frac{x_1-x_2}{(x_1-x_2)_\eps}-1\right)\theta_\delta(z_0-x_1)\theta_\delta(z_0-x_2)\ps{V_{\gamma e_i,\eps}(x_1)V_{\gamma e_j,\eps}(x_2)\V_\eps}_\delta d^2x_1d^2x_2\\
	&+\mu^2\sum_{i,j=1}^2\int_{\C^2}\frac{C(\alpha_k,e_i,e_j)}{(z_0-z_k)}\left(\frac{1}{(z_k-x_1)_\eps(z_k-x_2)_\eps}\right.\\
	&\hspace*{1cm}\left.-\frac1{(x_1-x_2)_\eps}\left(\frac{1}{(z_k-x_1)_\eps}-\frac{1}{(z_k-x_2)_\eps}\right)\right)\theta_\delta(z_0-x_1)\theta_\delta(z_0-x_2)\ps{V_{\gamma e_i,\eps}(x_1)V_{\gamma e_j,\eps}(x_2)\V_\eps}_\delta d^2x_1d^2x_2\\
	&+\mu^2\sum_{i,j=1}^2\int_{\C^2}\frac{C(\alpha_k,e_i,e_j)+C(\alpha_k,e_j,e_i)}{(z_0-z_k)(z_0-x_1)(x_1-x_2)_\eps}\theta_\delta(z_0-x_1)\theta_\delta(z_0-x_2)\ps{V_{\gamma e_i,\eps}(x_1)V_{\gamma e_j,\eps}(x_2)\V_\eps}_\delta d^2x_1d^2x_2.
\end{align*}
Similarly it will be convenient in the sequel to note that, using the symmetries between the integration variables and since the quantities $C(e_i,e_j,e_f)$ are explicit for $i,j,f\in\{1,2\}$, elementary but lengthy algebraic computations allow to write
\begin{align*}
&\hspace*{-1cm}(\mu\gamma)^3\sum_{i,j,f=1}^2\int_{\C^3} C(e_i,e_j,e_f) \frac{\theta_\delta(z_0-x_1)\theta_\delta(z_0-x_2)\theta_\delta(z_0-x_3)}{3(z_0-x_1)(z_0-x_2)(z_0-x_3)}\ps{V_{\gamma e_i,\eps}(x_1)V_{\gamma e_j,\eps}(x_2)V_{\gamma e_f,\eps}(x_3)\bm{\mathrm V}_\eps}_\delta d^2x_1d^2x_2d^2x_3\\
&\hspace*{-1cm}=-(\mu\gamma)^3\sum_{i\neq j}\ps{h_2,e_i}\int_{\C^3} \frac{\theta_\delta(z_0-x_1)\theta_\delta(z_0-x_2)\theta_\delta(z_0-x_3)}{(z_0-x_1)(x_1-x_2)_\eps(x_1-x_3)_\eps}\ps{V_{\gamma e_i,\eps}(x_1)V_{\gamma e_i,\eps}(x_2)V_{\gamma e_j,\eps}(x_3)\bm{\mathrm V}_\eps}_\delta d^2x_1d^2x_2d^2x_3\\
    &+\mathfrak{R}_{-3}(\eps,\delta)
\end{align*}
where $\mathfrak{R}_{-3}(\eps,\delta)$ is similar to $\mathfrak{R}_{-2}(\eps,\delta)$.
Collecting all the terms and using the identity based on the explicit expression for $B$ and $C$ 
\[
	q\left(B(e_i,\alpha_k)-B(\alpha_k, e_i)\right)+2C(\alpha_k,\gamma e_i-\alpha_k,e_i)=\frac4\gamma\ps{h_2,e_i}\omega_{3-i}(\alpha_k)\left(1+\frac12\ps{\alpha_k,\gamma e_i}\right)
\] (where $\omega_{3-i}(\alpha_k)$ is a shorthand for $\ps{\omega_{3-i},\alpha_k}$) we can write that
\begin{align*}
    &q^2\mathrm{I}_{1}+\frac q2\mathrm{I}_{2}+\mathrm{I}_3-\sum_{k=1}^N\left(\frac{w^{(3)}(\alpha_k)}{(z_0-z_k)^3}
    +\frac{\mathrm J_{-1}^{(k)}}{(z_0-z_k)^2}+\frac{\mathrm J_{-2}^{(k)}}{z_0-z_k}\right)=\\
    &+2\mu q\sum_{i=1}^{2}\ps{h_2,e_i}\int_\C \frac{\theta_\delta(z_0-x)}{(z_0-x)^3}\ps{V_{\gamma e_i,\eps}(x)\bm{\mathrm V}_\eps}_\delta d^2x\\
    &+\frac{\mu\gamma}{2}\sum_{i=1}^2\int_\C\sum_{k=1}^N\frac{q B(e_i,\alpha_k)+2C(\alpha_k,\gamma e_i,e_i)}{(z_0-x)^2(z_k-x)_\eps}\theta_{\delta}(z_0-x)\ps{V_{\gamma e_i,\eps}(z)\V_\eps}_\delta d^2x\\
    &+2\mu\sum_{i=1}^2\ps{h_2,e_i}\int_\C\sum_{k=1}^N\frac{\theta_{\delta}(z_0-x)\omega_{3-i}(\alpha_k)}{(z_0-x)(z_k-x)_\eps}\left(\frac{1+\frac12\ps{\alpha_k,\gamma e_i}}{(z_k-x)_\eps}+\sum_{l\neq k}\frac{\frac12\ps{\alpha_l, \gamma e_i}}{(z_l-x)_\eps}\right)\ps{V_{\gamma e_i,\eps}(z)\V_\eps}_\delta d^2x\\
    &+\frac{(\mu\gamma)^2}2\sum_{i,j=1}^2\int_{\C^2} \left(qB(e_i,e_j)+2\gamma C(e_i,e_i,e_j)\right)\frac{\theta_\delta(z_0-x_1)}{(z_0-x_1)^2}\frac{\theta_\delta(z_0-x_2)}{(x_1-x_2)_\eps}\ps{V_{\gamma e_i,\eps}(x_1)V_{\gamma e_j,\eps}(x_2)\bm{\mathrm V}_\eps}_\delta d^2x_1d^2x_2\\
    &-2\mu^2\gamma\sum_{i\neq j}^2\ps{h_2,e_i}\int_{\C^2} \frac{\theta_\delta(z_0-x_1)}{(z_0-x_1)}\frac{\theta_\delta(z_0-x_2)}{(x_1-x_2)_\eps^2}\ps{V_{\gamma e_i,\eps}(x_1)V_{\gamma e_j,\eps}(x_2)\bm{\mathrm V}_\eps}_\delta d^2x_1d^2x_2\\
    &-(\mu\gamma)^2\sum_{i,j=1}^2\int_{\C^2}\sum_{k=1}^N \frac{\theta_\delta(z_0-x_1)\theta_\delta(z_0-x_2)}{(z_0-x_1)(x_1-x_2)_\eps}\frac{C^{\sigma}(\alpha_k,e_i,e_j)}{(z_k-x_1)_\eps}\ps{V_{\gamma e_i,\eps}(x_1)V_{\gamma e_j,\eps}(x_2)\V_\eps}_\delta d^2x_1d^2x_2\\
    &\hspace*{-0cm}+(\mu\gamma)^3\sum_{i\neq j}\ps{h_2,e_i}\int_{\C^3} \frac{\theta_\delta(z_0-x_1)\theta_\delta(z_0-x_2)\theta_\delta(z_0-x_3)}{(z_0-x_1)(x_1-x_2)_\eps(x_2-x_3)_\eps}\ps{V_{\gamma e_i,\eps}(x_1)V_{\gamma e_i,\eps}(x_2)V_{\gamma e_j,\eps}(x_3)\bm{\mathrm V}_\eps}_\delta d^2x_1d^2x_2d^2x_3\\
    &+\mathfrak{R}(\eps,\delta)
\end{align*}
where $\mathfrak{R}(\eps,\delta)$ is a combination of $\mathfrak{R}_{-1,-2,-3}(\eps,\delta)$. At first we are concerned with the $\eps\rightarrow0$ limit of this remainder. There are two terms for which it is not obvious that they converge to zero (here $F$ is some bounded function):
\begin{align*}
    \int_{\C}\frac{F(x)}{(x-z_k)_\eps}\left(\frac{x-z_k}{(x-z_k)_\eps}-1\right)\ps{V_{\gamma e_i,\eps}(x)\V_\eps}_\delta d^2x,\quad\text{and}
\end{align*}
\begin{align*}
	&\mu^2\sum_{i,j=1}^2\int_{\C^2}\frac{C(\alpha_k,e_i,e_j)}{(z_0-z_k)}\left(\frac{1}{(z_k-x_1)_\eps(z_k-x_2)_\eps}-\frac1{(x_1-x_2)_\eps}\left(\frac{1}{(z_k-x_1)_\eps}-\frac{1}{(z_k-x_2)_\eps}\right)\right)\\
    &\quad\quad\quad\quad\quad\quad\theta_\delta(z_0-x_1)\theta_\delta(z_0-x_2)\ps{V_{\gamma e_i,\eps}(x_1)V_{\gamma e_j,\eps}(x_2)\V_\eps}_\delta d^2x_1d^2x_2.\\
\end{align*}
However making the change of variables $x\leftrightarrow4R\eps x$ (where recall $R$ from the regularization scheme~\eqref{regularization}) and using the same reasoning as in the proof of Lemma~\ref{approx_inverse} we see that actually such quantities will vanish in the $\eps\rightarrow0$ limit. Therefore $\lim\limits_{\eps\rightarrow0}\mathfrak{R}(\eps,\delta)=0$ so we may no longer consider it in subsequent computations.

\subsubsection{Second step: integrating by parts}We are now in position to address the second step of our proof, that is to use integration by parts in order to remove the singular terms that occur away from $z_0$. Indeed we can notice that
\begin{multline*}
    \partial_{x}\left(\frac{1}{(z_x-x)_\eps}\ps{V_{\gamma e_i,\eps}(x)\V_\eps}_\delta\right)=\left(\frac{1+\frac12\ps{\alpha_k,\gamma e_i}}{(z_k-x)^2_\eps}+\sum_{l\neq k}\frac{\frac12\ps{\alpha_l, \gamma e_i}}{(z_k-x)_\eps(z_l-x)_\eps}\right)\ps{V_{\gamma e_i,\eps}(x)\V_\eps}_\delta\\
    +\frac{\mu\gamma^2}{2}\sum_{j=1}^2\int_\C\frac{\ps{e_i,e_j}\theta_\delta(z_0-x_2)}{(z_k-x)_\eps(x-x_2)_\eps}\ps{V_{\gamma e_i,\eps}(x)V_{\gamma e_j,\eps}(x_2)\V_\eps}_\delta d^2x_2.
\end{multline*}
This implies, using integration by parts (thanks to Lemma~\ref{integrability} this is indeed possible), that
\begin{align*}
&2\mu\sum_{i=1}^2\ps{h_2,e_i}\int_\C\sum_{k=1}^N\frac{\theta_{\delta}(z_0-x)\omega_{3-i}(\alpha_k)}{(z_0-x)(z_k-x)_\eps}\left(\frac{1+\frac12\ps{\alpha_k,\gamma e_i}}{(z_k-x)_\eps}+\sum_{l\neq k}\frac{\frac12\ps{\alpha_l, \gamma e_i}}{(z_l-x)_\eps}\right)\ps{V_{\gamma e_i,\eps}(x)\V_\eps}_\delta d^2x\\
&=-2\mu\sum_{i=1}^2\ps{h_2,e_i}\int_\C\sum_{k=1}^N\partial_x\left(\frac{\theta_{\delta}(z_0-x)}{(z_0-x)}\right)\frac{\omega_{3-i}(\alpha_k)}{(z_k-x)_\eps}\ps{V_{\gamma e_i,\eps}(x)\V_\eps}_\delta d^2x\\
&-(\mu\gamma)^2\sum_{i,j=1}^2\ps{h_2,e_i}\ps{e_i,e_j}\int_{\C^2}\sum_{k=1}^N\frac{\theta_{\delta}(z_0-x_1)\theta_{\delta}(z_0-x_2)\omega_{3-i}(\alpha_k)}{(z_0-x_1)(z_k-x_1)_\eps(x_1-x_2)_\eps}\ps{V_{\gamma e_i,\eps}(x_1)V_{\gamma e_j,\eps}(x_2)\V_\eps}_\delta d^2x_1d^2x_2.
\end{align*}
Since $B(e_i,\alpha_k)+2C(\alpha_k,e_i,e_i)=2\ps{h_2,e_i}\omega_{i}(\alpha_k)+\ps{h_2,e_i}\omega_{3-i}(\alpha_k)=\ps{h_2,e_i}\ps{e_i,\alpha_k}$, we can use again integration by parts to get that 
\begin{align*}
    &\frac{\mu\gamma}{2}\sum_{i=1}^2\int_\C\sum_{k=1}^N\frac{q B(e_i,\alpha_k)+2C(\alpha_k,\gamma e_i,e_i)}{(z_0-x)^2(z_k-x)_\eps}\theta_{\delta}(z_0-x)\ps{V_{\gamma e_i,\eps}(x)\V_\eps}_\delta d^2x\\
    &+2\mu\sum_{i=1}^2\ps{h_2,e_i}\int_\C\sum_{k=1}^N\frac{\theta_{\delta}(z_0-x)\omega_{3-i}(\alpha_k)}{(z_0-x)(z_k-x)_\eps}\left(\frac{1+\frac12\ps{\alpha_k,\gamma e_i}}{(z_k-x)_\eps}+\sum_{l\neq k}\frac{\frac12\ps{\alpha_l, \gamma e_i}}{(z_l-x)_\eps}\right)\ps{V_{\gamma e_i,\eps}(x)\V_\eps}_\delta d^2x\\
\end{align*}
is actually equal to 
\begin{align*}
&-\mu q\sum_{i=1}^2\ps{h_2,e_i}\int_\C\partial_x\left(\frac{\theta_{\delta}(z_0-x)}{(z_0-x)^2}\right)\ps{V_{\gamma e_i,\eps}(x)\V_\eps}_\delta d^2x\\
&-2\mu\sum_{i=1}^2\ps{h_2,e_i}\int_\C\sum_{k=1}^N\frac{\partial_x\theta_{\delta}(z_0-x)}{(z_0-x)}\frac{\ps{\alpha_k,\omega_{3-i}}}{(z_k-x)_\eps}\ps{V_{\gamma e_i,\eps}(x)\V_\eps}_\delta d^2x\\
&-\frac{(\mu\gamma)^2q}{2} \sum_{i,j=1}^2\ps{h_2,e_i}\ps{e_i,e_j}\int_{\C^2}\frac{\theta_{\delta}(z_0-x_1)\theta_{\delta}(z_0-x_2)}{(z_0-x)^2(x_1-x_2)_\eps}\ps{V_{\gamma e_i,\eps}(x)\V_\eps}_\delta d^2x\\
&-(\mu\gamma)^2\sum_{i,j=1}^2\ps{h_2,e_i}\ps{e_i,e_j}\int_{\C^2}\sum_{k=1}^N\frac{\theta_{\delta}(z_0-x_1)\theta_{\delta}(z_0-x_2)\omega_{3-i}(\alpha_k)}{(z_0-x_1)(z_k-x_1)_\eps(x_1-x_2)_\eps}\ps{V_{\gamma e_i,\eps}(x_1)V_{\gamma e_j,\eps}(x_2)\V_\eps}_\delta d^2x_1d^2x_2.
\end{align*}
As a consequence we see that in the expression of 
\[q^2\mathrm{I}_{1}+\frac q2\mathrm{I}_{2}+\mathrm{I}_3-\sum_{k=1}^N\left(\frac{w^{(3)}(\alpha_k)}{(z_0-z_k)^3}+\frac{\mathrm J_{-1}^{(k)}}{(z_0-z_k)^2}+\frac{\mathrm J_{-2}^{(k)}}{z_0-z_k}\right)
\]
the only remaining terms are the $1$-fold integral given by
\begin{equation*}
\mathfrak{R}_1(\eps,\delta)\coloneqq -\mu \sum_{i=1}^2\ps{h_2,e_i}\int_\C\left( \frac{q\partial_x\theta_{\delta}(z_0-x)}{(z_0-x)^2}+ 2 \sum_{k=1}^N\frac{\partial_x\theta_{\delta}(z_0-x)}{(z_0-x)}\frac{\ps{\alpha_k,\omega_{3-i}}}{(z_k-x)_\eps}\right)\ps{V_{\gamma e_i,\eps}(z)\V_\eps}_\delta d^2x
\end{equation*}
as well as, using the definitions of $B$ and $C$, the $2$ and $3$-fold integrals
\begin{align*}
    &+2\mu^2\gamma\sum_{i\neq j}\ps{h_2,e_i}\int_{\C^2} \frac{\theta_\delta(z_0-x_1)}{(z_0-x_1)^2}\frac{\theta_\delta(z_0-x_2)}{(x_1-x_2)_\eps}\ps{V_{\gamma e_i,\eps}(x_1)V_{\gamma e_j,\eps}(x_2)\bm{\mathrm V}_\eps}_\delta d^2x_1d^2x_2\\
    &-2\mu^2\gamma\sum_{i\neq j}\ps{h_2,e_i}\int_{\C^2} \frac{\theta_\delta(z_0-x_1)}{(z_0-x_1)}\frac{\theta_\delta(z_0-x_2)}{(x_1-x_2)_\eps^2}\ps{V_{\gamma e_i,\eps}(x_1)V_{\gamma e_j,\eps}(x_2)\bm{\mathrm V}_\eps}_\delta d^2x_1d^2x_2\\
    &+(\mu\gamma)^2\sum_{i\neq j}\ps{h_2,e_i}\int_{\C^2}\frac{\theta_\delta(z_0-x_1)\theta_\delta(z_0-x_2)}{(z_0-x_1)(x_1-x_2)_\eps}\sum_{k=1}^N \frac{\ps{\alpha_k,e_i}}{(z_k-x_1)_\eps}\ps{V_{\gamma e_i,\eps}(x_1)V_{\gamma e_j,\eps}(x_2)\V_\eps}_\delta d^2x_1d^2x_2\\
	&+(\mu\gamma)^3\sum_{i\neq j}\ps{h_2,e_i}\int_{\C^3} \frac{\theta_\delta(z_0-x_1)\theta_\delta(z_0-x_2)\theta_\delta(z_0-x_3)}{(z_0-x_1)(x_1-x_2)_\eps(x_1-x_3)_\eps}\ps{V_{\gamma e_i,\eps}(x_1)V_{\gamma e_i,\eps}(x_2)V_{\gamma e_j,\eps}(x_3)\bm{\mathrm V}_\eps}_\delta d^2x_1d^2x_2d^2x_3.
\end{align*}
We can use integration by parts in the same way as we have proceeded for the $1$-fold integrals. For this we notice that the above quantity is nothing but
\[
-2\mu^2\gamma\sum_{i\neq j}\ps{h_2,e_i}\int_{\C^2}\partial_{x_1}\left( \frac{\theta_\delta(z_0-x_1)}{(z_0-x_1)}\frac{\theta_\delta(z_0-x_2)}{(x_1-x_2)_\eps}\ps{V_{\gamma e_i,\eps}(x_1)V_{\gamma e_j,\eps}(x_2)\bm{\mathrm V}_\eps}_\delta\right) d^2x_1d^2x_2+\mathfrak{R}_2(\eps,\delta)
\]
where we have introduced another remainder term
\begin{equation*}
    \mathfrak{R}_2(\eps,\delta)\coloneqq -2\mu^2\gamma\sum_{i,j=1}^2\ps{h_2,e_i}\int_{\C^2}\frac{\partial_{x_1}\theta_{\delta}(z_0-x_1)\theta_{\delta}(z_0-x_2)}{(z_0-x_1)(x_1-x_2)_\eps}\ps{V_{\gamma e_i,\eps}(x_1)V_{\gamma e_j,\eps}(x_2)\V_\eps}_\delta d^2x_1d^2x_2.
\end{equation*}
Combining all the terms we get
\[
q^2\mathrm{I}_{1}+\frac q2\mathrm{I}_{2}+\mathrm{I}_3-\sum_{k=1}^N\left(\frac{w^{(3)}(\alpha_k)}{(z_0-z_k)^3}+\frac{\mathrm J_{-1}^{(k)}}{(z_0-z_k)^2}+\frac{\mathrm J_{-2}^{(k)}}{z_0-z_k}\right)=\mathfrak{R}(\eps,\delta)+\mathfrak{R}_1(\eps,\delta)+\mathfrak{R}_2(\eps,\delta).
\]
\subsubsection{Last step: taking the limit}
Therefore in order to prove the Ward identity~\eqref{Ward_higher} it suffices to show that the remaining terms $\mathfrak{R}_{1,2}$ actually vanish as $\eps,\delta\rightarrow0$, since we have already seen that the term $\mathfrak{R}(\eps,\delta)$ coming from the symmetrization was disappearing in the limit. As we will see this is ensured by the regularity of the correlation functions, as stated in Proposition~\ref{regularity}. 

Indeed, let us start by considering $\mathfrak{R}_1(\eps,\delta)$. We need to ensure the convergence as $\eps,\delta\rightarrow0$ of the one-fold integral involved. To start with, we can use integration by parts following the lines of the proof of Proposition~\ref{regularity} to see that the limit of this term is actually given by the $\delta\rightarrow0$ limit of expressions of the type
\[
\int_\C \frac{\partial_x\theta_\delta(x)}{x^2}F_\delta(x)d^2x,
\]
where, thanks to Proposition~\ref{regularity}, $(F_\delta)_{\delta>0}$ is uniformly bounded in $C^1$ in a neighbourhood of the origin. Noticing that $\partial_x\theta_\delta(x)$ is supported in the annulus $A(1/2,1)$ of radii $1/2,1$ centered at the origin and making the change of variable $x\leftrightarrow\delta x$ the latter may be put under the form 
\[
\int_{A(1/2,1)} \frac{\partial_x\theta_1(x)}{x^2}\frac{F_\delta(\delta x)}{\delta}d^2x.
\]
Since $F_\delta$ is $C^1$ we can write its Taylor expansion near the origin; also since $\theta(x)$ is rotation invariant we know that
\[
  \int_{A(1/2,1)} \frac{\partial\theta_1(x)}{x^2}x^kd^2x=\int_{A(1/2,1)} \frac{\partial\theta_1(x)}{x^2}\bar x^{-k}d^2x=0  
\]
for any integer $k\neq3$. As a consequence integrating the Taylor expansion $F_\delta(\delta x)=F_\delta(0)+\delta \left(x \partial_xF_\delta(0)+\bar x \partial_{\bar x}F_\delta(0)\right) + o(\delta)$ yields $0$ plus a term which does vanish in the limit.
The expression $\mathfrak{R}_2(\eps,\delta)$ vanishes using the same reasoning; one subtlety here being that we have to consider what happens when $x_1$ and $x_2$ get close one to the other. However the same reasoning as the one we have described above combined with the one used in the proof of Proposition~\ref{regularity} allows to conclude in the same way. This shows that the following is true:
\[
    \lim\limits_{\delta\rightarrow0}\lim\limits_{\eps\rightarrow0} \left(\ps{\bm {\mathrm W}_\eps(z_0)\bm{\mathrm V}_\eps}_\delta+\frac18\sum_{l=1}^N\left(\frac{w(\alpha_l)}{(z_0-z_l)^3}
    +\frac{\bm {\mathcal W_{-1,\eps}}(z_l,\alpha_l)}{(z_0-z_l)^2}+\frac{\bm {\mathcal W_{-2,\eps}}(z_l,\alpha_l)}{z_0-z_l}\right)\ps{\prod_{k=1}^NV_{\alpha_k,\eps}(z_k)}_\delta\right)=0.
\]
To finish up with the proof of our main result, it remains to show that the limiting quantities 
\[
\bm {\mathcal W}_{-i}(z_l,\alpha_l)\ps{\bm{\mathrm V}}=\lim\limits_{\delta\rightarrow0}\lim\limits_{\eps\rightarrow0}\bm {\mathcal W}_{-i,\eps}(z_l,\alpha_l)\ps{\bm{\mathrm V}_\eps}_\delta
\]
do exist for $i=1,2$. Convergence of $\bm {\mathcal W}_{-1}(z_l,\alpha_l)\ps{\bm{\mathrm V}_\eps}_\delta$ follows from the very same argument that allows to prove the differentiability of the correlation functions. Namely quantities of the form \[\int_\C\frac{1}{(z_l-x)_\eps}\ps{V_{\gamma e_i,\eps}(x)\V_\eps}d^2x\] admit a limit when $\eps\rightarrow0$. As for $\bm {\mathcal W}_{-2,\eps}(z_l,\alpha_l)\ps{\bm{\mathrm V_\eps}}_\delta$, we note that \[
q\left(B(e_i,\alpha)-B(\alpha, e_i)\right)+2C(\alpha,\gamma e_i-\alpha,e_i)=\frac4\gamma\ps{h_2,e_i}\omega_{3-i}(\alpha)\left(1+\frac12\ps{\alpha,\gamma e_i}\right).
\]As a consequence symmetrization identities allow to rewrite the integrals in $\bm {\mathcal W}_{-2,\eps}(z_k,\alpha_k)\ps{\bm{\mathrm V_\eps}}$ as 
\begin{align*}
&-\mu\gamma\sum_{i=1}^2\int_\C\sum_{l\neq k}\left(\frac{C^{\sigma}(\alpha_k,\alpha_l,e_i)}{(z_k-z_l)_\eps(z_k-x)_\eps}-\frac{\ps{h_2,e_i}\omega_{3-i}(\alpha_k)\ps{\alpha_l,e_i}}{(z_l-x)_\eps(z_k-x)_\eps}\right)\ps{V_{\gamma e_1,\eps}(x)\V_\eps}_\delta d^2x\\
&+(\mu\gamma)^2\int_{\C^2}\frac{\ps{e_1,\alpha_k}}{(z_k-x_1)_\eps(x_1-x_2)_\eps}+\frac{\ps{e_2,\alpha_k}}{(z_k-x_2)_\eps(x_1-x_2)_\eps}\ps{V_{\gamma e_1,\eps}(x_1)V_{\gamma e_2,\eps}(x_2)\V_\eps}_\delta d^2x_1d^2x_2.
\end{align*}
We can proceed in the same way as in the proof that the correlation functions are $C^2$ (see Subsection~\ref{proof_C2}) to show that the limit does exist. This wraps up the proof of our main result.

\subsection{Global Ward identities}\label{sub:global_Ward}
The Ward identity \eqref{Ward_higher} that we have just proved shows that the probabilistic model thus defined is indeed consistent with the expectations of the physics literature and may be understood as a manifestation of the higher-spin symmetry enjoyed by the $\mathfrak{sl}_3$ Toda theory. In addition to this identity there is a second building block which is fundamental in the study of Toda theories: the existence of global Ward identities~\eqref{global_Ward} which are manifestations of the (global) higher-spin symmetry of the theory. In order to prove that the equations~\eqref{global_Ward} hold we will prove that the higher-spin current is \lq\lq holomorphic at infinity", a fact which is axiomatic in the physics literature and by which is meant that 
\[
    \bm{\mathrm{W}}(z_0)\sim\frac{1}{z_0^6}
\]
as $z_0\rightarrow\infty$. To do so we rely on the fact that $\bm{\mathrm{W}}$ behaves like a covariant tensor of order three in the following sense.
\begin{proposition}\label{holo_inf}
Under the Seiberg bounds~\eqref{seiberg_bounds} and for any M\"obius transform $\psi$ of the plane
\begin{equation}
    \ps{\bm {\mathrm W}(z_0)\prod_{k=1}^NV_{\alpha_k}(z_k)} =\psi'(z_0)^3\prod_{l=1}^N\norm{\psi'(z_l)}^{2\Delta_{\alpha_l}}\ps{\bm {\mathrm W}(\psi(z_0))\prod_{k=1}^NV_{\alpha_k}(\psi(z_k))}.
\end{equation}
In particular when $\psi:z\mapsto\frac1z$ the asymptotic behaviour of $\bm{\mathrm{W}}(z_0)$ as $z_0\rightarrow\infty$ is given by
\begin{equation}\label{eq:asymptot}
   \ps{\bm {\mathrm W}(z_0)\prod_{k=1}^NV_{\alpha_k}(z_k)} \sim-\frac{1}{z_0^6}\prod_{l=1}^N\norm{z_l}^{-4\Delta_{\alpha_l}}\ps{\bm {\mathrm W}(0)\prod_{k=1}^NV_{\alpha_k}(\frac1{z_k})}.
\end{equation}
\end{proposition}
Assume for now that this statement holds. Then we see that, when $z_0\rightarrow\infty$ the left-hand side in the Ward identity~\eqref{Ward_higher} is therefore  asymptotic to $\frac{1}{z_0^6}$. On the other hand the leading term in the right-hand side of Equation~\eqref{Ward_higher} is given (up to the prefactor $-\frac18$) by 
\[
    \frac1{z_0}\sum_{l=1}^N \bm {\mathcal W_{-2}}(z_l,\alpha_l)\ps{\prod_{k=1}^NV_{\alpha_k}(z_k)}.
\]
Therefore in order for these two asymptotics to be consistent we need to assume that that the $n=0$ global Ward identity~\eqref{global_Ward} holds:\[
\sum_{l=1}^N \bm {\mathcal W_{-2}}(z_l,\alpha_l)\ps{\prod_{k=1}^NV_{\alpha_k}(z_k)}=0.
\]
We may proceed in the same way for the other terms that appear in the asymptotic of the right-hand side of Equation~\eqref{global_Ward}. To do so let us wite the right-hand side of Equation~\eqref{Ward_higher} as
\[
-\frac18\sum_{l=1}^N\left(\frac1{z_0^3}\frac{w(\alpha_l)}{(1-\frac{z_l}{z_0})^3}
    +\frac1{z_0^2}\frac{\bm {\mathcal W_{-1}^{(l)}}}{(1-\frac{z_l}{z_0})^2}+\frac1{z_0}\frac{\bm {\mathcal W_{-2}^{(l)}}}{1-\frac{z_l}{z_0}}\right)\ps{\prod_{k=1}^{N}V_{\alpha_k}(z_k)}
\]
and expand in powers of $\frac1{z_0}$ the terms of the form $\frac1{(1-\frac{z_l}{z_0})^p}$, $p=1,2,3$. Then the asymptotic as $z_0\rightarrow\infty$ can be expanded as negative powers of $z_0$, and corresponding coefficients in this expansion are given by
\[
    \frac{1}{z_0^p}\sum_{l=1}^N \left(z_l^{p-1}\bm {\mathcal W_{-2}}(z_l,\alpha_l)+(p-1)z_l^{p-2}\bm {\mathcal W_{-1}}(z_l,\alpha_l)+\frac{(p-1)(p-2)}{2}z_l^{n-2}w(\alpha_l)\right)\ps{\prod_{k=1}^NV_{\alpha_k}(z_k)}.
\]
Therefore to be consistent with the asymptotic prescribed by Equation~\eqref{eq:asymptot} we need to have, as soon as $0\leq n\leq 4$,
\[
\sum_{l=1}^N \left(z_l^n\bm {\mathcal W_{-2}}(z_l,\alpha_l)+nz_l^{n-1}\bm {\mathcal W_{-1}}(z_l,\alpha_l)+\frac{n(n-1)}{2}z_l^{n-3}w(\alpha_l)\right)\ps{\prod_{k=1}^NV_{\alpha_k}(z_k)}=0.
\]
In other words the global Ward identities~\eqref{global_Ward} hold.
\begin{remark}
With a little bit of extra work it is possible to show that the right-hand side in the asymptotic of Equation~\eqref{eq:asymptot} is given by 
\[
-\frac18\sum_{l=1}^N \left(z_l^5\bm {\mathcal W_{-2}}(z_l,\alpha_l)+5z_l^{4}\bm {\mathcal W_{-1}}(z_l,\alpha_l)+10z_l^3w(\alpha_l)\right)\ps{\prod_{k=1}^NV_{\alpha_k}(z_k)}
\]
so we may not learn anything new from the exact value of the leading term in the asymptotic of the higher-spin current. 
\end{remark}
\begin{proof}[Proof of Proposition~\ref{holo_inf}]
Let us come back to the regularized correlation functions $\ps{\bm {\mathrm W}_\eps(z_0)\prod_{k=1}^NV_{\alpha_k,\eps}(z_k)}$ defined thanks to the mollified field
$\varphi_\eps=\X_\eps+\frac Q2\ln\hat g_\eps+\bm c$. Since the latter is smooth, its derivatives and the ones of $\varphi_\eps\circ\psi$ are well-defined. They are given by
\begin{align*}
    \partial\left(\varphi_\eps\circ\psi\right)&=\psi' \partial\varphi_\eps\circ\psi\\
    \partial^2\left(\varphi_\eps\circ\psi\right)&=\psi'' \partial\varphi_\eps\circ\psi+(\psi')^2 \partial^2\varphi_\eps\circ\psi\\
    \partial^3\left(\varphi_\eps\circ\psi\right)&=\psi''' \partial\varphi_\eps\circ\psi+3\psi'' \psi'\partial^2\varphi_\eps\circ\psi+(\psi')^3\partial^3\varphi_\eps\circ\psi.
\end{align*}
We can proceed in the same way for the map $\ln\norm{\psi'}$ and yields
\begin{align*}
    \partial\left(\ln\norm{\psi'}\right)=\frac{\psi''}{2\psi'}\quad\text{and}\quad\partial^2\left(\ln\norm{\psi'}\right)=\frac{\psi'''\psi'-(\psi'')^2}{2(\psi')^2}\cdot
\end{align*}
We can now apply the tensor $\bm{\mathrm{W}}$ to the field $\varphi_\eps\circ\psi+Q\ln\norm{\psi'}$ instead of $\varphi_\eps$, yielding
\begin{align*}
    &\bm{\mathrm{W}}(\varphi_\eps\circ\psi+Q\ln\norm{\psi'})=-\frac{q^2}{8}\ps{h_2,(\psi')^3\partial^3\varphi_\eps+3\psi'' \psi'\partial^2\varphi_\eps+\psi''' \partial\varphi_\eps}\circ\psi\\
    &+\frac{q}{4}\Big(\psi'\ps{h_2-h_1,\psi'' \partial\varphi_\eps+(\psi')^2 \partial^2\varphi_\eps}\ps{h_1,\partial\varphi_\eps}+\psi'\ps{h_3-h_2,\psi'' \partial\varphi_\eps+(\psi')^2 \partial^2\varphi_\eps}\ps{h_3,\partial\varphi_\eps}\\
    &+q\frac{\psi''}{2\psi'}\ps{3h_2,\psi'' \partial\varphi_\eps+(\psi')^2 \partial^2\varphi_\eps}+q\frac{\psi'''\psi'-(\psi'')^2}{2\psi'}\ps{h_2, \partial\varphi_\eps}\Big)\circ\psi\\
    &+(\psi')^3\ps{h_1,\partial\varphi_\eps}\ps{h_2,\partial\varphi_\eps}\ps{h_3,\partial\varphi_\eps}\circ\psi+\frac q2\psi'\psi''\ps{h_2,\partial\varphi_\eps}\ps{h_3-h_1,\partial\varphi_\eps}\circ\psi-q^2\frac{(\psi'')^2}{4\psi'}\ps{h_2,\partial\varphi_\eps}\circ\psi.
\end{align*}
Therefore we end up with the equality
\[
\bm{\mathrm{W}}(\varphi_\eps\circ\psi+Q\ln\norm{\psi'})=(\psi')^3\bm{\mathrm{W}}(\varphi_\eps)\circ\psi.
\]
On the other hand we know from~\cite[Equation (3.8)]{Toda_construction} that, provided that the limit exists, 
\begin{equation}
    \lim\limits_{\eps\rightarrow0}\ps{F(\varphi_\eps)\prod_{k=1}^NV_{\alpha_k,\eps}(z_k)} =\prod_{k=1}^N\norm{\psi'(z_k)}^{2\Delta_{\alpha_k}}\lim\limits_{\eps\rightarrow0}\ps{F(\varphi_\eps\circ\psi+Q\ln\norm{\psi'})\prod_{k=1}^NV_{\alpha_k,\eps}(\psi(z_k))}
\end{equation}
for any continuous map $F$ on $\mathrm{H}^{-1}(\C\to\mathfrak{h}_3^*,\hat g)$.
As a consequence with $F=\bm{\mathrm{W}}$, in the $\eps\rightarrow0$ limit we will be left with
\[
\ps{\bm {\mathrm W}(z_0)\prod_{k=1}^NV_{\alpha_k}(z_k)} =\psi'(z_0)^3\prod_{k=1}^N\norm{\psi'(z_k)}^{2\Delta_{\alpha_k}}\ps{\bm {\mathrm W}(\psi(z_0))\prod_{k=1}^NV_{\alpha_k}(\psi(z_k))}.
\]
This is the result we were looking for.
\end{proof}


\section{Degenerate fields and a BPZ-type equation}\label{section:BPZ}
In this section we shed light on another manifestation, beyond the Ward identities~\eqref{Ward_higher} and~\eqref{global_Ward}, of the higher-spin symmetry enjoyed by the $\mathfrak{sl}_3$ Toda theory. This is achieved by showing existence of degenerate fields ---that correspond to singular vectors in $W$-algebras--- which in turn implies a BPZ-type differential equation on certain correlation functions. To do so we rely on the explicit form of the $W$-descendent fields shown in the present document. To the best of our knowledge this approach is entirely new and differs from the standard method based on the study of singular vectors of $W$-algebras. For more details on this method see \emph{e.g.} Parts 2 and 3 in~\cite{FaLu}. Throughout this section we will denote by $\V=\prod_{k=1}^NV_{\alpha_k}(z_k)$ a product of Vertex Operators and by $\V_\eps$ its regularized counterpart.

\subsection{A first application to degenerate fields at level one}
To start with let us investigate the existence of (partially) degenerate fields at the level one. These correspond to Vertex Operators $V_\alpha$ for which the weight $\alpha$ takes a special value, which allows to write down the quantity $\bm{\mathcal W_{-1}}$ in terms of derivatives of the correlation functions.

As explained throughout this document, the first $W$-descendent can be expressed in terms of the Toda field via 
\[
\bm{\mathcal W_{-1}^{(1)}}\ps{\V}=\lim\limits_{\eps\rightarrow0}\ps{\bm{\mathrm W_{-1,\eps}}(z_1,\alpha_1)\V_\eps}
\]
where we have set, as before,
\[
\bm{\mathrm W_{-1,\eps}}(z,\alpha)=- qB(\alpha,\partial\varphi_\eps(z))
    -2C(\alpha,\alpha,\partial\varphi_\eps(z)).
\]
Our goal here is to express $\bm{\mathcal W_{-1}}$ in terms of derivatives of the field, and to do so we may introduce the first Virasoro descendent of $V_{\alpha_1}$ by setting 
\[\bm{\mathcal L_{-1}^{(1)}}\ps{\V}\coloneqq \partial_{z_1}\ps{\V}.
\]
Then we use Equation~\eqref{correl_approx} to rewrite the derivative of the correlation functions (we omit the $\delta$-dependence of the correlation functions since $z_1$ does not appear in the correlation functions we consider) as
 \begin{equation*}
  \partial_{z_1}\ps{\V_\eps} = \Big\langle\ps{\alpha_1,\partial\varphi_\eps(z_1)}\V_\eps\Big\rangle.
\end{equation*}
Hence the Virasoro descendent $\bm{\mathcal L_{-1}^{(1)}}\ps{\V}$ can be represented as the limit
\[
\lim\limits_{\eps\rightarrow0}\Big\langle\ps{\alpha_1,\partial\varphi_\eps(z_1)}\V_\eps\Big\rangle.
\]
As a consequence we can express the first $W$-descendent in terms of the first Virasoro descendent when $\bm{\mathrm W_{-1,\eps}}(z,\alpha)$ and $\ps{\alpha,\partial\varphi_\eps(z)}$ are proportional. Therefore we will say that the primary field is degenerate at the level one when there exists a real $\kappa$ for which 
\[
\bm{\mathrm W_{-1,\eps}}(z,\alpha)=\kappa\ps{\alpha,\partial\varphi_\eps(z)}.
\]
Put differently we need to find $\alpha$ such that for all $u\in\mathfrak{h}^*$,
\[
- qB(\alpha,u)
    -2C(\alpha,\alpha,u)=\kappa\ps{\alpha,u}.
\]
Taking respectively $u=e_1$ and $u=e_2$ and setting $\alpha\coloneqq \alpha^1\omega_1+\alpha^2\omega_2$ we see that we must have 
\begin{align*}
    \kappa \alpha^1= q\alpha^1-2\alpha^1 \frac{2\alpha^2+\alpha^1}3\quad\text{and}\quad \kappa \alpha^2= -q\alpha^2+2\alpha^2 \frac{2\alpha^1+\alpha^2}3
\end{align*}
which implies that either $\alpha^1\alpha^2=0$ or $q=\alpha^1+\alpha^2$.
\begin{proposition}\label{degenerate_1}
Degenerate fields at level one are given by the $V_{\alpha}$ with $\alpha$ of the form $\chi\omega_1$ or $\chi\omega_2$ for $\chi<q$, or $\lambda\omega_1 + (q-\lambda)\omega_2$ for some $0<\lambda<q$. In that case we have the relation
\begin{equation}
    \bm{\mathcal W_{-1}^{(1)}}\ps{V_{\alpha}(z_1)\prod_{k=2}^NV_{\alpha_k}(z_k)}=\frac{3w(\alpha)}{2\Delta_\alpha}\partial_{z_1}\ps{V_\alpha(z_1)\prod_{k=2}^NV_{\alpha_k}(z_k)},
\end{equation}
valid as soon as the $\bm z$ are distinct and the $\bm \alpha$ satisfies the Seiberg bounds~\eqref{seiberg_bounds}.
\end{proposition}

\subsection{Degenerate fields at the levels two and three}
We may proceed in the same way for other descendents at the levels two and three. Vertex Operators that satisfy similar requirements as above but for descendents at the second or third level will be called \emph{degenerate fields at the second or third level}. These are also sometimes referred to as \emph{fully degenerate fields} by contrast with degenerate fields at the first level only. In order to keep notations simple we will omit the Wick product convention in this subsection: for instance quantities such as $\ps{\partial\varphi(z),\partial\varphi(z)}$ are to be understood as $:\ps{\partial\varphi(z),\partial\varphi(z)}:$ . Similarly we no longer use the $\theta$-regularization of the correlation functions in order not to overload notations.

We start by considering the Virasoro descendents of order two. We define them by the expressions 
\begin{align}
    &\bm{\mathrm L_{-(1,1)}}(z,\alpha)\coloneqq\ps{\alpha,\partial^2\varphi(z)}+\ps{\alpha,\partial\varphi(z)}^2,\\
    &\bm{\mathrm L_{-2}}(z,\alpha)\coloneqq\ps{Q+\alpha,\partial^2\varphi(z)}-\ps{\partial\varphi(z),\partial\varphi(z)}.
\end{align}
By doing so it is immediate to see that
\begin{align*}
    \ps{\bm{\mathrm L_{-(1,1)}}(z_1,\alpha_1)\V}&=\partial_{z_1}^2\ps{\V}\\
    &=\left(\bm{\mathcal L_{-1}^{(1)}}\right)^2\ps{\V}.
\end{align*}
The effect of the $L_{-2}$ descendent on the correlation functions is less obvious; however conformal invariance allows to claim the following statement.
\begin{lemma}\label{lemma:descendents_two}
Set $\bm{\mathcal L_{-2}^{(1)}}\ps{\V}\coloneqq\lim\limits_{\eps\rightarrow0}\ps{\bm{\mathrm L_{-2,\eps}}(z,\alpha)V_{\alpha,\eps}(z)\V_\eps}$. Then
\begin{equation}
    \bm{\mathcal L_{-2}^{(1)}}\ps{V_{\alpha}(z)\V}=\sum_{k=2}^N\left(\frac{\partial_{z_k}}{z-z_k}+\frac{\Delta_{\alpha_k}}{(z-z_k)^2}\right)\ps{V_{\alpha}(z)\V}.
\end{equation}
\end{lemma} 
\begin{proof}
At the regularized level, the right-hand side divided by the non-zero quantity $\ps{V_{\alpha}(z)\V}$ is given by
\begin{align*}
&=\sum_{k\neq l}\frac{\ps{\alpha_k,\alpha_l}}{2(z-z_k)(z_l-z_k)_\eps}+\sum_{k=1}^N\frac{\ps{\alpha_k,\alpha}}{2(z-z_k)(z-z_k)_\eps}+\frac{\Delta_{\alpha_k}}{(z-z_k)^2} \quad + \rm I\\
&=\sum_{k\neq l}\frac{\ps{\alpha_k,\alpha_l}}{4(z_l-z_k)}\left(\frac{1}{z-z_k}-\frac{1}{z-z_l} \right)+\sum_{k=1}^N\frac{\ps{\alpha_k,Q+\alpha}}{2(z-z_k)^2}-\frac{\ps{\alpha_k,\alpha_k}}{4(z-z_k)^2} \quad +o_\eps(1)\quad+  \rm I\\
&=-\sum_{k\neq l}\frac{\ps{\alpha_k,\alpha_l}}{4(z-z_k)(z-z_l)}+\sum_{k=1}^N\frac{\ps{\alpha_k,Q+\alpha}}{2(z-z_k)^2}-\frac{\ps{\alpha_k,\alpha_k}}{4(z-z_k)^2} \quad +o_\eps(1)\quad+  \rm I\\
&=-\sum_{k,l=1}^N\frac{\ps{\alpha_k,\alpha_l}}{4(z-z_k)(z-z_l)}+\sum_{k=1}^N\frac{\ps{\alpha_k,Q+\alpha}}{2(z-z_k)^2} \quad +o_\eps(1)\quad + \rm I,
\end{align*}
where $\ps{V_{\alpha}(z)\V}\rm I$ is equal to:
\[
\frac {\mu \gamma}2\sum_{i=1}^2\int_\C \sum_{k=1}^N\frac{\ps{\alpha_k,e_i}}{2(z-z_k)(z_k-x)_\eps}\ps{V_{\gamma e_i,\eps}(x)V_{\alpha,\eps}(z)\bm{\mathrm V}_\eps}_\delta d^2x.
\]
Using the same reasoning as in Section~\ref{section_proof} involving symmetrisation and Stokes' formula, the latter is equal (up to terms that vanish in the $\eps\rightarrow0$ limit) to
\begin{align*}
&\frac {\mu \gamma}2\sum_{i=1}^2\int_\C\left(\sum_{k=1}^N\frac{\ps{\alpha_k,e_i}}{2(z-z_k)(z-x)_\eps}-\frac{\ps{\alpha,e_i}+\frac2\gamma}{(z-x)^2_\eps}\right)\ps{V_{\gamma e_i,\eps}(x)V_{\alpha,\eps}(z)\bm{\mathrm V}_\eps}_\delta d^2x\\
    &-\frac {(\mu \gamma)^2}4\sum_{i,j=1}^2\int_{\C^2}\frac{\ps{e_i,e_j}}{(z-x_1)_\eps(z-x_2)_\eps}\ps{V_{\gamma e_i,\eps}(x_1)V_{\gamma e_j,\eps}(x_2)V_{\alpha,\eps}(z)\bm{\mathrm V}_\eps}_\delta d^2x_1d^2x_2.
\end{align*}
This expression coincides with the left-hand side in the $\eps\rightarrow0$ limit by applying Gaussian integration by parts in the same spirit as in Section~\ref{section_proof}, concluding the proof:
\begin{align*}
\ps{\bm{\mathrm L_{-2,\eps}}(z,\alpha)V_{\alpha,\eps}(z)\V_\eps}&=\sum_{k=1}^N\frac{\ps{Q+\alpha,\alpha_k}}{2(z-z_k)^2_\eps}\ps{V_{\alpha,\eps}(z)\bm{\mathrm V}_\eps} -\frac {\mu \gamma}2\sum_{i=1}^2\int_\C\frac{\ps{Q+\alpha,e_i}}{(z-x)^2_\eps}\ps{V_{\gamma e_i,\eps}(x)V_{\alpha,\eps}(z)\bm{\mathrm V}_\eps}d^2x\\
    -&\sum_{k,l=1}^N\frac{\ps{\alpha_k,\alpha_l}}{4(z-z_k)_\eps(z-z_l)_\eps}\ps{V_{\alpha,\eps}(z)\bm{\mathrm V}_\eps}\\
    +\frac{\mu \gamma }4&\sum_{i=1}^2\int_\C\left(\sum_{k=1}^N\frac{\ps{\alpha_k,e_i}}{(z-z_k)_\eps(z-x)_\eps}+\frac{\gamma\ps{e_i,e_i}}{(z-x)_\eps(z-x)_\eps}\right)\ps{V_{\gamma e_i,\eps}(x)V_{\alpha,\eps}(z)\bm{\mathrm V}_\eps}d^2x\\
    -&\frac {(\mu \gamma)^2}4\sum_{i,j=1}^2\int_{\C^2}\frac{\ps{e_i,e_j}}{(z-x_1)_\eps(z-x_2)_\eps}\ps{V_{\gamma e_i,\eps}(x_1)V_{\gamma e_j,\eps}(x_2)V_{\alpha,\eps}(z)\bm{\mathrm V}_\eps} d^2x_1d^2x_2.
\end{align*}
\end{proof}
Thanks to the expression of the Virasoro descendents at the level two, we can readily search for degenerate fields at the level two. As before, we think of them as Vertex Operators for which the $W$-descendent at the level two, $\bm{\mathrm W_{-2}}(z,\alpha)$, can be expressed as a linear combination of Virasoro descendents at the level two. Put differently, these are $\alpha$ such that there exist $\kappa_1,\kappa_2$ real numbers with
\[
    \kappa_1\bm{\mathrm L_{-(1,1)}}+\kappa_2\bm{\mathrm L_{-2}}=\bm{\mathrm W_{-2}}.
\]
The latter implies that $\alpha=\alpha_1\omega_1+\alpha_2\omega_2$ is a solution of the following set of equations
\begin{align*}
    \left\lbrace
    \begin{array}{cc}
      \kappa_1\alpha_1^2-2\kappa_2&=2\frac{2\alpha_2+\alpha_1}{3}\\
        \kappa_1\alpha_2^2-2\kappa_2&=-2\frac{2\alpha_1+\alpha_2}{3}\\
        \kappa_1\alpha_1\alpha_2+\kappa_2&=2\frac{\alpha_1-\alpha_2}{3}\\
        \kappa_1\alpha_1+\kappa_2(q+\alpha_1)&=-(q+\alpha_1)\frac{2\alpha_2+\alpha_1}{3}\\
        \kappa_1\alpha_2+\kappa_2(q+\alpha_2)&=(q-\alpha_2)\frac{2\alpha_1+\alpha_2}{3}.
    \end{array}
    \right.
\end{align*}
Explicit computations show the following.
\begin{proposition}
Degenerate fields at the level two are given by the Vertex Operators $V_\alpha$ with $\alpha$ of the form $-\chi\omega_{i}$ where $\chi$ is either $\frac2\gamma$ or $\gamma$ and $i\in\{1,2\}$. In that case 
\begin{equation}
\begin{split}
    \bm{\mathrm W_{-2}}(\cdot,\alpha)&=-\frac{4}{\chi}\bm{\mathrm L_{-(1,1)}}(\cdot,\alpha)-\frac{4\chi}3\bm{\mathrm L_{-2}}(\cdot,\alpha),\text{ hence}\\
    \bm{\mathcal W_{-2}^{(1)}}\ps{V_{\alpha}(z_1)\prod_{k=2}^NV_{\alpha_k}(z_k)}&=-\left(\frac{4}{\chi}\left(\bm{\mathcal L_{-1}^{(1)}}\right)^2+\frac{4\chi}3\bm{\mathcal L_{-2}^{(1)}}\right)\ps{V_{\alpha}(z_1)\prod_{k=2}^NV_{\alpha_k}(z_k)}.
\end{split} 
\end{equation}
\end{proposition}
A similar reasoning remains true when we turn to descendents at the level three; nevertheless calculations are slightly more involved. Like before, we shall first introduce the three Virasoro descendents at the third level by setting 
\begin{equation}
        \begin{split}
            \bm{\mathrm L_{-(1,1,1)}}(z,\alpha)\coloneqq&\ps{\alpha,\partial^3\varphi(z)}+3\ps{\alpha,\partial\varphi(z)}\ps{\alpha,\partial^2\varphi(z)}+\ps{\alpha,\partial\varphi(z)}^3\\
    \bm{\mathrm L_{-(1,2)}}(z,\alpha)\coloneqq&\ps{Q+\alpha,\partial^3\varphi(z)}+\ps{Q+\alpha,\partial^2\varphi(z)}\ps{\alpha,\partial\varphi(z)}\\
    &-2\ps{\partial^2\varphi(z),\partial\varphi(z)}-\ps{\partial\varphi(z),\partial\varphi(z)}\ps{\alpha,\partial\varphi(z)}\cdot
        \end{split}
\end{equation}
These are such that (provided that the objects exist)
\begin{equation}\label{eq:L111_L12_limit}
\begin{split}
            \ps{\bm{\mathrm L_{-(1,1,1)}}(z_1,\alpha_1)\V}&=\left(\bm{\mathcal L_{-1}^{(1)}}\right)^3\ps{\V}\\
            \ps{\bm{\mathrm L_{-(1,2)}}(z_1,\alpha_1)\V}&=\bm{\mathcal L_{-1}^{(1)}}\bm{\mathcal L_{-2}^{(1)}}\ps{\V}\cdot
        \end{split}
\end{equation}
We also introduce 
\begin{equation}
\bm{\mathrm L_{-3}}(z,\alpha)\coloneqq\ps{Q+\frac\alpha2,\partial^3\varphi(z)}-2\ps{\partial^2\varphi(z),\partial\varphi(z)}
\end{equation}
which is has been defined in order to satisfy the below equation:
\begin{equation}\label{eq:L3_limit}
\ps{\bm{\mathrm L_{-3}}(z_1,\alpha_1)\V}=\sum_{k=2}^N\left(\frac{\partial_{z_k}}{(z_1-z_k)^2}+\frac{2\Delta_{\alpha_k}}{(z_1-z_k)^3}\right)\ps{\V}.
\end{equation}
To see that we may proceed like in the proof of Lemma~\ref{lemma:descendents_two} and the very same arguments still apply. Nonetheless to motivate this claim simply note that 
\begin{align*}
&\sum_{k=1}^N\left(\frac{\partial_{z_k}}{(z-z_k)^2}+\frac{2\Delta_{\alpha_k}}{(z-z_k)^3}\right)\ps{V_{\alpha}(z)\V}\\
&=\sum_{k\neq l}\frac{\ps{\alpha_k,\alpha_l}}{2(z-z_k)^2(z_l-z_k)}+\sum_{k=1}^N\frac{\ps{\alpha_k,\alpha}}{2(z-z_k)^3}+\frac{2\Delta_{\alpha_k}}{(z-z_k)^3} \quad + \quad\text{integral terms}\\
&=\sum_{k\neq l}\frac{\ps{\alpha_k,\alpha_l}}{4(z_l-z_k)}\left(\frac{1}{(z-z_k)^2}-\frac{1}{(z-z_l)^2} \right)+\sum_{k=1}^N\frac{\ps{\alpha_k,Q+\frac\alpha2}}{(z-z_k)^3}-\frac{\ps{\alpha_k,\alpha_k}}{2(z-z_k)^3} \quad + \quad \text{integral terms}\\
&=-\sum_{k\neq l}\frac{\ps{\alpha_k,\alpha_l}}{2(z-z_k)^2(z-z_l)}+\sum_{k=1}^N\frac{\ps{\alpha_k,Q+\frac\alpha2}}{(z-z_k)^3}-\frac{\ps{\alpha_k,\alpha_k}}{2(z-z_k)^3} \quad + \quad \text{integral terms}\\
&=-\sum_{k,l=1}^N\frac{\ps{\alpha_k,\alpha_l}}{2(z-z_k)^2(z-z_l)}+\sum_{k=1}^N\frac{\ps{\alpha_k,Q+\frac\alpha2}}{(z-z_k)^3} \quad + \quad \text{integral terms}\\
&=\Big\langle-2\ps{\partial^2\varphi(z),\partial\varphi(z)}V_{\alpha}(z)\V\Big\rangle + \Big\langle\langle Q+\frac\alpha2,\partial^3\varphi(z)\rangle V_{\alpha}(z)\V\Big\rangle.
\end{align*}

We now turn to the $W$-descendents at the third level. With a reasoning similar to the one of Lemma~\ref{lemma:descendents_two} we end up with:
\begin{lemma}\label{lemma:descendents_three}
Let us set
\begin{equation}
    \begin{split}
        \bm{\mathrm W_{-3}}(z,\alpha)&\coloneqq q^2\ps{h_2,\partial^3\varphi(z)}+\frac q2\left(2B(\partial^3\varphi(z),\alpha)-B(\alpha,\partial^3\varphi(z))\right)-C(\partial^3\varphi(z),\alpha,\alpha)\\
    &-2qB(\partial^2\varphi(z),\partial\varphi(z))+4C(\partial^2\varphi(z),\partial\varphi(z),\alpha)+4C(\partial\varphi(z),\partial^2\varphi(z),\alpha)\\
    &-8\ps{h_1,\partial\varphi(z)}\ps{h_2,\partial\varphi(z)}\ps{h_3,\partial\varphi(z)}.
    \end{split}
\end{equation}
Then
\begin{equation}\label{eq:W3_limit}
   \lim\limits_{\eps\rightarrow0}\ps{\bm{\mathrm W_{-3,\eps}}(z_1,\alpha_1)\V_\eps}=\sum_{k=2}^N\left(\frac{\bm {\mathcal W_{-2}^{(k)}}}{z_1-z_k}
    +\frac{\bm {\mathcal W_{-1}^{(k)}}}{(z_1-z_k)^2}+\frac{w(\alpha_k)}{(z_1-z_k)^3}\right)\ps{\V}.
\end{equation}
\end{lemma}
\begin{proof}
Computations parallel the ones made in Section~\ref{section_proof}. The right-hand side in~\eqref{eq:W3_limit}, when divided by $\ps{V_{\alpha}(z)\V}$, can be expanded as
\begin{align*}
    &\sum_{k,l,p\text{ distinct}}\frac{C(\alpha_k,\alpha_l,\alpha_p)}{(z-z_k)(z_k-z_l)_\eps(z_k-z_p)_\eps}+\sum_{k\neq l}\frac{C^{\sigma}(\alpha,\alpha_k,\alpha_l)}{(z-z_k)(z_k-z)_\eps(z_k-z_l)_\eps}\\
	&+\sum_{k\neq l}\frac{q\left(B(\alpha_l,\alpha_k)-B(\alpha_k,\alpha_l)\right)+2C(\alpha_k,\alpha_l-\alpha_k,\alpha_l)}{2(z-z_k)(z_k-z_l)_\eps^2}+\frac{q B(\alpha_k,\alpha_l)+2C(\alpha_k,\alpha_k,\alpha_l)}{2(z-z_k)^2(z_k-z_l)_\eps}\\
&+\sum_{k=1}^N\frac{q\left(B(\alpha,\alpha_k)-B(\alpha_k,\alpha)\right)+2C(\alpha_k,\alpha-\alpha_k,\alpha)}{2(z-z_k)(z-z_k)_\eps^2}+\frac{q B(\alpha_k,\alpha)+2C(\alpha_k,\alpha_k,\alpha)}{2(z-z_k)^2(z_k-z)_\eps}+\frac{w(\alpha_k)}{(z-z_k)^3}\\
&+\text{integral terms}.
\end{align*}
The first quantity can be easily dealt with by recursive application of symmetrisation identities:
\begin{align*}
&\sum_{k,l,p\text{ distinct}}\frac{C(\alpha_k,\alpha_l,\alpha_p)}{(z-z_k)(z_k-z_l)(z_k-z_p)}=\frac13\sum_{k,l,p\text{ distinct}}\frac{C(\alpha_k,\alpha_l,\alpha_p)}{(z-z_k)(z-z_l)(z-z_p)}\\
&=\sum_{k,l,p}\frac{\ps{\alpha_k,h_1}\ps{\alpha_l,h_2}\ps{\alpha_p,h_3}}{(z-z_k)(z-z_l)(z-z_p)}-\sum_{k,l}\frac{C(\alpha_k,\alpha_k,\alpha_l)}{(z-z_k)^2(z-z_l)}-\sum_{k=1}^N\frac{\ps{\alpha_k,h_1}\ps{\alpha_k,h_2}\ps{\alpha_k,h_3}}{(z-z_k)^3}.
\end{align*}
Similarly we can write that
\begin{align*}
&\sum_{k\neq l}\frac{q B(\alpha_k,\alpha_l)+2C(\alpha_k,\alpha_k,\alpha_l)}{2(z-z_k)^2(z_k-z_l)}+\frac{q\left(B(\alpha_l,\alpha_k)-B(\alpha_k,\alpha_l)\right)+2C(\alpha_k,\alpha_l-\alpha_k,\alpha_l)}{2(z-z_k)(z_k-z_l)^2}\\
&=\sum_{k\neq l}\frac{q B(\alpha_k,\alpha_l)+2C(\alpha_k,\alpha_k,\alpha_l)}{2(z-z_k)^2(z_k-z_l)}+\frac{q\left(B(\alpha_l,\alpha_k)-B(\alpha_k,\alpha_l)\right)+2C(\alpha_k,\alpha_l-\alpha_k,\alpha_l)}{4(z-z_k)(z-z_l)(z_k-z_l)}\\
&=\sum_{k\neq l}\frac{q B(\alpha_k,\alpha_l)+2C(\alpha_k,\alpha_k,\alpha_l)}{2(z-z_k)^2(z-z_l)}\\
&=\sum_{k, l}\frac{q B(\alpha_k,\alpha_l)+2C(\alpha_k,\alpha_k,\alpha_l)}{2(z-z_k)^2(z-z_l)}-\sum_{k=1}^N\frac{q B(\alpha_k,\alpha_k)+2C(\alpha_k,\alpha_k,\alpha_k)}{2(z-z_k)^3}\cdot
\end{align*}
Therefore the renormalized right-hand side in~\eqref{eq:W3_limit} is actually equal to
\begin{equation}\label{eq:W3_developpe}
\begin{split}
    &\sum_{k,l,p}\frac{\ps{\alpha_k,h_1}\ps{\alpha_l,h_2}\ps{\alpha_p,h_3}}{(z-z_k)(z-z_l)(z-z_p)}+\sum_{k,l}\frac{q B(\alpha_k,\alpha_l)-2C^\sigma(\alpha,\alpha_k,\alpha_l)}{2(z-z_k)^2(z-z_l)}\\
    &+\sum_{k=1}^N\frac{-2q^2\ps{\alpha_k,h_2} +q\left(B(\alpha,\alpha_k)-2B(\alpha_k,\alpha)\right)+2C(\alpha_k,\alpha,\alpha)}{2(z-z_k)^3}+o_\eps(1)\hspace{0.2cm}+\text{integral terms}.
    \end{split}
\end{equation}
It remains to treat the integral terms. This is done with a treatment similar to the one we have used in Section~\ref{section_proof}, the only difference being that there is an extra Vertex Operator within the correlation function. Nevertheless along the same lines we get (up to a factor $\ps{V_{\alpha,\eps}(z)V_\eps}$ and a $o_\eps(1)$ quantity) a sum of the terms
\begin{align*}
&-\mu\gamma\sum_{i=1}^2\int_\C\left(\frac{-\frac{2q}\gamma \ps{e_i,h_2}}{(z-x)^3_\eps}+\sum_{k=1}^N\left(\frac{q B(\alpha_k,e_i)+ 2C(\alpha_k,\alpha_k,e_i)}{2(z-z_k)^2(z-x)_\eps}+\frac{q B(e_i,\alpha_k)+2\gamma C(e_i,e_i,\alpha_k)}{2(z-z_k)(z-x)^2_\eps}\right)\right.\\
&\left.\hspace{5cm}+\sum_{k,l=1}^N\frac{C(\alpha_k,\alpha_l,e_i)}{(z-z_k)(z-z_l)(z-x)_\eps}\right)\ps{V_{\gamma e_i,\eps}(x)V_{\alpha,\eps}(z)\bm{\mathrm V}_\eps}_\delta d^2x\\
&+(\mu \gamma)^2\sum_{i,j=1}^2\int_{\C^2}\left(\sum_{k=1}^N\frac{C(e_i,e_j,\alpha_k)}{(z-z_k)(z-x_1)_\eps(z-x_2)_\eps}+\frac{q B(e_i,e_j)+\gamma C(e_i,e_i,e_j)}{2(z-x_1)_\eps^2(z-x_2)_\eps}\right)\\
&\hspace{8cm}\ps{V_{\gamma e_i,\eps}(x_1)V_{\gamma e_j,\eps}(x_2)V_{\alpha,\eps}(z)\bm{\mathrm V}_\eps}_\delta d^2x_1d^2x_2\\
&-(\mu \gamma)^3\sum_{i,j,f=1}^2\int_{\C^3}\frac{\ps{e_i,h_1}\ps{e_j,h_2}\ps{e_f,h_3}}{(z-x_1)_\eps(z-x_2)_\eps(z-x_3)_\eps}\ps{V_{\gamma e_i,\eps}(x_1)V_{\gamma e_j,\eps}(x_2)V_{\gamma e_f,\eps}(x_3)V_{\alpha,\eps}(z)\bm{\mathrm V}_\eps} d^2x_1d^2x_2d^2x_3\\
\end{align*}
coming from the reasoning developed in Section~\ref{section_proof}, and additional ones given by
\begin{align*}
&-\mu\gamma\sum_{i=1}^2\ps{h_2,e_i}\int_\C\ps{V_{\gamma e_i,\eps}(x)V_{\alpha,\eps}(z)\bm{\mathrm V}_\eps}_\delta\left(\frac{\ps{2Q+(\frac2\gamma-3\gamma)\omega_{3-i},\alpha}+2\ps{e_i,\alpha}\ps{\omega_{3-i},\alpha}}{2(z-x)^3_\eps}\right.\\
&\hspace*{2cm}\left.+\sum_{k=1}^N\frac{\ps{\alpha,e_i}\omega_{3-i}(\alpha_k)+\ps{\alpha_k,e_i}\omega_{3-i}(\alpha)}{(z-z_k)^2(z-x)_\eps}+\frac{\ps{\alpha,e_i}\omega_{3-i}(\alpha_k)+\ps{\alpha_k,e_i}\omega_{3-i}(\alpha)}{(z-z_k)(z-x)^2_\eps}\right) d^2x\\
&-(\mu \gamma)^2\sum_{i,j=1}^2\int_{\C^2}\frac{C^\sigma(\alpha,e_i,e_j)}{(z-x_1)_\eps^2(z-x_2)_\eps}\ps{V_{\gamma e_i,\eps}(x_1)V_{\gamma e_j,\eps}(x_2)V_{\alpha,\eps}(z)\bm{\mathrm V}_\eps}_\delta d^2x_1d^2x_2
\end{align*}
that account for the extra Vertex operator in the correlation function and which are obtained by recursive integration by parts.
Therefore using Gaussian integration by parts as well as the explicit expression of $B$ and $C$ we see that the expression~\eqref{eq:W3_developpe} coincides with $\ps{\bm{\mathrm W_{-3,\eps}}(z,\alpha)V_{\alpha,\eps}(z)\V_\eps}$, up to a term that vanishes in the $\eps\rightarrow0$ limit.
\end{proof}

We are now in position to address the question of finding degenerate fields at the level three.
\begin{proposition}\label{degenerate_2}
Degenerate fields at the levels two and three are given by the $V_{\alpha}$ with $\alpha$ of the form $-\chi\omega_1$ or $-\chi\omega_2$ with $\chi\in\lbrace \gamma,\frac2\gamma\rbrace$. In that case both relations
\begin{equation}
    \bm{\mathrm W_{-2}}=-\frac{4}{\chi}\bm{\mathrm L_{-(1,1)}}-\frac{4\chi}3\bm{\mathrm L_{-2}}\quad\text{and}
\end{equation}
\begin{equation}\label{eq:degenerate_W3}
        \bm{\mathrm W_{-3}}=-\left(\frac{\chi}{3}+\frac2{\chi}\right)\bm{\mathrm L_{-3}}+\frac4{\chi}\bm{\mathrm L_{-(1,2)}}+\frac{8}{\chi^3}\bm{\mathrm L_{-(1,1,1)}},
\end{equation}
are valid as soon as the $\bm z$ are distinct and that $\bm \alpha$ satisfies the Seiberg bounds~\eqref{seiberg_bounds}.

As a consequence correlation functions of the form $\ps{V_{-\chi\omega_1}(z)\V}$ are solutions of the following BPZ identity:
\begin{multline}\label{eq:BPZ}
    \left[\frac{8}{\chi^3}\partial_z^3+\frac4\chi\partial_z\sum_{k=1}^N\left(\frac{\partial_{z_k}}{z-z_k}+\frac{\Delta_{\alpha_k}}{(z-z_k)^2}\right)+\left(\frac{\chi}{3}+\frac2\chi\right)\sum_{k=1}^N\left(\frac{\partial_{z_k}}{(z-z_k)^{2}}+\frac{2\Delta_{\alpha_k}}{(z-z_k)^3}\right)\right.\\
\left.    -\sum_{k=1}^N\left(\frac{\bm{\mathcal W_{-2}^{(k)}}}{z-z_k}+\frac{\bm{\mathcal W_{-1}^{(k)}}}{(z-z_k)^{2}}+\frac{w(\alpha_k)}{(z-z_k)^3}\right)\right]\ps{V_{-\chi\omega_1}(z)\V}=0.
\end{multline}
\end{proposition}
\begin{proof}
Equation~\eqref{eq:BPZ} simply corresponds to inserting the equality in Equation~\eqref{eq:degenerate_W3} within a correlation function and using Equations~\eqref{eq:L111_L12_limit},~\eqref{eq:L3_limit} and~\eqref{eq:W3_limit}. 
Proving Equation~\eqref{eq:degenerate_W3} follows from calculations very similar to those we have done for degenerate fields at the levels one and two.
\end{proof}

\subsection{Implications on a four-point correlation function}
The expression of the BPZ-type equation~\eqref{eq:BPZ} is not always tractable when it comes to deriving exact expressions for the correlation functions. However when considering a small number of Vertex Operators some cancellations occur, allowing to express the differential equation in several variables~\eqref{eq:BPZ} as a differential equation in only one variable.
Indeed let us consider a four-point correlation function with one point at infinity $\ps{V_{\alpha}(z)V_{\alpha_0}(0)V_{\alpha_1}(1)V_{\alpha_\infty}(\infty)}$, defined by the limit
\begin{equation}
\ps{V_{\alpha}(z)V_{\alpha_0}(0)V_{\alpha_1}(1)V_{\alpha_\infty}(\infty)}\coloneqq\lim\limits_{z'\rightarrow\infty}\norm{ z'}^{4\Delta_{\alpha_\infty}}\ps{V_{\alpha}(z)V_{\alpha_0}(0)V_{\alpha_1}(1)V_{\alpha_\infty}(z')}.
\end{equation}
This limit is non-zero and admits the following representation:
\begin{equation*}
\ps{V_{\alpha}(z)V_{\alpha_0}(0)V_{\alpha_1}(1)V_{\alpha_\infty}(\infty)}=\norm{z}^{-\ps{\alpha_0,\alpha}}\norm{z-1}^{-\ps{\alpha_1,\alpha}}\mathcal H(z)
\end{equation*}
where $\mathcal H(z)$ is equal to, with $\bm\alpha=(\alpha,\alpha_0,\alpha_1,\alpha_\infty)$:
\begin{equation*}
\left (\prod_{i=1}^{2} \frac{\Gamma(s_i)\mu_i^{-s_i}}\gamma\right)\E\left[\prod_{i=1}^{2}\left(\int_\C\frac{\hat g(y_i)^{-\frac{\gamma}{4}\sum_{k=1}^4\ps{\alpha_k,e_i}}}{\prod_{k=1}^3\norm{z_k-y_i}^{\gamma\ps{\alpha_k,e_i}}}M^{\gamma e_i,\hat g}(d^2y_i)\right)^{-s_i}\right].
\end{equation*}
Then by using the global Ward identities we can express the different quantities $\bm{\mathcal W_{-i}^{(k)}}$ for $k=0,1,\infty$ and $i=1,2$ in terms of $\bm{\mathcal W_{-i}}$, $i=1,2$, and $\bm{\mathcal W_{-1}^{(1)}}$. Heuristically this is due to the fact that there are six such quantities linked by five linearly independent constraints. Inverting this system of equations yields the following: 
\begin{equation}
    \begin{split}
    \bm{\mathcal W_{-3}}\ps{V_{\alpha}(z)\V}=\left[- \left(\frac2z+\frac1{z-1}\right)\bm{\mathcal W_{-2}}-\left(\frac1{z^2}+\frac2{z(z-1)}\right)\bm{\mathcal W_{-1}}+\frac{\bm{\mathcal W_{-1}^{(1)}}}{(z(z-1))^2}\right.&\\
    \left.-\frac{w_\alpha+w_\infty}{z^2(z-1)}-\frac{w_0}{z^3(z-1)}+\frac{w_1}{z(z-1)^2}\left(\frac1z+\frac1{z-1}\right)\right]\ps{V_{\alpha}(z)\V}&.
    \end{split}
\end{equation}
In the special case where the Vertex Operator $V_\alpha$ is fully degenerate, we can use Propositions~\ref{degenerate_1} and~\ref{degenerate_2} to rewrite the above using differential operators:
\begin{equation}
    \begin{split}
    &\left[\frac{8}{\chi^3}\partial_z^3+\frac4\chi\partial_z\bm{\mathcal L_{-2}}-\left(\frac{\chi}{3}+\frac2\chi\right)\bm{\mathcal L_{-3}}\right]\ps{V_{-\chi\omega_1}(z)\V}\\
    =&\left[ \left(\frac2z+\frac1{z-1}\right)\left(\frac4\chi\partial_z^2+\frac{4\chi}{3} \bm{\mathcal L_{-2}}\right)-\left(\frac1{z^2}+\frac2{z(z-1)}\right)\left(\frac{5\chi}3+\frac2\chi\right)\partial_z+\frac{\bm{\mathcal W_{-1}^{(1)}}}{(z(z-1))^2}\right.\\
    &\left.-\frac{w+w_\infty}{z^2(z-1)}-\frac{w_0}{z^3(z-1)}+\frac{w_1}{z(z-1)^2}\left(\frac1z+\frac1{z-1}\right)\right]\ps{V_{-\chi\omega_1}(z)\V}.
    \end{split}
\end{equation}
We proceed in the same way for the Virasoro descendents. Using the three global Ward identities given by conformal covariance of the correlation function of Vertex Operators (see \cite[Theorem 3.1]{Toda_construction}) we end up with
\begin{equation}
    \begin{split}
    \bm{\mathcal L_{-3}}\ps{V_{-\chi\omega_1}(z)\V}&=\left[\frac{3z^2-3z+1}{(z(z-1))^2}\partial_z+\frac{\Delta-\Delta_\infty}{z(z-1)}\left(\frac1z+\frac1{z-1}\right)\right.\\
    &\left.+\frac{\Delta_0}{z^2(z-1)}\left(\frac2z+\frac1{z-1}\right)-\frac{\Delta_1}{z(z-1)^2}\left(\frac1z+\frac2{z-1}\right)\right]\ps{V_{-\chi\omega_1}(z)\V}\\
    \bm{\mathcal L_{-2}}\ps{V_{-\chi\omega_1}(z)\V}&=\left[\frac{2z-1}{z(1-z)}\partial_z+\frac{\Delta_\infty-\Delta}{z(z-1)}-\frac{\Delta_0}{z^2(z-1)}+\frac{\Delta_1}{z(z-1)^2}\right]\ps{V_{-\chi\omega_1}(z)\V}.
    \end{split}
\end{equation}
Combining the two last equations shows that applying the following differential operator to a four-point correlation function with a degenerate field $\ps{V_{-\chi\omega_1}(z)V_{\alpha_0}(0)V_{\alpha_1}(1)V_{\alpha_\infty}(\infty)}$ 
\begin{equation}
    \begin{split}
    &\frac{8z^2(z-1)}{\chi^3}\partial_z^3+\frac{4}{\chi}z(3-5z)\partial^2_z\\
    &+\Big[\frac{2}{\chi}\left(\frac{4z^2-5z+2}{z-1}+2z(\Delta_\infty-\Delta)-2\Delta_0+2\Delta_1\frac z{z-1}\right)+\chi\frac{12z^2-15z+4}{z-1}\Big]\partial_z\\
    &+\frac{2}{\chi}\left((1-2z)\frac{\Delta_\infty-\Delta}{z-1}+(3z-2)\frac{\Delta_0}{z(z-1)}-(3z-1)\frac{\Delta_1}{(z-1)^2}\right)\\
    &+\frac{\chi}{3}\left((7-10z)\frac{\Delta_\infty-\Delta}{z-1}+(9z-6)\frac{\Delta_0}{z(z-1)}+(7-9z)\frac{\Delta_1}{(z-1)^2}\right)\\
    &+w+w(\infty)+\frac{w(0)}{z}-\frac{w(1)}{(z-1)^2}\left(2z-1\right)
    \end{split}
\end{equation}
will yield the quantity $\frac{\bm{\mathcal W_{-1}^{(1)}}}{z-1}\ps{V_{-\chi\omega_1}(z)V_{\alpha_0}(0)V_{\alpha_1}(1)V_{\alpha_\infty}(\infty)}$.

If we further assume that the Vertex Operator $V_{\alpha_1}(1)$ is semi-degenerate, \emph{i.e.} that $\alpha_1=\kappa\omega_2$ for some real $\kappa$, then this last quantity may be expressed as a derivative in $z$ of the correlation function:
\begin{align*}
\bm{\mathcal W_{-1}}(\alpha_1,1)\ps{V_{-\chi\omega_1}(z)V_{\alpha_0}(0)V_{\alpha_1}(1)V_{\alpha_\infty}(\infty)}=(q-2\ps{h_1,\alpha_1})\partial_{z_1}\Big\vert_{z_1=1}\ps{V_{-\chi\omega_1}(z)V_{\alpha_0}(0)V_{\kappa\omega_2}(z_1)V_{\alpha_\infty}(\infty)}&\\
=(q-2\ps{h_1,\alpha_1})\left(z\partial_z+\sum_{k=1}^4\Delta_k-2\Delta_\infty\right)\ps{V_{-\chi\omega_1}(z)V_{\alpha_0}(0)V_{\kappa\omega_2}(1)V_{\alpha_\infty}(\infty)}.&
\end{align*}
The first equality follows from Proposition~\ref{degenerate_1} while the second one is a consequence of the conformal covariance of the correlation functions.
Then some (lengthy) algebraic manipulations show that we can write
\[
    \ps{V_{-\chi\omega_1}(z)V_{\alpha_0}(0)V_{\kappa\omega_2}(1)V_{\alpha_\infty}(\infty)}=\norm{z}^{\chi\ps{h_1,\alpha_0}}\norm{z-1}^{\frac{\chi\kappa}3}\mathcal H(z)
\]
with $\mathcal H$ solution of the hypergeometric differential equation of order three
\begin{equation}\label{eq:hyper}
    \begin{split}
    &\Big[z\left(A_1+z\partial_z\right)\left(A_2+z\partial_z\right)\left(A_3+z\partial_z\right)-\left(B_1-1+z\partial_z\right)\left(B_2-1+z\partial_z\right)z\partial_z\Big]\mathcal H=0.
    \end{split}
\end{equation}
In the above equation we have set 
\begin{equation}
\begin{split}
    A_i&:=\frac{\chi}2\ps{\alpha_0+\kappa\omega_2-\chi\omega_1-Q,h_1}+\frac{\chi}2\ps{\alpha_\infty-Q,h_i}\\
    B_i&:=1+\frac{\chi}2\ps{\alpha_0-Q,h_1-h_{i+1}}.
\end{split}
\end{equation}

The fact that $\mathcal H$ is a (at least distributional) solution of Equation~\eqref{eq:hyper} allows to claim that $\mathcal H$ is actually a real analytic function via a standard elliptic regularity argument. Indeed, we can apply the differential operator $\partial_{\bar{z}}^3$ to Equation~\eqref{eq:hyper}. By doing so we see that $\mathcal{H}$, viewed as a function of two real variables, is a solution of a partial differential equation $\mathrm{P}\mathcal{H}=0$ with analytic coefficients on $\R^2\setminus\{(0,0);(1,0)\}$, and whose term of highest degree is given by $z^3(z-1)\Delta^3$ ($z=x+iy$) where $\Delta$ is the standard Laplace operator $\Delta f(x,y)=\frac{\partial^2 f}{\partial x^2}+\frac{\partial^2 f}{\partial y^2}$. In particular $\mathcal{H}$ is seen to be a solution of $\mathrm{P}\mathcal{H}=0$ where $\mathrm{P}$ is an analytic hypoelliptic operator on $\R^2\setminus\{(0,0);(1,0)\}$. This implies that $\mathcal{H}$ is real analytic on $\C\setminus\{0,1\}$.
\subsection{Explicit expression for the descendent fields: some perspectives}
One of the contributions of the present document is to provide an explicit expression for the descendent fields. In the above subsections we have shed light on some applications of this property to the study of degenerate fields at the first, second and third level, and its implications on a BPZ-type differential equation. Let us now stress some future directions of work allowed by this new framework.
\subsubsection{Liouville theory} Using the same reasoning as above it is possible to write down an explicit expression of the descendent fields in Liouville theory. Namely if $\lambda=(\lambda_1,\cdots,\lambda_r)$ is a finite, non-increasing sequence of positive integers (a \emph{Young diagram}) we can write 
\begin{equation}\label{eq:desc_Liou}
\bm{\mathrm L_{-\lambda_1}}\cdots \bm{\mathrm L_{-\lambda_r}}V_\alpha(z)=\bm{\mathrm L_{-\lambda}}(z,\alpha)V_\alpha(z),
\end{equation}
where $\bm{\mathrm L_{-\lambda}}(z,\alpha)$ admits an expansion of the form
\begin{equation}\label{eq:desc_Liou2}
\bm{\mathrm L_{-\lambda}}(z,\alpha)=\sum_{\norm\mu=\norm\lambda}a_{\mu\vert\lambda}(\alpha)\partial^\mu\varphi.
\end{equation}
Here the sum ranges over Young diagrams with $\norm\mu\coloneqq\sum_{i\geq 1}\mu_i=\norm\lambda$ and we have denoted 
\[
\partial^\mu\varphi\coloneqq \partial^{\mu_1}\varphi\times\cdots\times\partial^{\mu_r}\varphi.
\]
Coefficients $a_{\mu\vert\lambda}(\alpha)$ can be determined via a simple recursive procedure. Determining these coefficients could provide an algorithmic way of deriving BPZ-type equations for general singular vectors $V_{\ps{r,s}}\coloneqq V_{-(r-1)\frac\gamma2-(s-1)\frac2\gamma}$. We stress that this method, new to the best of our knowledge, can be implemented within the probabilistic framework of Liouville theory, in contrast with the usual one based on irreducible representations of Verma modules~\cite{FF}.
\subsubsection{Toda theories}
The reasoning presented in the context of Liouville theory extends to the study of Toda theories. Namely it is possible to write down $W$-descendents in a form similar to that of Equations~\eqref{eq:desc_Liou} and~\eqref{eq:desc_Liou2}. This could be the starting point to the derivation of BPZ-type equations for Toda theories. Again one of the key advantages of the formulation presented is that it could be implemented within our probabilistic framework and provide a general procedure for rigorously proving that BPZ-type equations hold true for Toda theories. In particular this representation should allow us to recover some of the results from~\cite{BW92}, ~\cite{BEFS} and~\cite{RSW18}.

\section{Fusion estimates and Auxiliary computations}\label{section_appendix}
In this last section we provide some auxiliary results to prove the regularity of the correlation functions.

\subsection{Proof of Lemma~\ref{approx_inverse}}
The first one corresponds to the fact that the definition we have provided for regularizing the map $x\mapsto\frac1{x^p}$ gives  the expected result in the $\eps\rightarrow0$ limit. 
\begin{proof}[Proof of Lemma~\ref{approx_inverse}]
First assume that we are given $0<\eps<\frac{\norm{x}}{4R}$. If $\norm{x+\eps(z_1-z_2)}\leq\frac{\norm{x}}2$, then necessarily either $\norm{z_1}\geq R$ or $\norm{z_2}\geq R$, and this implies that $\eta(z_1)\eta(z_2)$ vanishes since $\eta$ is compactly supported in the domain $B(0,R)$. Therefore we can apply integration by parts and the change of variables $z_i\leftrightarrow\eps z_i$ to get
\[
    \frac{1}{(x)_\eps^p}=\int_{\norm{x+\eps(z_1-z_2)}>\frac{\norm{x}}{2}}\frac{1}{(x+\eps(z_1-z_2))^p}\eta(z_1)\eta(z_2)d^2z_1d^2z_2.
\]
As a consequence
\[
    \frac{x^p}{(x)_\eps^p}-1=\int_{\norm{x+\eps(z_1-z_2)}>\frac{\norm{x}}{2}}\frac{x^p-(x+\eps(z_1-z_2))^p}{(x+\eps(z_1-z_2))^p}\eta(z_1)\eta(z_2)d^2z_1d^2z_2.
\]
We also know that, since we have assumed that $0<\eps<\frac{\norm x}{4R}$, we have the bound $\norm{\frac{\eps(z_1-z_2)}{x}}<1$ on the domain where $\eta(z_1)\eta(z_2)$ does not vanish. As a consequence we can expand the integrand as a power series in the variable $\frac{\eps(z_1-z_2)}{x}$.  Since $\eta$ is compactly supported, the integral of this power series is absolutely convergent. Therefore, we conclude by noticing that the first two terms in the expansion vanish in the limit (the first one is identically zero; for the second order term we use the $z_1\leftrightarrow z_2$ symmetry), so that we can factorize by $\left(\frac{\eps}{x}\right)^2$.

Now if we assume that $\norm x\leq 4R\eps$ then
\[
    \frac{x^p}{(x)_\eps^p}=\int_{\norm{x+\eps(z_1-z_2)}>\frac{\norm{x}}{2}}\frac{(\frac x\eps)^p}{(\frac x\eps+(z_1-z_2))^p}\eta(z_1)\eta(z_2)d^2z_1d^2z_2.
\]
Since $\norm x\leq 4R\eps$, the integrand can be bounded by the quantity $3^{-p}$, uniformly on $\eps>0$. This allows to conclude our proof of Lemma~\ref{approx_inverse}.
\end{proof}
\subsection{Technical estimates}
For future convenience we rewrite the correlation functions as
\begin{equation*}
\left\langle\V_\eps\right\rangle_{\delta}=\left (\prod_{i=1}^{2} \frac{\Gamma(s_i)\mu_i^{-s_i}}\gamma\right)\prod_{j<k}\norm{z_j-z_k}_\eps^{-\langle\alpha_j,\alpha_k\rangle }H_{\eps,\delta}(\bm z,\bm\alpha)
\end{equation*}
where $s_i$ is the one in Equation~\eqref{eq:definition_si} and
\begin{equation}\label{eq:definition_H}
H_{\eps,\delta}(\bm z,\bm\alpha)\coloneqq\E\left[\prod_{i=1}^{2}\left(\int_\C\theta_\delta(z_0-y_i)\frac{\hat g_\eps(y_i)^{-\frac{\gamma}{4}\sum_{k=1}^N\ps{\alpha_k,e_i}}}{\prod_{k=1}^N\norm{z_k-y_i}_\eps^{\gamma\ps{\alpha_k,e_i}}}M^{\gamma e_i,\hat g}_\eps(d^2y_i)\right)^{-s_i}\right].
\end{equation}
We also introduce for finite complex vectors $\bm x_1\coloneqq\left(x_1^{(1)},\cdots,x_1^{(r_1)}\right), \bm x_2\coloneqq\left(x_2^{(1)},\cdots,x_2^{(r_2)}\right)$ the notation $H^{(\bm x_1, \bm x_2)}_{\eps,\delta}(\bm z,\bm\alpha)$, which is defined in a way similar to $H_{\eps,\delta}(\bm z,\bm\alpha)$ but for the correlation functions where extra Vertex Operators with directions $\gamma e_1$ or $\gamma e_2$ are inserted:
\begin{equation*}
\ps{\prod_{i=1}^2V_{\gamma e_i,\eps}(\bm x_i)\V_\eps}_\delta\coloneqq \ps{\prod_{k=1}^{r_1}V_{\gamma e_1,\eps}\left(x_1^{(k)}\right)\prod_{j=1}^{r_2}V_{\gamma e_2,\eps}\left(x_2^{(j)}\right)\V_\eps}_\delta.
\end{equation*}
As before, this correlation is the product of a constant, a GFF prefactor and an expectation $H_{\eps,\delta}(\bm z,\bm\alpha)$. The exact expression for $H_{\eps,\delta}(\bm z,\bm\alpha)$ is lengthy to write down: one should take Equation~\eqref{eq:definition_H} and appropriately modify indices.

\subsubsection{Bounds on the Toda correlation functions: proof of Lemma~\ref{integrability}}\label{proof_integrability}
Let us start by observing that with the $\eps$-regularization, no singularities come from the prefactor. Moreover, viewed as a function of $x_i^{(l)}$, the quantity $\norm{x_i^{(l)}-z_k}_\eps^{-\gamma\ps{e_i,\alpha_k}}$ can be bounded by some constant times $\norm{1+x_i^{(l)}}_\eps^{-\gamma\ps{e_i,\alpha_k}}$. As a consequence this prefactor together with the constant part can be bounded by 
\begin{equation}\label{bound_prefactor}
C_\eps\prod_{l=1}^{r_1}\norm{1+|x_1^{(l)}|}_\eps^{-\gamma\ps{\sum\limits_{k}\alpha_k,e_1}}\prod_{m=1}^{r_2}\norm{1+|x_2^{(m)}|}_\eps^{-\gamma\ps{\sum\limits_{k}\alpha_k,e_2}}\prod_{(i,l)\neq(j,m)}\norm{x_i^{(l)}-x_j^{(m)}}_\eps^{-\frac{\gamma^2}2A_{ij}}.
\end{equation}

For item (1), we need to study the expectation part $H^{(\bm x_1,\bm x_2)}_{\eps,\delta}$ for large $\bm x_1, \bm x_2$. Inside the integrals in the expression $H^{(\bm x_1,\bm x_2)}_{\eps,\delta}$, we identify singularities of the form with $i=1,2$:
\begin{equation*}
\frac{1}{\prod_{k=1}^N\norm{z_k-y_i}_\eps^{\gamma\ps{\alpha_k,e_i}}}\frac{1}{\prod_{l=1}^{r_1}\norm{x_1^{(l)}-y_i}_\eps^{\gamma\ps{e_1,e_i}}}\frac{1}{\prod_{m=1}^{r_2}\norm{x_2^{(m)}-y_i}_\eps^{\gamma\ps{e_2,e_i}}}\cdot
\end{equation*}
We can factorize out these deterministic factors on $\bm{x}_1,\bm{x}_2$ ``at infinity'' from the expectation. Then what is left in the expectation is approximately the part corresponding to $H_{\eps,\delta}$. More precisely, write
\begin{equation*}
\frac{1}{\prod_{l=1}^{r_1}\norm{x_1^{(l)}-y_i}_\eps^{\gamma\ps{e_1,e_i}}}=\frac{\prod_{l=1}^{r_1}\norm{1+|x_1^{(l)}|}_\eps^{\gamma\ps{e_1,e_i}}}{\prod_{l=1}^{r_1}\norm{x_1^{(l)}-y_i}_\eps^{\gamma\ps{e_1,e_i}}}\cdot\prod_{l=1}^{r_1}\norm{1+|x_1^{(l)}|}_\eps^{-\gamma\ps{e_1,e_i}}
\end{equation*}
and similarly for $\bm{x}_2$. Remark that the first factor above goes to $1$ as $x^{(l)}\to\infty$. It follows that, for some positive constant $C$,
\begin{equation*}
\begin{split}
    H^{(\bm x_1,\bm x_2)}_{\eps,\delta}&\leq C\prod_{i=1}^{2}\left(\prod_{l=1}^{r_1}\norm{1+|x_1^{(l)}|}_\eps^{-\gamma\ps{e_1,e_i}}\prod_{m=1}^{r_2}\norm{1+|x_2^{(m)}|}_\eps^{-\gamma\ps{e_2,e_i}}\right)^{-s_i}\\
    &\leq C\prod_{l=1}^{r_1}\norm{1+|x_1^{(l)}|}_\eps^{\gamma\ps{\sum\limits_{k}\alpha_k-2Q+\gamma r_1e_1+\gamma r_2e_2,e_1}}\prod_{m=1}^{r_2}\norm{1+|x_2^{(m)}|}_\eps^{\gamma\ps{\sum\limits_{k}\alpha_k-2Q+\gamma r_1e_1+\gamma r_2e_2,e_2}}.
\end{split}
\end{equation*}
Combining with the bound~\eqref{bound_prefactor} for the prefactor yields:
\begin{equation*}
\begin{split}
&\left\langle\prod_{i=1}^{r_1}V_{\gamma e_1,\eps}\left(x_1^{(i)}\right)\prod_{j=1}^{r_2}V_{\gamma e_2,\eps}\left(x_2^{(j)}\right)\V_\eps \right\rangle_{\delta}\\
&\leq C_\eps\prod_{l=1}^{r_1}\norm{1+|x_1^{(l)}|}_\eps^{\gamma\ps{-2Q+\gamma r_1e_1+\gamma r_2e_2,e_1}}\prod_{m=1}^{r_2}\norm{1+|x_2^{(m)}|}_\eps^{\gamma\ps{-2Q+\gamma r_1e_1+\gamma r_2e_2,e_2}}\prod_{(i,l)\neq(j,m)}\norm{x_i^{(l)}-x_j^{(m)}}_\eps^{-\frac{\gamma^2}2A_{ij}}\\
&= C_\eps\prod_{l=1}^{r_1}\norm{1+|x_1^{(l)}|}_\eps^{-4}\prod_{m=1}^{r_2}\norm{1+|x_2^{(m)}|}_\eps^{-4}\prod_{(i,l)\neq(j,m)}\left(\frac{\norm{1+x_i^{(l)}}\norm{1+x_j^{(m)}}}{\norm{x_i^{(l)}-x_j^{(m)}}_\eps}\right)^{\frac{\gamma^2}2A_{ij}}.
\end{split}
\end{equation*}
The last term on the right-hand side being bounded this finishes the proof of item (1).

For item (2), the same reasoning remains valid and therefore the same bounds are still true when looking at the behaviour near infinity of the correlation functions when we take the $\eps,\delta\to 0$ limit of the correlation functions. 

For item (3), we investigate the behaviour when $\eps$ goes to zero of the fusion of two insertions, that is when $\norm{z_1-z_2}\rightarrow0$ but with all other insertions staying at fixed positive distance at least $\rho$. Suppose that for both $k=1,2$, $\ps{\alpha_1+\alpha_2-Q,e_k}<0$. In this case, $H_{\eps,\delta}$ remains bounded when $z_1$ and $z_2$ merge (since in that case the expectation remains well-defined according to \cite[Lemma 4.1]{Toda_construction}) and therefore the behaviour of the correlation functions is governed by the prefactor, of which the singularity is of order $\norm{z_1-z_2}^{-\ps{\alpha_1,\alpha_2}}$.

The analysis is a bit more subtle when $\ps{\alpha_1+\alpha_2,e_2}\geq \gamma+\frac{2}{\gamma}$ but $\ps{\alpha_1+\alpha_2,e_1}<\gamma+\frac2\gamma$. However we can adapt the result in the Liouville case by a simple argument. Indeed, by H\"older's inequality we have that for positive $p_1,p_2$ with $\frac{1}{p_1}+\frac{1}{p_2}=1$,
\begin{equation*}
    H_{\eps,\delta}\leq\prod_{i=1}^{2}\E\left[\left(\int_\C\theta_\delta(z_0-y_i)\frac{\hat g_\eps(y_i)^{-\frac{\gamma}{4}\sum_{k=1}^N\ps{\alpha_k,e_i}}}{\prod_{k=1}^N\norm{z_k-y_i}_\eps^{\gamma\ps{\alpha_k,e_i}}}M^{\gamma e_i,\hat g}_\eps(d^2y_i)\right)^{-p_is_i}\right]^{\frac{1}{p_i}}.
\end{equation*}
We can further ignore the $\theta_\delta$ term since $z_0$ is away from the singularities in the integrand. Since $\ps{\alpha_1+\alpha_2,e_1}<\gamma+\frac{2}{\gamma}$ and $s_1>0$, the expectation corresponding to the index $i=1$ remains bounded for any $p_1>1$ (see again \cite[Lemma 4.1]{Toda_construction}). The behaviour of the expectation with index $i=2$ when $z_1$ and $z_2$ merge is reduced to the Liouville case (the so-called freezing estimate, see~\cite[Lemma 6.5]{KRV_loc}): for any $\eta>0$,
\begin{equation*}
\E\left[\left(\int_\C\frac{\hat g_\eps(y_2)^{-\frac{\gamma}{4}\sum_{k=1}^N\ps{\alpha_k,e_2}}}{\prod_{k=1}^N\norm{z_k-y_2}_\eps^{\gamma\ps{\alpha_k,e_2}}}M^{\gamma e_2,\hat g}_\eps(d^2y_2)\right)^{-p_2s_2}\right]\leq C_\eps\norm{z_1-z_2}^{\frac{\ps{\alpha_1+\alpha_2-Q,e_2}^2}{2}-\eta}.
\end{equation*}
Choosing $p_2$ close enough to $1$, we see that the two-point fusion estimate is governed by
\begin{equation}
    \norm{z_1-z_2}^{-\ps{\alpha_1,\alpha_2}}\norm{z_1-z_2}^{-\eta'+\frac{1}{2}\ps{\alpha_1+\alpha_2-Q,e_2}^2}.
\end{equation}

Item (4) is a direct consequence of the previous bound.

For item (5): the bound of item (3) ensures local integrability near $z_1$ of $\ps{V_{\gamma e_1}(x)\V}$ in $L^p(\C)$ for $1\leq p<\frac{4}{4-\gamma^2}$ if we have $\frac{2}{\gamma}-\gamma\leq\ps{\alpha_0,e_1}<\frac{2}{\gamma}+\gamma$, for $1\leq p<\frac{2}{\gamma\ps{\alpha_0,e_1}}$ if we assume that $0<\ps{\alpha_0,e_1}<\frac{2}{\gamma}-\gamma$, and in $L^{\infty}(\C)$ if $\ps{\alpha_0,e_1}\leq0$.
\qed

\subsubsection{A generalized fusion estimate}
We now provide a generalized fusion estimate for which we consider the situation where finitely many points merge within a correlation functions. In our applications, we only need the special case where the weights of the merging Vertex Operators are of the form $\gamma e_i$, for $i=1,2$, and where points merge pairwise. This will be the setup that we investigate in this section.

We assume that $\bm z$ are distinct and that  $\bm\alpha$ are such that the Seiberg bounds~\eqref{seiberg_bounds} hold. Without loss of generality we assume that $\rho\coloneqq \min\limits_{1\leq k\leq N}\norm{z_k}$ is positive. We are interested in estimating the following correlation functions with $2p$ extra points:
\begin{equation*}
    \ps{\prod_{l=1}^{p}V_{\gamma e_{\sigma(l)}}(x_l)V_{\gamma e_{\tau(l)}}(y_l)\V}
\end{equation*}
where $\sigma(l),\tau(l)\in\{1,2\}$ so the extra weights are either $\gamma e_1$ or $\gamma e_2$.

To get a reasonable estimate, we need to first reorder the points $(\bm{x},\bm{y})$ in such a way that the collision can only happen pairwise, i.e. between each pair $x_i$ and $y_i$. In short, we restrict the locations on the points: we assume that $x_l$, $y_l$ belong to sets $A_l$, $B_l$ in such a way that any two of these sets are disjoint, except for $A_l$ and $B_l$ with the same indices.

More precisely, we introduce domains $\bm A\coloneqq A_1\times\cdots\times A_p$ and $\bm B\coloneqq B_1\times\cdots\times B_p$ such that:
\begin{itemize}
    \item For any $1\leq l\leq p$, $A_l$ and $B_l$ are either annuli or balls, and are contained in $B(0,\frac\rho2)$. This is to avoid collisions between $(\bm{x},\bm{y})$ and the given $\bm{z}$;
    \item There exists some positive distance $d>0$, for which the $(A_l)_{1\leq l\leq p}$ are at distance at least $d$ one to the other (and similarly for the $(B_l)_{1\leq l\leq p}$);
    \item For any $1\leq l,m\leq p$, the distance between $A_l$ and $B_m$ is zero if and only if $l=m$.
\end{itemize}

\begin{lemma}\label{fusion}
There exist two positive constants $C$ and $\zeta$ such that, for any $(\bm x,\bm y)\in \bm A\times\bm B$:
\begin{equation}
    \ps{\prod_{l=1}^p V_{\gamma e_{\sigma(l)},\eps}(x_l)V_{\gamma e_{\tau(l)},\eps}(y_l)\V_\eps}_\delta\leq C\prod_{l=1}^p\norm{x_l- y_l}^{-2+\zeta}
\end{equation}
uniformly in $\eps$ and $\delta$. In particular the integral
\begin{equation}
    \int_{\bm A}\int_{\bm B}\frac{1}{\prod_{l=1}^p(x_l-y_l)}\ps{\prod_{i=1}^2V_{\gamma e_i}(\bm x_i)V_{\gamma e_i}(\bm y_i)\V} d^2\bm x d^2\bm y
\end{equation}
is absolutely convergent.
\end{lemma}




\begin{proof}[Proof of Lemma~\ref{fusion}]
The proof of this claim parallels the one in the Liouville case~\cite[Lemma 3.1]{Oi19}. The study is slightly more involved in the present case because of the presence of an additional GFF in the construction of the correlation functions, but the basic idea remains the same.

We separate as before between the prefactor and the expectation term, see Equation~\eqref{eq:definition_H}. Since the domains are disjoint except for $A_l$ and $B_m$ with $l=m$, the prefactor part is bounded by
\begin{equation*}
C\prod_{l=1}^p\norm{x_l-y_l}^{-\gamma^2\ps{e_{\sigma(l)},e_{\tau(l)}}}
\end{equation*}
where $C$ is some positive constant.

For the expectation term the analysis is slightly more subtle, but one can simplify the problem by noticing that when $\sigma(l)\neq\tau(l)$, there is no singularity in the integral (since the sign in the power is the opposite one). As a consequence, and without loss of generality, we may assume $\sigma(l)=\tau(l)=1$ for $l\in E_1\coloneqq\{1,\dots,r\}$ and $\sigma(l)=\tau(l)=2$ for $l\in E_2\coloneqq\{r+1,\dots,p\}$. Therefore the integrand can be bounded by some constant times
\begin{equation*}
\E\left[\prod_{i=1,2}\left(\int_{\C}\frac{F_i(w_i)}{\prod_{l\in E_i} \norm{x_l-w_i}_\eps^{2\gamma^2}\norm{y_l-w_i}_\eps^{2\gamma^2}}M^{\gamma e_i,\hat g}_\eps(d^2w_i)\right)^{-s_i}\right]
\end{equation*}
where $F_1$ and $F_2$ are smooth in a neighbourhood of the $x_l,y_l$. We are now to distinguish between two cases:
\begin{itemize}
    \item If $\gamma<\sqrt{\frac{2}{3}}$, then when $x_l$ and $y_l$ merge the singularity remains integrable since in that case $4\gamma^2<\gamma \ps{Q,e_i}$ so the expectation still makes sense according to \cite[Lemma 4.1]{Toda_construction}. Therefore if $\gamma<\sqrt{\frac{2}{3}}$, the fusion estimate is governed by the prefactor, which scales as $|x_l-y_l|^{-2\gamma^2}$ with $-2\gamma^2>-\frac{4}{3}>-2$.
    \item If $\sqrt{\frac{2}{3}}\leq\gamma<\sqrt2$, we claim that, for any positive $\zeta$, it is bounded by
    \begin{equation*}
    \prod_{l=1}^p\norm{x_l-y_l}^{\frac{(3\gamma-\frac{2}{\gamma})^2}{2}-\zeta}.
    \end{equation*}
    For $\sqrt{\frac{2}{3}}\leq\gamma<\sqrt2$, one checks that $-2\gamma^2+\frac{(3\gamma-\frac{2}{\gamma})^2}{2}>-2$.
\end{itemize}

We now prove the last claim following~\cite[Lemma 3.1]{Oi19}. Using that the exponents $s_i$ are positive, we see that we can bound the expectation by (some constant times)
\begin{equation*}
\E\left[\left(\sum_{l=1}^r\int_{B(x_l,\frac{d}{4})} \frac{M^{\gamma e_1,\hat g}_\eps(d^2w_1)}{ \norm{x_l-w_1}_\eps^{2\gamma^2}\norm{y_l-w_1}_\eps^{2\gamma^2}}\right)^{-s_1}\left(\sum_{l=r+1}^p\int_{B(x_l,\frac{d}{4})}\frac{M^{\gamma e_2,\hat g}_\eps(d^2w_2)}{ \norm{x_l-w_2}_\eps^{2\gamma^2}\norm{y_l-w_2}_\eps^{2\gamma^2}}\right)^{-s_2}\right].
\end{equation*}
We can suppose that the GFFs in different balls $B(x_l,\frac{d}{4})$ are independent: this is a classical manipulation of Kahane's convexity inequality, since the covariance of different GFFs in different balls is uniformly bounded from below and above by a global constant (see the proof of \cite[Lemma 4.1]{Toda_construction}). Furthermore, we use the elementary inequality, for positive $a$'s and $s>0$,
\begin{equation*}
\left(\sum\limits_{l=1}^{r}a_l\right)^{-s}\leq \prod\limits_{l=1}^{r}a_l^{-s/r}
\end{equation*}
to reduce the problem to the following estimate: for $s>0$ and any $\zeta>0$ there exists a positive constant $C$ such that
\begin{equation*}
\E\left[\left(\int_{B(x,\frac{d}{4})}\frac{M^{\gamma e_1,\hat g}_\eps(d^2w)}{ \norm{x-w}_\eps^{2\gamma^2}\norm{y-w}_\eps^{2\gamma^2}}\right)^{-s}\right]\leq C\norm{x-y}^{\frac{(3\gamma-\frac{2}{\gamma})^2}{2}-\zeta}.
\end{equation*}
Since for $\gamma^2\geq\frac{2}{3}$, $4\gamma\geq q=\gamma+\frac{2}{\gamma}$, the classical freezing estimate~\cite[Lemma 6.5]{KRV_loc} yields the bound $|x-y|^{\frac{(4\gamma-q)^2}{2}-\zeta}$ for any $\zeta>0$, which is the above.

For the second item of Lemma~\ref{fusion} we use the first estimate to see that all one needs to prove is that for positive $\zeta$ the quantity
\begin{equation*}
\int_{\bm A}\int_{\bm B}\frac{1}{\prod_{l=1}^p\norm{x_l-y_l}^{3-\zeta}}d^2\bm x d^2\bm y
\end{equation*}
is finite. Since the domains are disjoint we see that it is enough to check that the two-fold integral
\begin{equation*}
\int_{B(0,\frac{\rho}{6})}\int_{A(0,\frac\rho{6},\frac{\rho}{3})}\frac{1}{\norm{x-y}^{3-\zeta}}d^2x d^2y
\end{equation*}
converges. This can be easily seen by using polar coordinates.
\end{proof}

\subsection{Second order differentiability for the correlation functions (Proposition~\ref{regularity})}\label{proof_C2}

Recall the expression for the first derivative of the regularized correlation functions (we keep the same notations and performed an integration by parts with Stokes' formula):
\begin{equation*}
\begin{split}
    &\partial_{z_0}\ps{V_{\alpha_0,\eps}(z_0)\bm{\mathrm V}_\eps}_\delta={}-\frac{1}{2}\sum_{k=1}^N\frac{\ps{\alpha_0,\alpha_k}}{(z_0-z_k)_\eps}\ps{V_{\alpha_0,\eps}(z_0)\bm{\mathrm V}_\eps}_\delta\\
    &+\mu\sum_{i=1}^{2}\int_{B(z_0,r)^c}\frac{\gamma\ps{\alpha_0,e_i}}2\frac{\theta_\delta(z_0-x) }{(z_0-x)_\eps}\ps{V_{\gamma e_i,\eps}(x)V_{\alpha_0,\eps}(z_0)\bm{\mathrm V}_\eps}_\delta d^2x\\
    &+\mu\sum_{i=1}^2\oint_{\partial B(z_0,r)}\frac{\theta_\delta(z_0-\xi) }{(z_0-\xi)_\eps}\ps{V_{\gamma e_i,\eps}(\xi)V_{\alpha_0,\eps}(z_0)\bm{\mathrm V}_\eps}_\delta \frac{\sqrt{-1}d\bar\xi}2\\
    &-\mu\sum_{i=1}^2\int_{A}\sum_{k=1}^N\frac{\gamma\ps{\alpha_k,e_i}}2\frac{\theta_\delta(z_0-x)}{(z_k-x)_\eps}\ps{V_{\gamma e_i,\eps}(x)V_{\alpha_0,\eps}(z_0)\bm{\mathrm V}_\eps}_\delta d^2x\\
    &-\sum_{i,j=1}^2\frac{(\mu\gamma)^2\ps{e_i,e_j}}2\int_{A}\int_{B(z_0,r)^c} \frac{\theta_\delta(z_0-x_1)\theta_\delta(z_0-x_2)}{(x_1-x_2)_\eps}\ps{V_{\gamma e_i,\eps}(x_1)V_{\gamma e_j,\eps}(x_2)V_{\alpha_0,\eps}(z_0)\bm{\mathrm V}_\eps}_\delta d^2x_1d^2x_2.
\end{split}
\end{equation*}
where $A=A(z_0,\delta/2,r)$ is the annulus of radii $\delta/2$ and $r$ centered at $z_0$. We have already seen that the correlation functions are $C^1$ with respect to $z_0$. It follows that the first three terms above are $C^1$ in $z_0$, since the integrand remains smooth as $\eps,\delta\to 0$ (we stay away from the singular point $z_0$).

To prove that the correlation functions are $C^2$, we differentiate the two last terms and prove that the $\eps,\delta\to 0$ limit is finite. Differentiating the penultimate line with respect to $z_0$ yields 
\begin{equation*}
\begin{split}
&+\mu\sum_{i=1}^2\int_{A}\sum_{k=1}^N\frac{\ps{\alpha_k,\gamma e_i}}{2(z_k-x)_\eps}\left(\sum_{l=1}^N\frac{\theta_\delta(z_0-x)\ps{\alpha_0,\alpha_l}}{2(z_0-z_k)_\eps}+\frac{\ps{\alpha_0,\gamma e_i}\theta_\delta(z_0-x)}{2(z_0-x)_\eps}\right)\ps{V_{\gamma e_i,\eps}(x)V_{\alpha_0,\eps}(z_0)\bm{\mathrm V}_\eps}_\delta d^2x\\
&-\mu\sum_{i=1}^2\int_{A}\sum_{k=1}^N\frac{\ps{\alpha_k,\gamma e_i}}{2(z_k-x)_\eps}\partial_{z_0}\theta_\delta(z_0-x)\ps{V_{\gamma e_i,\eps}(x)V_{\alpha_0,\eps}(z_0)\bm{\mathrm V}_\eps}_\delta d^2x\\
&-\mu^2\sum_{i,j=1}^2\int_{A}\int_\C\sum_{k=1}^N\frac{\ps{\alpha_k,\gamma e_i}\theta_\delta(z_0-x_1)}{2(z_k-x_1)_\eps}\frac{\ps{\alpha_0,\gamma e_j}\theta_\delta(z_0-x_2)}{2(z_0-x_2)_\eps}\ps{V_{\gamma e_i,\eps}(x_1)V_{\gamma e_j,\eps}(x_2)V_{\alpha_0,\eps}(z_0)\bm{\mathrm V}_\eps}_\delta d^2x_1d^2x_2\\
&+\mu^2\sum_{i,j=1}^2\int_{A}\int_\C\sum_{k=1}^N\frac{\ps{\alpha_k,\gamma e_i}\theta_\delta(z_0-x_1)}{2(z_k-x_1)_\eps}\partial_{z_0}\theta_\delta(z_0-x_2)\ps{V_{\gamma e_i,\eps}(x_1)V_{\gamma e_j,\eps}(x_2)V_{\alpha_0,\eps}(z_0)\bm{\mathrm V}_\eps}_\delta d^2x_1d^2x_2.
\end{split}
\end{equation*}

For the terms involving derivatives of $\theta_\delta$, we again integrate them by parts:
\begin{align*}
\mu\sum_{i=1}^2\oint_{\partial B(z_0,r)}\sum_{k=1}^N\frac{\ps{\alpha_k,\gamma e_i}}{2(z_k-\xi)_\eps}\theta_\delta(z_0-\xi)\ps{V_{\gamma e_i,\eps}(\xi)V_{\alpha_0,\eps}(z_0)\bm{\mathrm V}_\eps}_\delta  \frac{\sqrt{-1}d\bar\xi}2+\hspace{5.5cm}&\\
\hspace{-2cm}\mu\sum_{i=1}^2\int_{A}\sum_{k=1}^N\frac{\ps{\alpha_k,\gamma e_i}}{2(z_k-x)_\eps}\left(\sum_{l=1}^N\frac{\theta_\delta(z_0-x)\ps{\alpha_0,\alpha_l}}{2(z_0-z_k)_\eps}+\frac{\ps{\alpha_l,\gamma e_i}\theta_\delta(z_0-x)}{2(x-z_l)_\eps}+\frac{\theta_\delta(z_0-x)}{2(x-z_k)_\eps}\right)\ps{V_{\gamma e_i,\eps}(x)V_{\alpha_0,\eps}(z_0)\bm{\mathrm V}_\eps}_\delta &d^2x\\
-\mu^2\sum_{i,j=1}^2\int_{A}\int_{B(z_0,r)^c}\sum_{k=1}^N\frac{\ps{\alpha_k,\gamma e_i}\theta_\delta(z_0-x_1)}{2(z_k-x_1)_\eps}\times\hspace{9cm}&\\
\left(\frac{\ps{\alpha_0,\gamma e_j}\theta_\delta(z_0-x_2)}{2(z_0-x_2)_\eps}+\frac{\ps{\gamma e_i,\gamma e_j}\theta_\delta(z_0-x_2)}{2(x_1-x_2)_\eps}-\partial_{z_0}\theta_\delta(z_0-x_2)\right)\ps{V_{\gamma e_i,\eps}(x_1)V_{\gamma e_j,\eps}(x_2)V_{\alpha_0,\eps}(z_0)\bm{\mathrm V}_\eps}_\delta d^2x_1&d^2x_2\\
-\mu^2\sum_{i,j=1}^2\int_{A^2}\sum_{k=1}^N\frac{\ps{\alpha_k,\gamma e_i}\theta_\delta(z_0-x_1)}{2(z_k-x_1)_\eps}\times\hspace{10cm}&\\
\left(\frac{\ps{\alpha_0,\gamma e_j}\theta_\delta(z_0-x_2)}{2(z_0-x_2)_\eps}+\frac{\ps{\gamma e_i,\gamma e_j}\theta_\delta(z_0-x_2)}{2(x_1-x_2)_\eps}-\partial_{z_0}\theta_\delta(z_0-x_2)\right)\ps{V_{\gamma e_i,\eps}(x_1)V_{\gamma e_j,\eps}(x_2)V_{\alpha_0,\eps}(z_0)\bm{\mathrm V}_\eps}_\delta d^2x_1&d^2x_2.
\end{align*}
All but the two last lines remain bounded as $\eps,\delta\rightarrow0$ thanks to Lemmas~\ref{integrability} and~\ref{fusion}. 

Using again integration by parts for the last line we get regular terms plus:
\begin{align*}
-\mu^2\sum_{i,j=1}^2\int_{A^2}\sum_{k,l=1}^N\frac{\ps{\alpha_k,\gamma e_i}\theta_\delta(z_0-x_1)}{2(z_k-x_1)_\eps}&\frac{\ps{\alpha_k,\gamma e_j}\theta_\delta(z_0-x_2)}{2(z_k-x_2)_\eps}\ps{V_{\gamma e_i,\eps}(x_1)V_{\gamma e_j,\eps}(x_2)V_{\alpha_0,\eps}(z_0)\bm{\mathrm V}_\eps}_\delta d^2x_1d^2x_2\\
+\mu^3\sum_{i,j,l=1}^2\int_{A^2}\int_\C\sum_{k=1}^N\frac{\ps{\alpha_k,\gamma e_i}\theta_\delta(z_0-x_1)}{2(z_k-x_1)_\eps}&\frac{\ps{\gamma e_j,\gamma e_l}\theta_\delta(z_0-x_2)}{2(x_3-x_2)_\eps}\\
&\ps{V_{\gamma e_i,\eps}(x_1)V_{\gamma e_j,\eps}(x_2)V_{\gamma e_l,\eps}(x_3)V_{\alpha_0,\eps}(z_0)\bm{\mathrm V}_\eps}_\delta d^2x_1d^2x_2d^2x_3.
\end{align*}
Using the symmetries in the $x_2,x_3$ variables, the last integral vanishes on $A^3$. This means that this term remains bounded when $\eps,\delta\rightarrow0$ by using Lemma~\ref{fusion}.

To finish up with the proof that the correlation functions are $C^2$ it remains to take care of the derivative of the term
\begin{equation*}
-\sum_{i,j=1}^2\frac{(\mu\gamma)^2\ps{e_i,e_j}}2\int_{A}\int_{B(z_0,r)^c} \frac{\theta_\delta(z_0-x_1)\theta_\delta(z_0-x_2)}{(x_1-x_2)_\eps}\ps{V_{\gamma e_i,\eps}(x_1)V_{\gamma e_j,\eps}(x_2)V_{\alpha_0,\eps}(z_0)\bm{\mathrm V}_\eps}_\delta d^2x_1d^2x_2.
\end{equation*}
We proceed in the same way by using integration by parts to get rid of the terms involving derivatives of $\theta_\delta$. We may follow the same lines as in the previous computation, apart from the fact that the new integration domains we will consider either contain or don't the previous integration variables. We won't write down the details here since they are quite lengthy and not informative, nonetheless we see that in the end it is enough to show that the integrals (here $\rho$ is some positive number)
\begin{equation*}
\int_{A^2}\int_{B(z_0,r)^c} \frac{\mathds{1}_{\norm{x_1-x_2}>\rho}}{(x_1-x_3)(x_2-x_3)}\ps{\prod\limits_{k=1}^{3}V_{\gamma e_{\star},\eps}(x_k)V_{\alpha_0,\eps}(z_0)\bm{\mathrm V}_\eps}_\delta d^2x_1d^2x_2d^2x_3
\end{equation*}
and
\begin{equation*}
\int_{A^2}\int_{(B(z_0,r)^c)^2} \frac{\mathds{1}_{\norm{x_1-x_2}>\rho}\mathds{1}_{\norm{x_3-x_4}>\rho}}{(x_1-x_3)(x_2-x_4)}\ps{\prod\limits_{k=1}^{4}V_{\gamma e_{\star},\eps}(x_k)V_{\alpha_0,\eps}(z_0)\bm{\mathrm V}_\eps}_\delta d^2x_1d^2x_2d^2x_3d^2x_4
\end{equation*}
are absolutely convergent, which follows from Lemma~\ref{fusion}. Therefore the $\frac{\partial^2}{\partial z_0^2}$ derivative of the correlation functions are well-defined. Treating the mixed derivatives $\frac{\partial^2}{\partial z_0\partial \bar z_0}$ leads to the same conclusion. Hence existence of both the $\frac{\partial^2}{\partial z_0^2}$ and $\frac{\partial^2}{\partial z_0\partial \bar z_0}$ derivatives: this implies that the correlation functions are $C^2$.

\begin{remark}
The same reasoning shows that the correlation functions are smooth, in conjunction with Lemma~\ref{fusion}. This is in agreement with the Liouville case as shown in~\cite{Oi19}.
\end{remark}
\subsection*{Data Availability Statement}
Data sharing is not applicable to this article as no datasets were generated or analysed.

\bibliographystyle{plain}
\bibliography{biblio}

\end{document}